\documentclass[12pt]{amsart}
\usepackage{mathrsfs}
\usepackage{bbm}%
\usepackage{bbding}
\usepackage[mathcal]{euscript}
\usepackage{graphicx}
\usepackage{amsthm,amscd}
\usepackage{calc}
\usepackage{mathrsfs,dsfont}
\usepackage{CJK,fancyhdr,amscd}
\usepackage{anysize}
\usepackage{amsmath,amssymb,amsfonts}
\usepackage{epsfig,enumerate}
\usepackage{latexsym}
\usepackage{indentfirst, latexsym}
\usepackage{graphics}
\usepackage[all,poly,knot]{xy}
\usepackage[pagebackref]{hyperref}

\providecommand{\U}[1]{\protect\rule{.1in}{.1in}}

\numberwithin{equation}{section}

\newtheorem{theorem} {Theorem} [section]
\newtheorem{proposition}[theorem]{Proposition}
\newtheorem{corollary}  [theorem]     {Corollary}
\newtheorem{lemma}  [theorem]     {Lemma}
\newtheorem{example}  [theorem]     {Example}

\newtheorem{remark}  [theorem]     {Remark}
\newtheorem{definition}  [theorem]     {Definition}

\newtheorem{conjecture}  [theorem]     {Conjecture}
\newtheorem{observation}  [theorem]     {Observation}

\newcommand{\dz}{d \bar{z}}
\newcommand{\pz}{\partial\bar{z}}
\newcommand{\w}{\wedge}

\renewcommand{\1}{\mathds{1}}
\newcommand{\G}{\mathbb{G}}
\newcommand{\db}{\overline{\partial}}

\newcommand{\lc}{\lrcorner}

\newcommand{\lk}{\left(}
\newcommand{\rk}{\right)}

\newcommand{\im}{\mathrm{im}}
\newcommand{\Om}{\Omega}
\newcommand{\btheorem}{\begin{theorem}}
\newcommand{\etheorem}{\end{theorem}}
\newcommand{\bproposition}{\begin{proposition}}
\newcommand{\eproposition}{\end{proposition}}
\newcommand{\bdefinition}{\begin{definition}}
\newcommand{\edefinition}{\end{definition}}
\newcommand{\bcorollary}{\begin{corollary}}
\newcommand{\ecorollary}{\end{corollary}}
\newcommand{\bproof}{\begin{proof}}
\newcommand{\eproof}{\end{proof}}
\newcommand{\beq}{\begin{equation}}
\newcommand{\eeq}{\end{equation}}
\newcommand{\ee}{\end{eqnarray*}}
\newcommand{\be}{\begin{eqnarray*}}

\newcommand{\elemma}{\end{lemma}}
\newcommand{\blemma}{\begin{lemma}}

\newcommand{\om}{\omega}
\renewcommand{\>}{\rightarrow}

\newcommand{\p}{\partial}

\newcommand{\bd}{\begin{enumerate} }
\newcommand{\ed}{\end{enumerate}}

\def\p{\partial}
\def\o{\overline}
\def\b{\bar}
\def\mb{\mathbb}
\def\mc{\mathcal}

\def\w{\wedge}
\def\om{\omega}

\def\l{\lrcorner}

\begin{document}

\title{Power series proofs for local stabilities of K\"{a}hler and balanced structures with mild $\p\db$-lemma}

\begin{abstract}
By use of a natural map introduced recently by the first and third
authors from the space of pure-type complex differential forms on a
complex manifold to the corresponding one on the small differentiable
deformation of this manifold, we will give a power series proof for
Kodaira--Spencer's local stability theorem of K\"{a}hler structures.
We also obtain two new local stability theorems, one of balanced
structures on an $n$-dimensional balanced manifold with the
$(n-1,n)$-th mild $\p\db$-lemma by power series method and the other one on $p$-K\"{a}hler structures with the deformation invariance of $(p,p)$-Bott--Chern numbers.
\end{abstract}

\author{Sheng Rao}
\address{School of Mathematics and statistics, Wuhan  University,
Wuhan 430072, China; Department of Mathematics, University of
California at Los Angeles, CA 90095-1555, USA}
\email{likeanyone@whu.edu.cn}
\thanks{Rao is partially supported by the National Natural Science Foundations of China No. 11671305, 11771339, 11922115
and China Scholarship Council/University of California, Los Angeles Joint Scholarship
Program. Wan is partially supported by Scientific Research Foundation of Chongqing University of Technology.
Zhao is partially supported by the National Natural Science Foundations of China No. 11801205.}

\author{Xueyuan Wan}
  \address{Mathematical Science Research Center, Chongqing University of Technology, Chongqing 400054, China}
  \email{xwan@cqut.edu.cn}

\author{Quanting Zhao}
\address{School of Mathematics and Statistics \&
Hubei Key Laboratory of Mathematical Sciences, Central China Normal
University, Wuhan, 430079, P.R.China.} \email{zhaoquanting@126.com;
zhaoquanting@mail.ccnu.edu.cn}

\date{\today}

\subjclass[2010]{Primary 32G05; Secondary 13D10, 14D15, 53C55}
\keywords{Deformations of complex structures; Deformations and
infinitesimal methods, Formal methods; deformations, Hermitian and
K\"ahlerian manifolds}

\maketitle

\tableofcontents

\section{Introduction}
The local stability of some special complex structure is an
interesting topic in deformation theory of complex structures and
the power series method, initiated by Kodaira--Nirenberg--Spencer and Kuranishi,
plays a prominent role there. One main goal
of this paper is to present a power series proof for the classical
Kodaira--Spencer's local stability of K\"ahler structures, which is a
problem at latest dated back to \cite[Remark 1 on Page 180]{MK}: \lq\lq A
good problem would be to find an elementary proof (for example,
using power series methods). Our proof uses nontrivial results from
partial differential equations."
\begin{theorem}[{\cite[Theorem 15]{KS}}]\label{stab-Kahler}
Let $\pi: \mathcal{X} \rightarrow B$ be a differentiable family of compact complex manifolds. If a fiber $X_0:= \pi^{-1}(t_0)$ admits a K\"ahler metric, then, for a sufficiently small neighborhood $U$ of $t_0$ on $B$, the fiber $X_t:=\pi^{-1}(t)$ over any point $t\in U$ still admits a K\"ahler metric, which depends smoothly on $t$ and coincides for $t=t_0$ with the given K\"ahler metric on $X_0$.
\end{theorem}

The other goal is to prove a new local stability theorem of balanced
structures when the reference fiber satisfies the $(n-1,n)$-th mild
$\p\db$-lemma, a new-type $\p\db$-lemma, using the power series
method developed above, and also one of $p$-K\"{a}hler structures with the deformation invariance of $(p,p)$-Bott--Chern numbers
by two different proofs. Recall that
a \emph{balanced metric} $\omega$ on an $n$-dimensional complex
manifold is a real positive $(1,1)$-form satisfying
$d(\omega^{n-1})=0,$
and a complex manifold is called \emph{balanced} if there
exists such a metric on it.

This paper is a sequel to \cite{lry,RZ15}, whose notions are
adopted here. All manifolds in this paper are assumed to be compact
complex $n$-dimensional manifolds. The symbol $A^{p,q}(X,E)$ stands for the space of
the holomorphic vector bundle $E$-valued $(p,q)$-forms on a
complex manifold $X$. A \emph{Beltrami differential} on $X$, generally
denoted by $\phi$, is an element in
$A^{0,1}(X, T^{1,0}_X)$, where $T^{1,0}_X$ denotes the holomorphic
tangent bundle of $X$. Then $\iota_\phi$ or $\phi\lrcorner$ denotes
the contraction operator with respect to $\phi\in A^{0,1}(X,T^{1,0}_X)$ or
other analogous vector-valued complex differential forms alternatively if there
is no confusion. We also follow the convention
\begin{equation}\label{0e-convention}
e^{\spadesuit}=\sum_{k=0}^\infty \frac{1}{k!} \spadesuit^{k},
\end{equation}
where $\spadesuit^{k}$ denotes $k$-time action of the operator
$\spadesuit$. Since the dimension of $X$ is finite, the summation in
the above formulation is always finite.

We will always consider the differentiable family $\pi: \mathcal{X} \rightarrow
B$ of compact complex $n$-dimensional manifolds
over a sufficiently small domain in $\mathbb{R}^k$ with the
reference fiber $X_0:= \pi^{-1}(t_0)$ and the general fibers $X_t:=
\pi^{-1}(t).$ For simplicity we set $k=1$. Denote by
$\zeta:=(\zeta^\alpha_j(z,t))$ the holomorphic coordinates of $X_t$
induced by the family with the holomorphic coordinates $z:=(z^i)$ of
$X_0$, under a coordinate covering $\{\mathcal{U}_j\}$ of
$\mathcal{X}$, when $t$ is assumed to be fixed, as the standard
notions in deformation theory described at the beginning of
\cite[Chapter 4]{MK}. This family induces a canonical differentiable family of
integrable Beltrami differentials on $X_0$, denoted by $\varphi(z,t)$,
$\varphi(t)$ and $\varphi$ interchangeably, to be explained at the beginning of
Section \ref{kahler-section}.

In \cite{RZ15}, the first and third authors introduce an extension
map
$$e^{\iota_{\varphi(t)}|\iota_{\overline{\varphi(t)}}}:
 A^{p,q}(X_0)\> A^{p,q}(X_t),$$ which plays an important role in
 this paper.
\begin{definition}\label{map}\rm
For $s\in
 A^{p,q}(X_0)$, we define
$$\label{lbro} e^{\iota_{\varphi(t)}|\iota_{\overline{\varphi(t)}}}(s)=
s_{i_1\cdots i_pj_1\cdots
j_q}(z(\zeta))\left(e^{\iota_{\varphi(t)}}\left(dz^{i_1}\wedge\cdots\wedge
dz^{i_p}\right)\right)\wedge
\left(e^{\iota_{\overline{\varphi(t)}}}\left(d\overline{z}^{j_1}\wedge\cdots\wedge
d\overline{z}^{j_q}\right)\right),
$$
where $s$ is locally written as
$$s=s_{i_1\cdots i_pj_1\cdots
j_q}(z)dz^{i_1}\wedge\cdots\wedge dz^{i_p}\wedge
d\overline{z}^{j_1}\cdots\wedge d\overline{z}^{j_q}$$ and the
operators $e^{\iota_{\varphi(t)}}$,
$e^{\iota_{\overline{\varphi(t)}}}$ follow the convention
\eqref{0e-convention}. It is easy to check that this map is a real
linear isomorphism as in \cite[Lemma $2.8$]{RZ15}.
\end{definition}

Now let us describe our approach to reprove Kodaira--Spencer's
local stability of K\"{a}hler structures. We will use Kuranishi's completeness theorem \cite{ku} to reduce the proof to Kuranishi family $\varpi:\mc{K}\to T$ and a power series
method to construct a natural K\"{a}hler extension
$\tilde{\omega}_t$ of the K\"{a}hler form $\omega_0$ on $X_0$, such that $\tilde{\omega}_t$ is a
K\"{a}hler form on the general fiber $\varpi^{-1}(t)=X_t$. More
precisely, the extension is given by
$$e^{\iota_{\varphi}|\iota_{\o{\varphi}}}:A^{1,1}(X_0)\to A^{1,1}(X_t),\quad \omega_0\to \tilde{\omega}_t:=e^{\iota_{\varphi}|\iota_{\o{\varphi}}}(\omega(t)),$$
where $\omega(t)$ is a family of smooth $(1,1)$-forms on $X_0$,
depending smoothly on $t$, and $\omega(0)=\omega_0$. This method
is developed in
\cite{LSY,Sun,SY,lry,RZ,RZ2,RZ15}.
The following proposition will be used many times in this paper:
\begin{proposition}[{\cite[Theorem 3.4]{lry}, \cite[Proposition 2.2]{RZ15}}]\label{main1}
Let $\phi\in A^{0,1}(X,T^{1,0}_X)$. Then on the space $A^{*,*}(X)$,
\beq\label{ext-old} d\circ e^{\iota_\phi}=e^{\iota_\phi}(d+\p\circ
\iota_\phi-\iota_\phi\circ
\p-\iota_{\db\phi-\frac{1}{2}[\phi,\phi]}).\eeq
\end{proposition}

By a careful use of Proposition \ref{main1}, we can show that
$d(e^{\iota_{\varphi}|\iota_{\b{\varphi}}}(\omega))=0$ with $\omega\in A^{1,1}(X_0)$ is equivalent
to
 \begin{equation}\label{0w.18}
    \begin{cases}
      ([\p,\iota_{\varphi}]+\b{\p})(\1-\b{\varphi}\varphi)\lc\omega=0,\\
      ([\b{\p},\iota_{\b{\varphi}}]+\p)(\1-\varphi\b{\varphi})\lc\omega=0.
    \end{cases}
 \end{equation}
In this paper we always use the notations: $\varphi
\overline{\varphi} = \overline{\varphi} \lc \varphi,
\overline{\varphi} \varphi = \varphi \lc \overline{\varphi}$ and
$\1$ is the identity operator defined as:
$$\1=\frac{1}{p+q}\left(\sum_i^n dz^i\otimes \frac{\p\ }{\p z^i}+\sum_i^n d\bar z^i\otimes \frac{\p\ }{\p \bar z^i}\right)$$
when it acts on $(p,q)$-forms of a complex manifold. Obviously, the
identity operator is a real operator. It is worth noticing that this
definition is a little different from that in \cite{RZ15}, where for
a $(p,q)$-form $\alpha$ on a complex manifold,
$$\1(\alpha)= (p+q)\cdot\alpha$$
since the constant factor $\frac{1}{p+q}$ didn't appear in that
definition.

Then one has the crucial reduction:
\begin{proposition}[=Proposition \ref{2.9}]\label{02.9}
 If the power series
  $$\omega(t)=\sum_{k=0}^{\infty}\sum_{i+j=k}\omega_{i,j}t^i\b{t}^j$$ of $(1,1)$-forms on $X_0$ is
a real formal solution of the system of equations
\begin{equation}\label{0reduced-equsys}
    \begin{cases}
      \b{\p}\omega=\b{\p}(\b{\varphi}\varphi\l\omega-\varphi\l\b{\varphi}\l\omega)-\p(\varphi\l\omega),\\
      \p\omega=\p(\varphi\b{\varphi}\l\omega-\b{\varphi}\l\varphi\l\omega)-\b{\p}(\b{\varphi}\l\omega),\\
    \end{cases}
\end{equation}
then it is also one of the system (\ref{0w.18}).
\end{proposition}
Base on Observation \ref{expslt}, one can solve the system
\eqref{0reduced-equsys} easily by the K\"ahlerian condition on the
reference fiber inductively. Then the  H\"older convergence and regularity
argument in Subsections \ref{convergence} and \ref{regularity} gives
rise to the desired K\"ahler form on the deformation $X_t$ of
$X_0$.

In Section \ref{balanced}, we will discuss the local stability
problem of balanced structures on the complex manifolds also satisfying
various $\p\db$-lemmata. The \emph{$(n-1,n)$-th mild
$\p\b{\p}$-lemma} is introduced in
Subsection \ref{subs-mild}: an $n$-dimensional complex manifold $X$
satisfies the $(n-1,n)$-th mild $\p\b{\p}$-lemma, if  for any
$(n-2,n)$-complex differential form $\xi$ on $X$, there exists an
$(n-2,n-1)$-form $\theta$ on $X$ such that
\[ \p \db \theta = \p \xi. \]
Then another main result of this paper can be described as follows:
\begin{theorem}[=Theorem \ref{blc-inv}]\label{0blc-inv}
Let $X_0$ be a compact balanced manifold of complex dimension $n$,
satisfying the $(n-1,n)$-th mild $\p\db$-lemma. Then $X_t$ also
admits a balanced metric for $t$ small.
\end{theorem}

Nilmanifolds with invariant abelian complex structures satisfy the $(n-1,n)$-th mild $\p\db$-lemma as shown in
Corollary \ref{abelian}.
Many examples and results, such as \cite[Proposition 4.4, Remark 4.6, Remark 4.7 and Example 4.10]{au}
and \cite[Corollary 8 and Corollary 9]{FY}, become consequences of this theorem.

It is an obvious generalization of C. Wu's result \cite[Theorem 5.13]{w} that
the balanced structure is preserved under small deformation if the
reference fiber satisfies the $\p\db$-lemma. Fu--Yau \cite[Theorem
6]{FY} show that the balanced structure is deformation open,
assuming that the $(n-1,n)$-th weak $\p\db$-lemma, introduced by them, holds on the
general fibers $X_t$ for $t \neq 0$. Recall that the \emph{$(n-1,n)$-th weak $\p\db$-lemma} on a compact complex manifold $X$  says that if for any real
$(n-1,n-1)$-form $\psi$ on $X$ such that $\db \psi$ is $\p$-exact,
there exists an $(n-2,n-1)$-form $\theta$ on $X$ such that $ \p \db
\theta = \db \psi. $ It is well known from \cite{ab} that a small deformation
of the Iwasawa manifold, which satisfies the $(2,3)$-th weak
$\p\db$-lemma but does not satisfy the mild one from Example
\ref{ex-dms-notms}, may not be balanced. Thus, the condition
\lq\lq $(n-1,n)$-th mild $\p\db$-lemma" in Theorem \ref{0blc-inv} can't be replaced
by the weak one. In \cite[Example 3.7]{UV} or Example \ref{not-fy}, Ugarte--Villacampa construct an explicit family of nilmanifolds $I_{\lambda}$ of complex dimension $3$ with invariant balanced metric on each fiber, for $ \lambda \in [0,1)$. However, the general fiber $X_t$ doesn't satisfy the $(n-1,n)$-th weak $\p\db$-lemma for $t \neq 0$. Fortunately, the mild one holds on the reference fiber and thus, Fu--Yau's theorem is not applicable to this example, while ours is applicable.

Based on Fu--Yau's theorem, Angella--Ugarte \cite[Theorem
4.9]{au} prove that if $X_0$ admits a locally conformal balanced
metric and satisfies the $(n-1,n)$-th strong $\p\db$-lemma, then
$X_t$ is balanced for $t$ small. They define the \emph{$(n-1,n)$-th strong
$\p\db$-lemma} of a complex manifold $X$ as: for any $\p$-closed $(n-1,n)$-form $\Gamma$
on $X$ of the type $\Gamma=\p\xi+\db \psi$, there exists a suitable $\theta$
on $X$ with $ \p \db \theta = \Gamma. $ Interestingly, a complex manifold satisfies the $(n-1,n)$-th strong
$\p\db$-lemma if and only if both of the mild one and the dual mild
one hold on it, and the $(n-1,n)$-th dual mild $\p\db$-lemma
guarantees that a locally conformal balanced metric is also a global
one. The \emph{$(n-1,n)$-th dual mild $\p\db$-lemma} refers to that the induced mapping
$\iota^{n-1,n}_{BC,\db}: H^{n-1,n}_{BC}(X) \rightarrow
H^{n-1,n}_{\db}(X)$ from the $({n-1,n})$-th Bott--Chern cohomology group by the identity map is injective.
From this point of view, one can understand
Angella--Ugarte's theorem more intrinsically.
Besides, there indeed exist examples that satisfy the $(n-1,n)$-th mild $\p\db$-lemma but not the strong one,
such as a nilmanifold endowed with an invariant abelian complex structure
from Corollary \ref{abelian} and \cite[Proposition 2.9]{AU}.

Similarly to the K\"ahler case, we will prove Theorem \ref{0blc-inv} in Subsection \ref{proof-bal}
by reducing the proof to Kuranishi family and constructing a power series $\Omega(t)\in A^{n-1,n-1}(X_0)$ such
that
\begin{equation}\label{0wan.2}
    \begin{cases}
      d(e^{\iota_{\varphi}|\iota_{\b{\varphi}}}(\Omega(t)))=0,\\
      \Omega(t)=\o{\Omega(t)},\\
      \Omega(0)=\omega^{n-1},
    \end{cases}
\end{equation}
where $\om$ is the original balanced metric on $X_0$.  By
Proposition \ref{main1} again and setting
\[
\tilde{\Omega}(t)=e^{-\iota_{(\1-\b{\varphi}\varphi)^{-1}\b{\varphi}}}\circ
e^{-\iota_{\varphi}}\circ
e^{\iota_{\varphi}|\iota_{\b{\varphi}}}(\Omega(t)),
\]
one reduces the obstruction system \eqref{0wan.2} of
equations to:
\begin{equation}\label{wan.15}
    \begin{cases}
      \big(\b{\p}+\p\circ \iota_{\varphi}+\b{\p}\circ \iota_{\varphi}\circ \iota_{(\1-\b{\varphi}\varphi)^{-1}\b{\varphi}}+\frac{1}{2}\p\circ\ \iota_{\varphi}\circ \iota_{\varphi}\circ \iota_{(\1-\b{\varphi}\varphi)^{-1}\b{\varphi}}\big)\tilde{\Omega}(t)=0,\\
      \big(\p+\b{\p}\circ \iota_{(\1-\b{\varphi}\varphi)^{-1}\b{\varphi}}+\p\circ \iota_{\varphi}\circ \iota_{(\1-\b{\varphi}\varphi)^{-1}\b{\varphi}}\big)\tilde{\Omega}(t)=0,\\
      \tilde{\Omega}(0)=\omega^{n-1},
    \end{cases}
 \end{equation}
and solves this system formally also by the power series method,
when the \lq\lq$(n-1,n)$-th mild $\p\db$-lemma" just comes
from the strategy of using Observation \ref{expslt} to resolve \eqref{wan.15}.
 Inspired by the  H\"older  convergence
and regularity argument for the integrable Beltrami differential
$\varphi(t)$ in the deformation theory of complex structures, we
complete that of $\tilde{\Omega}(t)$. The key point is to deal
with the Green's operator
in the explicit canonical solution of $\tilde{\Omega}(t)$ there. It is worthy noticing that this analogous proof is different in the resolution of the obstruction equation \eqref{wan.15} by use of the Bott--Chern Green's operator and thus undergoes more difficult regularity argument. It is also applicable to the $(1,1)$-case without the K\"ahler or deformation invariance of $(1,1)$-Bott--Chern numbers assumption on the reference fiber $X_0$ essentially in Kodaira--Spencer's original proof:

\begin{proposition}[]\label{}
Assume that the reference
fiber $X_0$ satisfies the $(1,2)$-th mild $\p\db$-lemma, that is, any $d$-closed $\p$-exact $(1,2)$-form on $X_0$ is also $\p\db$-exact. Then any $d$-closed
$(1,1)$-form $\Omega_0$ on $X_0$ can be extended as a $d$-closed $(1,1)$-form
$\Omega_t$ varying smoothly at $t$ on its
small differentiable deformation $X_t$.
\end{proposition}

Notice that the $(1,2)$-mild $\p\db$-lemma is different from the $\p\db$-lemma on a complex manifold.
It is easy to see that the $(1,2)$-mild $\p\db$-lemma amounts to the injectivity of the mapping
\[ \iota_{BC,\p}^{1,2}:H^{1,2}_{BC}(X) \rightarrow H^{1,2}_{\p}(X),\]
or equivalently, that of
\[ \iota_{BC,\db}^{2,1}:H^{2,1}_{BC}(X) \rightarrow H^{2,1}_{\db}(X).\]
See Example \ref{NdbMdb} for an example of a non-$\p \db$-manifold which satisfies the (1,2)-mild $\p\db$-lemma.

Section \ref{other-str} is devoted to the local stabilities of $p$-K\"ahler structures.
In Subsection \ref{d-closed}, by means of Wu's result \cite[Theorem 5.13]{w}, we
will use the cohomological method, originally from \cite{KS}, to
get:
\begin{proposition}[=Proposition \ref{d-ext}]\label{0d-ext}
Let $r$ and $s$ be non-negative integers. Assume that the reference
fiber $X_0$ satisfies the $\p\db$-lemma. Then any $d$-closed
$(r,s)$-form $\Omega_0$ and $\p_0\db_0$-closed $(r,s)$-form $\Psi_0$
on $X_0$ can be extended unobstructed to a $d$-closed $(r,s)$-form
$\Omega_t$ and a $\p_t\db_t$-closed $(r,s)$-form $\Psi_t$ on its
small differentiable deformation $X_t$, respectively.
\end{proposition}
We can also prove this proposition in the $(n-1,n-1)$-case by another way inspired by the
results of \cite{au,FY,w}. It is impossible to prove Theorem \ref{0blc-inv}
by this method since the proof would rely on the deformation
openness of $(n-1,n)$-th mild $\p\db$-lemma, which contradicts with
Ugarte--Villacampa's Example \ref{not-fy}.

Finally, in Subsection \ref{p-kahler}, we study some basic
properties of $p$-K\"ahler structures, a possibly more intrinsic
notion for the local stabilities of complex structures. Based on
un-obstruction of extension for transverse forms and (the proof of) Proposition \ref{0d-ext}, we use two different approaches to obtain:
\begin{theorem}[= Theorem \ref{c-p-kahler}+ Remark \ref{p-k-rem}]
For any positive integer $p\leq n-1$, any
small differentiable deformation $X_t$ of a compact
$p$-K\"ahler manifold $X_0$ satisfying the deformation invariance of $(p,p)$-Bott--Chern numbers is still
$p$-K\"ahlerian.
\end{theorem}

\noindent \textbf{Notation} Without specially mentioned, the
hermitian metrics will be identified with their fundamental forms.
We only consider the
small differentiable deformations in this paper, i.e.,
the parameter $t$ is always assumed to be small. All sub-indices in
the power series, such as $i,j,\cdots$, are set no less than zero,
while Einstein sum convention is adopted in the local settings and
calculations. In many places, we fix a K\"ahler metric or a balanced
one on the reference fiber of the differentiable family to induce the
dual operators and the associated Hodge decomposition with respect
to $\db$ and $\p$ on it. A complex differential form, linear
operator or current is called \emph{real} if it is invariant under conjugation.

\textbf{Acknowledgement}: We would like to express our
gratitude to Professors Kefeng Liu, Huitao Feng, Fangyang
Zheng for their constant encouragement and help, Lucia Alessandrini, Xiaojun Huang, Dan Popovici, Luis Ugarte
for their interest and useful comments. The first
author would like to thank Dr. Xiaoshan Li, Jie Tu, Yi Wang for several
useful discussions and Wanke Yin for his lectures on positive currents at the seminar.
This work also benefits from NCTS of
Taiwan University, The Institute of Mathematical Sciences at the
Chinese University of Hong Kong, the Mathematics Department of UCLA
and the Institut de Mathematiques de Toulouse.

\section{Stability of K\"{a}hler structures}\label{kahler-section}
We introduce some basics on deformation theory of complex structures to be used throughout
this paper. For holomorphic family of compact complex manifolds, we adopt the definition \cite[Definition 2.8]{k}; while for differentiable one, we follow:
\begin{definition}[{\cite[Definition 4.1]{k}}]\label{} Let $\mc{X}$ be a differentiable manifold, $B$ a domain of $\mathbb{R}^k$ and $\pi$ a smooth map of $\mc{X}$ onto $B$.
By a \emph{differentiable family of $n$-dimensional compact complex manifolds} we mean the triple $\pi:\mc{X}\to B$ satisfying the following conditions:
\begin{enumerate}[$(i)$]
    \item \label{}
The rank of the Jacobian matrix of $\pi$ is equal to $k$ at every point of $\mc{X}$;
    \item \label{}
For each point $t\in B$, $\pi^{-1}(t)$ is a compact connected subset of $\mc{X}$;
    \item \label{}
$\pi^{-1}(t)$ is the underlying differentiable manifold of the $n$-dimensional compact complex manifold $X_t$ associated to each $t\in B$;
     \item \label{}
There is a locally finite open covering $\{\mathcal{U}_j\ |\ j=1,2,\cdots\}$ of $\mc{X}$ and complex-valued smooth functions $\zeta_j^1(p),\cdots,\zeta_j^n(p)$, defined on $\mathcal{U}_j$
such that for each $t$, $$\{p\rightarrow (\zeta_j^1(p),\cdots,\zeta_j^n(p))\ |\ \mathcal{U}_j\cap \pi^{-1}(t)\neq \emptyset\}$$ form a system of local holomorphic coordinates of $X_t$.
\end{enumerate}
\end{definition}

Let us sketch Kodaira--Spencer's proof of local stability theorem
\cite{KS}. Let $F_t$ be the orthogonal projection to the kernel
$\mathds{F}_t$ of the \emph{first $4$-th order Kodaira--Spencer
operator} (also often called \emph{Bott--Chern Laplacian})
\begin{equation}\label{bc-Lap}
\square_{{BC},t}=\p_t\db_t\db_t^*\p_t^*+\db_t^*\p_t^*\p_t\db_t+\db_t^*\p_t\p_t^*\db_t+\p_t^*\db_t\db_t^*\p_t+\db_t^*\db_t+\p_t^*\p_t
\end{equation}
and $\G_t$ the corresponding Green's operator with respect to $\alpha_t$ on $X_t$.
Here
$$\alpha_t=\sqrt{-1}g_{i\bar{j}}(\zeta,t)d \zeta^i \w d \overline{\zeta}^j$$
is a hermitian metric on $X_t$ depending differentiably on $t$ and $\alpha_0$ is a K\"ahler metric on $X_0$. By a cohomological
argument with the upper semi-continuity theorem, they prove that
$F_t$ and $\G_t$ depend differentiably on $t$. Then they can
construct the desired K\"ahler metric on $X_t$ as
$$\widetilde{\alpha_t}=\frac{1}{2}(F_t\alpha_t+\overline{F_t\alpha_t}).$$
See also \cite[Subsection 9.3]{V}.

Now let us describe our basic philosophy to reprove the
Kodaira--Spencer's local stability of K\"{a}hler structures.
By (the proof of) Kuranishi's completeness theorem \cite{ku}, for any compact complex manifold $X_0$, there exists a complete holomorphic family
$\varpi:\mc{K}\to T$ of complex manifolds at the reference point $0\in T$ in the sense that for any differentiable family $\pi:\mc{X}\to B$ with $\pi^{-1}(s_0)=\varpi^{-1}(0)=X_0$, there is a sufficiently small neighborhood $E\subseteq B$ of $s_0$, and smooth maps $\Phi: \mathcal {X}_E\rightarrow \mathcal {K}$,  $\tau: E\rightarrow T$ with $\tau(s_0)=0$ such that the diagram commutes
$$\xymatrix{\mathcal {X}_E \ar[r]^{\Phi}\ar[d]_\pi& \mathcal {K}\ar[d]^\varpi\\
(E,s_0)\ar[r]^{\tau}  & (T,0),}$$
$\Phi$ maps $\pi^{-1}(s)$ biholomorphically onto $\varpi^{-1}(\tau(s))$ for each $s\in E$, and $$\Phi: \pi^{-1}(s_0)=X_0\rightarrow \varpi^{-1}(0)=X_0$$ is the identity map.
This family is called \emph{Kuranishi family} and constructed as follows. Let $\{\eta_\nu\}_{\nu=1}^m$ be a base for $\mathbb{H}^{0,1}(X_0,T^{1,0}_{X_0})$, where some suitable hermitian metric is fixed on $X_0$ and $m\geq 1$; Otherwise the complex manifold $X_0$ would be \emph{rigid}, i.e., for any differentiable family $\kappa:\mc{M}\to P$ with $s_0\in P$ and $\kappa^{-1}(s_0)=X_0$, there is a neighborhood $V \subseteq P$ of $s_0$ such that $\kappa:\kappa^{-1}(V)\to V$ is trivial. Then one can construct a holomorphic family
\begin{equation}\label{phi-ps-pp}\varphi(t) = \sum_{|I|=1}^{\infty}\varphi_{I}t^I:=\sum_{j=1}^{\infty}\varphi_j(t),\ I=(i_1,\cdots,i_m),\ t=(t_1,\cdots,t_m)\in \mathbb{C}^m,\end{equation} {for $|t|< \rho$ a small positive constant,} of Beltrami differentials as follows:
\begin{equation}\label{phi-ps-0}
 \varphi_1(t)=\sum_{\nu=1}^{m}t_\nu\eta_\nu
\end{equation}
and for $|I|\geq 2$,
\begin{equation}\label{phi-ps}
  \varphi_I=\frac{1}{2}\db^*\G\sum_{J+L=I}[\varphi_J,\varphi_{L}].
\end{equation}
It is obvious that $\varphi(t)$ satisfies the equation
$$\varphi(t)=\varphi_1+\frac{1}{2}\db^*\G[\varphi(t),\varphi(t)].$$
Let
$$T=\{t\ |\ \mathbb{H}[\varphi(t),\varphi(t)]=0 \}.$$
Thus, for each $t\in T$, $\varphi(t)$ satisfies
\begin{equation}\label{int}
\b{\p}\varphi(t)=\frac{1}{2}[\varphi(t),\varphi(t)],
\end{equation}
and determines a complex structure $X_t$ on the underlying differentiable manifold of $X_0$. More importantly, $\varphi(t)$ represents the complete holomorphic family $\varpi:\mc{K}\to T$ of complex manifolds. Roughly speaking, Kuranishi family $\varpi:\mc{K}\to T$ contains all sufficiently small differentiable deformations of $X_0$.

By means of these, one can reduce the local stability Theorem \ref{stab-Kahler} to the Kuranishi family by shrinking $E$ if necessary, that is, it suffices to construct a K\"{a}hler metric on each $X_t$. From now on, one uses $\varphi(t)$ and  $\varphi$ interchangeably to denote this holomorphic family of integrable Beltrami differentials, and assumes $m=1$ for simplicity.

Using this reduction, we should
construct a natural K\"{a}hler extension $\tilde{\omega}_t$ of a given K\"{a}hler metric
$\omega_0$ on $X_0$, such that $\tilde{\omega}_t$ is a K\"{a}hler  metric on
the general fiber $\varpi^{-1}(t)=X_t$. More precisely, the extension
is given by
$$e^{\iota_{\varphi}|\iota_{\o{\varphi}}}:A^{1,1}(X_0)\to A^{1,1}(X_t),\quad \omega_0\to \tilde{\omega}_t:=e^{\iota_{\varphi}|\iota_{\o{\varphi}}}(\omega(t)),$$
where $\omega(t)$ is a family of smooth $(1,1)$-forms on $X_0$,
depending smoothly on $t$, and $\omega(0)=\omega_0$. As we need to
construct a K\"{a}hler extension, the following conditions appear:
\begin{enumerate}
  \item $d\left(e^{\iota_{\varphi}|\iota_{\o{\varphi}}}(\omega(t))\right)=0,$
  \item $\omega(t)=\o{\omega(t)}$.
\end{enumerate}
As $t$ is sufficiently small, $\omega(t)$ is positive by the
convergence argument and thus
$e^{\iota_{\varphi}|\iota_{\o{\varphi}}}(\omega(t))$ is a K\"{a}hler
form on $X_t$. Here we will use an elementary power series method to
complete the construction.

\subsection{Obstruction equations}
We will discuss the obstruction equation to extend the
$d$-closed pure-type complex differential forms on a complex manifold to the
ones on its
small differentiable deformation in this subsection. The argument in this and next subsections is applicable to a general differentiable family of complex manifolds.

For a general $\alpha\in A^{p,q}(X_0)$, by Proposition \ref{main1} and the
integrability condition \eqref{int}, one has
\begin{align}\label{2.5.1}
  \begin{split}
    d(e^{\iota_{\varphi}|\iota_{\b{\varphi}}}(\alpha))&=d\circ e^{\iota_{\varphi}}\circ e^{-\iota_{\varphi}}\circ e^{\iota_{\varphi}|\iota_{\b{\varphi}}}(\alpha)\\
    &=e^{\iota_{\varphi}}\circ\left([\p,\iota_{\varphi}]+\b{\p}+\p\right)\circ e^{-\iota_{\varphi}}\circ e^{\iota_{\varphi}|\iota_{\b{\varphi}}}(\alpha)\\
    &=e^{\iota_{\varphi}|\iota_{\b{\varphi}}}\circ\left(e^{-\iota_{\varphi}|-\iota_{\b{\varphi}}}\circ e^{\iota_{\varphi}}\circ\left([\p,\iota_{\varphi}]+\b{\p}+\p\right)
    \circ e^{-\iota_{\varphi}}\circ
    e^{\iota_{\varphi}|\iota_{\b{\varphi}}}(\alpha)\right).
  \end{split}
\end{align}
Here
$$e^{-\iota_{\varphi(t)}|-\iota_{\overline{\varphi(t)}}}: A^{p,q}(X_t)\>
A^{p,q}(X_0)$$
 is the
 inverse map of
$e^{\iota_{\varphi(t)}|\iota_{\overline{\varphi(t)}}}$, defined by
\begin{equation}\label{ephi-inv}
 \begin{aligned}
   & e^{-\iota_{\varphi(t)}|-\iota_{\overline{\varphi(t)}}}(s)\\
 =&s_{i_1\cdots i_p \bar{j}_1 \cdots \bar{j}_q}(\zeta)
 \bigg(e^{-\iota_{\varphi(t)}}\Big( \big(dz^{i_1}+\varphi(t)\lc
dz^{i_1}\big) \wedge \cdots \wedge
\big(dz^{i_p}+\varphi(t) \lc dz^{i_p}\big) \Big) \wedge\\
&\qquad\qquad\qquad e^{-\iota_{\overline{\varphi(t)}}} \Big(
\big(d\overline{z}^{j_1}+\overline{\varphi(t)}\lc
d\overline{z}^{j_1}\big) \wedge \cdots \wedge \big(
d\overline{z}^{j_q}+\overline{\varphi(t)}\lc d\overline{z}^{j_q}
\big)\Big)\bigg),
\end{aligned}
\end{equation}
where $s\in A^{p,q}(X_t)$ is locally written as
\begin{align*}
s=s_{i_1\cdots i_p \bar{j}_1 \cdots \bar{j}_q}(\zeta)
&(dz^{i_1}+\varphi(t)\lc dz^{i_1}) \wedge \cdots\wedge
(dz^{i_p}+\varphi(t)\lc dz^{i_p})\\
\wedge &(d\overline{z}^{j_1}+\overline{\varphi(t)}\lc
d\overline{z}^{j_1})\wedge\cdots\wedge
(d\overline{z}^{j_q}+\overline{\varphi(t)}\lc d\overline{z}^{j_q}),
\end{align*}
 and the operators $e^{-\iota_{\varphi(t)}}$,
$e^{-\iota_{\overline{\varphi(t)}}}$ also follow the convention
(\ref{0e-convention}) as in the proof of \cite[Lemma $2.8$]{RZ15}.
We introduce one more new notation $\Finv$ to denote the
\emph{simultaneous contraction} on each component of a complex
differential form. For example,
$(\1-\b{\varphi}\varphi+\b{\varphi})\Finv\alpha$ means that the
operator $(\1-\b{\varphi}\varphi+\b{\varphi})$ acts on $\alpha$
simultaneously as:
\begin{align}\label{2.8}
\begin{split}
&(\1-\b{\varphi}\varphi+\b{\varphi})\Finv(f_{i_1\cdots
i_p\o{j_1}\cdots\o{j_q}}dz^{i_1}\wedge\cdots \wedge dz^{i_p}\wedge
d\b{z}^{j_1}\wedge\cdots \wedge d\b{z}^{j_q})
\\
=&f_{i_1\cdots
i_p\o{j_1}\cdots\o{j_q}}(\1-\b{\varphi}\varphi+\b{\varphi})\l
dz^{i_1} \wedge\cdots\wedge(\1-\b{\varphi}\varphi +\b{\varphi})\l
dz^{i_p}
\\&\qquad\quad\quad\wedge(\1-\b{\varphi}\varphi+\b{\varphi})\l
d\b{z}^{j_1}\wedge\cdots\wedge(\1-\b{\varphi}\varphi+\b{\varphi})\l
d\b{z}^{j_q},
\end{split}
\end{align}
if $\alpha$ is locally expressed by:
$$\alpha=f_{i_1\cdots i_p\o{j_1}\cdots\o{j_q}}dz^{i_1}\wedge \cdots\wedge dz^{i_p}\wedge d\b{z}^{j_1}\wedge \cdots\wedge d\b{z}^{j_q}.$$
This new simultaneous contraction is well-defined since $\varphi(t)$
is a  global $(1,0)$-vector valued $(0, 1)$-form on $X_0$ (See
\cite[Pages $150-151$]{MK}) as reasoned in \cite[Proof of Lemma
2.8]{RZ15}. Notice that
$(\1-\b{\varphi}\varphi+\b{\varphi})\Finv\alpha\neq
\1\Finv\alpha-\b{\varphi}\varphi\Finv\alpha+\b{\varphi}\Finv\alpha$
in general. Using this notation, one can rewrite the extension map
$e^{\iota_{\varphi}|\iota_{\b{\varphi}}}$ in Definition \ref{map}:
$$e^{\iota_{\varphi}|\iota_{\b{\varphi}}}=(\1+\varphi+\b{\varphi})\Finv.$$
Then one has:
\begin{lemma}For any $\alpha\in A^{p,q}(X_0)$,
\begin{align}\label{2.3}
  e^{-\iota_{\varphi}}\circ e^{\iota_{\varphi}|\iota_{\b{\varphi}}}(\alpha)=(\1-\b{\varphi}\varphi+\b{\varphi})\Finv\alpha.
\end{align}
\end{lemma}

\begin{proof}
Following the above notations and the definition of
$e^{\iota_{\varphi}|\iota_{\b{\varphi}}}$, we have
\begin{align}\label{2.1}
\begin{split}
  e^{\iota_{\varphi}|\iota_{\b{\varphi}}}(\alpha)&=f_{i_1\cdots i_p\o{j_1}\cdots\o{j_q}}(\1+\varphi)\l dz^{i_1}\wedge \cdots\wedge (\1+\varphi)\l dz^{i_p}\\
  &\qquad\qquad\quad\wedge (\1+\b{\varphi})\l d\b{z}^{j_1}\wedge \cdots\wedge (\1+\b{\varphi})\l d\b{z}^{j_q}.
  \end{split}
\end{align}

On the other hand,
\begin{align}\label{2.2}
  \begin{split}
    &e^{\iota_{\varphi}}\circ(\1-\b{\varphi}\varphi+\b{\varphi})\Finv\alpha\\
    =&e^{\iota_{\varphi}}\Big(f_{i_1\cdots i_p\o{j_1}\cdots\o{j_q}}(\1-\b{\varphi}\varphi+\b{\varphi})\l dz^{i_1}\wedge \cdots\wedge (\1-\b{\varphi}\varphi+\b{\varphi})\l dz^{i_p}\\
  &\qquad\qquad\qquad\wedge (\1-\b{\varphi}\varphi+\b{\varphi})\l d\b{z}^{j_1}\wedge \cdots\wedge (\1-\b{\varphi}\varphi+\b{\varphi})\l d\b{z}^{j_q}\Big)\\
  =&f_{i_1\cdots i_p\o{j_1}\cdots\o{j_q}}(\1+\varphi)\l(\1-\b{\varphi}\varphi+\b{\varphi})\l dz^{i_1}\wedge \cdots\wedge (\1+\varphi)\l(\1-\b{\varphi}\varphi+\b{\varphi})\l dz^{i_p}\\
  &\qquad\quad\quad\wedge (\1+\varphi)\l(\1-\b{\varphi}\varphi+\b{\varphi})\l d\b{z}^{j_1}\wedge \cdots\wedge (\1+\varphi)\l(\1-\b{\varphi}\varphi+\b{\varphi})\l d\b{z}^{j_q}\\
  =&f_{i_1\cdots i_p\o{j_1}\cdots\o{j_q}}(\1+\varphi)\l dz^{i_1}\wedge \cdots\wedge (\1+\varphi)\l dz^{i_p}\wedge (\1+\b{\varphi})\l d\b{z}^{j_1}\wedge \cdots\wedge (\1+\b{\varphi})\l d\b{z}^{j_q},
  \end{split}
\end{align}
where the last equality holds by
$$(\1+\varphi)\l(\1-\b{\varphi}\varphi+\b{\varphi})\l dz^{i_k}=(\1+\varphi)\l dz^{i_k}$$
and
$$(\1+\varphi)\l(\1-\b{\varphi}\varphi+\b{\varphi})\l d\b{z}^{j_k}=(\1-\b{\varphi}\varphi+\b{\varphi})\l d\b{z}^{j_k}+(\b{\varphi}\varphi)\l d\b{z}^{j_k}=(\1+\b{\varphi})\l d\b{z}^{j_k}.$$
Therefore, (\ref{2.3}) is proved by (\ref{2.1}) and (\ref{2.2}).
\end{proof}

Similarly:
\begin{lemma} For any $\alpha\in A^{p,q}(X_0)$,
 \begin{align}\label{2.4.1}
   e^{-\iota_{\varphi}|-\iota_{\b{\varphi}}}\circ e^{\iota_{\varphi}}(\alpha)=\left((\1-\b{\varphi}\varphi)^{-1}-(\1-\b{\varphi}\varphi)^{-1}\b{\varphi}\right)\Finv\alpha,
 \end{align}
 where $\left((\1-\b{\varphi}\varphi)^{-1}-(\1-\b{\varphi}\varphi)^{-1}\b{\varphi}\right)$ acts on $\alpha$ just as
 (\ref{2.8}).
\end{lemma}
Notice that the more intrinsic proofs of \eqref{2.3} and
\eqref{2.4.1} can be found in the proof of \cite[Proposition
2.12]{RZ15}.
 Substituting (\ref{2.3}) and (\ref{2.4.1}) into (\ref{2.5.1}), one has
\begin{proposition} For any $\alpha\in A^{p,q}(X_0)$,
\begin{equation}
 \begin{aligned}\label{2.7}
    &d(e^{\iota_{\varphi}|\iota_{\b{\varphi}}}(\alpha))\\=\ &
    e^{\iota_{\varphi}|\iota_{\b{\varphi}}}\left(\left((\1-\b{\varphi}\varphi)^{-1}-(\1-\b{\varphi}\varphi)^{-1}\b{\varphi}\right)\Finv\left([\p,\iota_{\varphi}]
    +\b{\p}+\p\right)(\1-\b{\varphi}\varphi+\b{\varphi})\Finv\alpha\right).
 \end{aligned}
\end{equation}
\end{proposition}

From (\ref{2.4.1}), we know that
$$
  e^{-\iota_{\varphi}|-\iota_{\b{\varphi}}}\circ e^{\iota_{\varphi}}: A^{p,q}(X_0)\to \bigoplus_{i=0}^{\min\{q,n-p\}}A^{p+i,q-i}(X_0).
$$
Thus, by carefully comparing the form types in both sides  of
(\ref{2.7}), we have
 \begin{align}\label{2.6}
   \b{\p}_t(e^{\iota_{\varphi}|\iota_{\b{\varphi}}}(\alpha))=e^{\iota_{\varphi}|\iota_{\b{\varphi}}}\left((\1-\b{\varphi}\varphi)^{-1}\Finv([\p,\iota_{\varphi}]+\b{\p})(\1-\b{\varphi}\varphi)\Finv\alpha\right).
 \end{align}
The $d$-closed condition
$d(e^{\iota_{\varphi}|\iota_{\b{\varphi}}}(\alpha))=0$ amounts to
$$\begin{cases}
\b{\p}_t(e^{\iota_{\varphi}|\iota_{\b{\varphi}}}(\alpha))=0,\\
 \p_t(e^{\iota_{\varphi}|\iota_{\b{\varphi}}}(\alpha))=\o{\b{\p}_t\circ e^{\iota_{\varphi}|\iota_{\b{\varphi}}}}({\alpha})=0,
    \end{cases}
$$
which, together with (\ref{2.6}), implies that
 \begin{equation}\label{w.18}
    \begin{cases}
      ([\p,\iota_{\varphi}]+\b{\p})(\1-\b{\varphi}\varphi)\Finv\alpha=0,\\
      ([\b{\p},\iota_{\b{\varphi}}]+\p)(\1-\varphi\b{\varphi})\Finv\alpha=0
    \end{cases}
 \end{equation}
by the invertibility of the operators
$e^{\iota_{\varphi}|\iota_{\b{\varphi}}}$,
$(\1-\b{\varphi}\varphi)^{-1}\Finv$ and their conjugations. Remark
that the operator $\Finv$ in \eqref{w.18} is just the ordinary
contraction operator $\lc$ when acting on $A^{1,1}(X)$ of a complex
manifold $X$. Actually, one can obtain an equivalent expression of
\eqref{2.7}
 \begin{align*}
    &d(e^{\iota_{\varphi}|\iota_{\b{\varphi}}}(\alpha))\\=&
    e^{\iota_{\varphi}|\iota_{\b{\varphi}}}\left((\1-\b{\varphi}\varphi)^{-1}\Finv([\p,\iota_{\varphi}]+\b{\p})(\1-\b{\varphi}\varphi)\Finv\alpha
    +(\1-\varphi\b{\varphi})^{-1}\Finv([\db,\iota_{\b\varphi}]+{\p})(\1-\varphi\b{\varphi})\Finv\alpha\right).
 \end{align*}

From the original proof of the stability theorem sketched at the
beginning of this section, one knows that the system \eqref{w.18} of
obstruction equations indeed has a real solution $\alpha(t)\in
A^{1,1}(X_t)$ with $\alpha(0)=\omega(0)$. To get this solution by a
power series method, we will formulate an effective obstruction
system \eqref{reduced-equsys} of equations in Proposition \ref{2.9}
for \eqref{w.18}.

We will use the commutator formula repeatedly, which is originated
from \cite{T,To89} and whose various versions appeared in
\cite{F,BK,Li,LSY,C} and also \cite{LR,lry} for vector bundle valued
forms.
\begin{lemma}\label{aaaa} For $\phi, \psi\in
A^{0,1}(X,T^{1,0}_X)$ and $\alpha\in A^{*,*}(X)$ on a complex
manifold $X$,
\begin{equation}\label{f1}
[\phi,\psi]\lrcorner\alpha=-\p(\psi\lrcorner(\phi
\lrcorner\alpha))-\psi\lrcorner(\phi \lrcorner\p\alpha)
+\phi\lrcorner\p(\psi\lrcorner\alpha)+\psi
\lrcorner\p(\phi\lrcorner\alpha).
\end{equation}
\end{lemma}

There are several formulae to be established, whose proofs are given
in Appendix \ref{7proof}.
\begin{proposition}\label{7for}
Let $\phi \in A^{0,1}(T^{1,0}_X)$ and $\alpha \in A^{p,q}(X)$ on a
complex manifold $X$. Then we have:
\begin{enumerate}[$(1)$]
\item\quad \label{7.1}
$\phi \lc \overline{\phi} \lc \alpha - (\phi \lc \overline{\phi})
\lc \alpha = \overline{\phi} \lc \phi \lc \alpha - (\overline{\phi}
\lc \phi) \lc \alpha$.

\item\quad \label{7.3}
$[\phi,\phi] \lc \overline{\phi}\lc \alpha = 2\phi \lc \p (\phi \lc
\overline{\phi} \lc \alpha) - \phi \lc \phi \lc \p(
\overline{\phi}\lc \alpha)$.

\item\quad \label{7.5}
For $\psi \in A^{r,s}(T^{1,0}_{X})$, $\db (\psi \lc \alpha) = (\db
\psi) \lc \alpha + (-1)^{r+s+1} \psi \lc \db \alpha. $

In particular, if $\psi \in A^{1,0}(T^{1,0}_{X})$ or
$A^{0,1}(T^{1,0}_{X})$, then
$$\db (\psi \lc \alpha) = (\db \psi) \lc \alpha + \psi \lc \db
\alpha. $$

\item\quad \label{7.7}
$\overline{\phi} \lc \overline{\phi} \lc \phi \lc \alpha - \phi \lc
\overline{\phi} \lc \overline{\phi} \lc \alpha = 2 ( \overline{\phi}
\lc \phi \overline{\phi} \lc \alpha - \overline{\phi} \phi \lc
\overline{\phi} \lc \alpha )$.

In particular, if $\alpha \in A^{1,1}(X)$, then one has
\begin{equation*}
\overline{\phi} \lc \overline{\phi} \lc \phi \lc \alpha =
2(\overline{\phi} \lc \phi \overline{\phi} \lc \alpha),
\end{equation*}
since $\phi \lc \overline{\phi} \lc \overline{\phi} \lc \alpha = 0$
and $\overline{\phi} \phi \lc \overline{\phi} \lc \alpha = 0$.
\end{enumerate}
\end{proposition}

We first explain the homogenous notation for a power series to be
used here and henceforth. Assuming that $\alpha(t)$ is a power
series of (bundle-valued) $(p,q)$-forms, expanded as
  $$\alpha(t)=\sum_{k=0}^{\infty}\sum_{i+j=k}\alpha_{i,j}t^i\b{t}^j,$$
one uses the notation
\[ \begin{cases}
\alpha(t) = \sum^{\infty}_{k=0} \alpha_k, \\[4pt]
\alpha_k = \sum_{i+j=k} \alpha_{i,j}t^i \overline{t}^j, \\
\end{cases} \]
where $\alpha_k$ is the $k$-degree homogeneous part in the expansion
of $\alpha(t)$ and all $\alpha_{i,j}$ are smooth (bundle-valued)
$(p,q)$-forms on $X_0$ with $\alpha(0)=\alpha_{0,0}$. Similarly,
according to the expansion \eqref{phi-ps-pp}, one will also adopt this
notation to other terms related with $\varphi$, such as
\[ \lk \1 -\overline{\varphi}\varphi \rk^{\!-1} \lc
\varphi  = \sum_{k=1}^{\infty} \big( \lk \1 -\overline{\varphi}\varphi
\rk^{\!-1} \lc \varphi \big)_k, \] where $\big( \lk \1 -
\overline{\varphi}\varphi \rk^{\!-1} \lc \varphi \big)_k$ stands
for the $k$-degree homogeneous part of the power series  $\lk \1 -
\overline{\varphi}\varphi \rk^{\!-1} \lc \varphi$ in $t,\b t$.
Then we come to the crucial reduction:
\begin{proposition}\label{2.9} If the power series
  $$\omega(t)=\sum_{k=0}^{\infty}\sum_{i+j=k}\omega_{i,j}t^i\b{t}^j$$ of $(1,1)$-forms on $X_0$ is
a real formal solution of the system of equations
\begin{equation}\label{reduced-equsys}
    \begin{cases}
      \b{\p}\omega=\b{\p}(\b{\varphi}\varphi\l\omega-\varphi\l\b{\varphi}\l\omega)-\p(\varphi\l\omega),\\
      \p\omega=\p(\varphi\b{\varphi}\l\omega-\b{\varphi}\l\varphi\l\omega)-\b{\p}(\b{\varphi}\l\omega),
    \end{cases}
\end{equation}
up to the degree $N$, then $\db(\varphi(t)\lc \omega(t))_k=0$ for
each $0\leq k\leq N+1$ and thus $\omega(t)$ is also one real formal
solution of the system (\ref{w.18}) up to the degree $N$.
\end{proposition}

We will realize the importance of the $\db$-closedness
$\db(\varphi(t)\lc \omega(t))_{N+1}=0$, which fulfills
\eqref{db-condition} in Observation \ref{expslt} under the K\"ahler
condition and guarantees the existence of a real solution of
$\eqref{reduced-equsys}_{N+1}$, the equation \eqref{reduced-equsys}
at the $(N+1)$-th degree.

\begin{proof}
Comparing the power series expansion of the equations (\ref{w.18})
and (\ref{reduced-equsys}), we use induction on the degrees to
complete the proof.

Denote by $\omega_k$ the homogenous $k$-part of the power series
$\omega(t)$, i.e., $\omega_k=\sum_{i+j=k}\omega_{i,j}t^i\b{t}^j$.
Without danger of confusion, we will use $\varphi$ and $\omega$ to
denote $\varphi(t)$ and $\omega(t)$, respectively.

The case $N=0$ is trivial.
By induction, assuming that the proposition holds for the degrees
$\leq N-1$, we need to show the proposition for the degree $N$. That
is, if $(\ref{reduced-equsys})_k$ has a real solution and
$$\db(\varphi\lc \omega)_{k}=0$$
for the degrees $k\leq N$, then we will show that this solution is
also one for the system $(\ref{w.18})_N$ and  satisfies
$\db(\varphi\lc \omega)_{N+1}=0$. Without loss of generality, we
always assume that $N\geq 4$ since the lower-degree cases are also
obtained by the same formulation as follows.

Here is an important observation to be proved in Appendix \ref{app-b}.
\begin{observation}\label{closed}
$$\db(\varphi\lc \omega)_{k}=0,\qquad k\leq N+1.$$
\end{observation}

So by Lemma \ref{aaaa}, Proposition \ref{7for}.\eqref{7.1} and the
induction assumption, we have
\begin{align*}
  &\quad\varphi\l\p((\1-\b{\varphi}\varphi)\l\omega)\\
&=\varphi\l\p(\omega-\varphi\b{\varphi}\l\omega+\varphi\b{\varphi}\l\omega-\b{\varphi}\varphi\l\omega)\\
&=\varphi\l(-\p(\b{\varphi}\l\varphi\l\omega)-\b{\p}(\b{\varphi}\l\omega))+\varphi\l\p(\varphi\b{\varphi}\l\omega-\b{\varphi}\varphi\l\omega)\\
&=-\varphi\l\p(\varphi\l\b{\varphi}\l\omega)-\varphi\l\b{\p}(\b{\varphi}\l\omega)\\
&=-\varphi\l\p(\varphi\l\b{\varphi}\l\omega)-\b{\p}(\varphi\l\b{\varphi}\l\omega)+\b{\p}\varphi\l\b{\varphi}\l\omega\\
&=-\varphi\l\p(\varphi\l\b{\varphi}\l\omega)-\b{\p}(\varphi\l\b{\varphi}\l\omega)+\frac{1}{2}[\varphi,\varphi]\l\b{\varphi}\l\omega\\
&=-\b{\p}(\varphi\l\b{\varphi}\l\omega)-\frac{1}{2}\varphi\l\varphi\l\p(\b{\varphi}\l\omega)-\frac{1}{2}\p(\varphi\l\varphi\l\b{\varphi}\l\omega).
\end{align*}
Here the subscript $N$ is omitted. Using Proposition
\ref{7for}.\eqref{7.7} and Observation \ref{closed}, one has
$$
\left(\p(\varphi\l\b{\varphi}\varphi\l\omega)+\varphi\l\p(\omega-\b{\varphi}\varphi\l\omega)+\b{\p}(\varphi\l\b{\varphi}\l\omega)\right)_N=0,
$$
whose left-hand side is exactly the difference of the first
equations in $(\ref{reduced-equsys})_N$ and $(\ref{w.18})_N$.
Therefore, this real solution of the first equation in
$(\ref{reduced-equsys})_N$ is also that in $(\ref{w.18})_N$, and
similarly for the second equations in $(\ref{w.18})_N$ and
$(\ref{reduced-equsys})_N$.
\end{proof}

\subsection{Construction of power series}

For the resolution of the system \eqref{reduced-equsys}, we need
several more lemmas. As usual, the \emph{$\p\b{\p}$-lemma} refers
to: for every pure-type $d$-closed form on $X_0$, the properties of
$d$-exactness, $\p$-exactness, $\b{\p}$-exactness and
$\p\b{\p}$-exactness are equivalent.

\blemma\label{slvdb-1} Let $X$ be a complex manifold satisfying the
$\p\db$-lemma. Consider the system of equations:
\beq\label{conjugatequ-1}\left\{
\begin{aligned}
\p x &= \beta, \\
\db x &= \overline{\gamma}, \\
\end{aligned} \right. \eeq
where $\beta,\gamma$ are $(p+1,p)$-forms on $X$. The system of
equations \eqref{conjugatequ-1} has a solution if and only if the
following three statements hold: \bd
\item\label{c-1} $\db \beta + \p\overline{\gamma} = 0$,
\item\label{c-2} The $\p$-equation $\p x = \beta $ has a solution,
\item\label{c-3} The $\db$-equation $\db x = \overline{\gamma}$ has a solution.
\ed \elemma

\bproof If the system \eqref{conjugatequ-1} has a solution $\eta$,
 then both the $\p$- and $\db$-equations have solutions.
And it is easy to see that
\[ \db \beta + \p \overline{\gamma} = \db \p \eta + \p \db \eta = 0.\]
Conversely, let $\eta_1$ be a solution of the $\p$-equation,
$\eta_2$ a solution of the $\db$-equation with $\beta,\gamma$
satisfying $\db \beta + \p \overline{\gamma} = 0$, which yields that
$\p \eta_1 = \beta$.

We claim that there exists some $\tau \in A^{p,p-1}(X)$ such that
$\eta_2 + \db \tau$ satisfies the system \eqref{conjugatequ-1}. In
fact, it is obvious that $\eta_2 + \db \tau$ satisfies the second
equation of \eqref{conjugatequ-1}. As to the first one, we only need
to show that $\beta - \p \eta_2$ is $\p \db$-exact. It is easy to
check that
\[ \db( \beta - \p \eta_2) =\db \beta + \p \overline{\gamma} = 0.\]
And note that $\beta - \p \eta_2$ is $\p$-exact by the equality
\[ \beta - \p \eta_2= \p ( \eta_1 - \eta_2).\]
From the $\p\db$-Lemma, there exists some $\tau \in A^{p,p-1}(X)$
such that $\beta -\p \eta_2 = \p \db \tau$, equivalently saying that
$\eta_2 + \db \tau$ satisfies the first equation of
\eqref{conjugatequ-1}. Therefore, the claim is proved and $\eta_2 +
\db \tau$ is the solution of the system \eqref{conjugatequ-1}.
\eproof

\bcorollary\label{slvdb-2} Let $X$ be a complex manifold satisfying
the $\p\db$-lemma. The system of equations:
\beq\label{conjugatequ-2}\left\{
\begin{aligned}
\p x &= \beta,\\
\db x &= \overline{\beta}, \\
\end{aligned} \right. \eeq
where $\beta$ is a $(p+1,p)$-form on $X$, has a real solution if and
only if the following two statements hold: \bd \item $\db \beta + \p
\overline{\beta} = 0$,
\item The $\db$-equation
$\db x = \overline{\beta}$ has a solution. \ed \ecorollary

\bproof If $\eta_2$ is a solution of the $\db$-equation and $\db
\beta + \p \overline{\beta} = 0$, it is clear that
$\overline{\eta_2}$ satisfies the $\p$-equation $\p
\overline{\eta_2} = \beta$. Then Lemma \ref{slvdb-1} assures that
the system \eqref{conjugatequ-2} of equations admits a solution,
denoted by $\eta$. And $\frac{\eta + \overline{\eta}}{2}$ will be a
real solution of the system \eqref{conjugatequ-2}. \eproof

Since
$\beta$ and $\overline{\gamma}$ involved in this paper are mostly
$\db$-exact and $\p$-exact, respectively, such as \eqref{reduced-equsys}, we have:
\begin{observation}\label{expslt}
\emph{Let $\beta=\db \zeta$ and
$\overline{\gamma} = \p \xi$ for some suitable-type complex
differential forms $\zeta$ and $\xi$, respectively, which automatically fulfill the
condition \eqref{c-1} in Lemma \ref{slvdb-1}. The conditions
\eqref{c-2} and \eqref{c-3} rely on the equalities
\begin{equation}\label{db-condition}
\left\{
\begin{aligned}
\p \beta = \p (\db \zeta) =0, \\
\db \overline{\gamma} = \db (\p \xi)=0. \\
\end{aligned} \right.
\end{equation}
Then the $\p\db$-lemma will produce $\mu$ and $\nu$, satisfying the
equations \beq\label{mslmm} \left\{ \begin{aligned}
\p \db \mu = \db \zeta =\beta, \\
\db \p \nu = \p \xi = \overline{\gamma}. \\
\end{aligned} \right.  \eeq
The combined expression \beq\label{cbslt} \db \mu + \p \nu \eeq is
our choice for the solution of the system \eqref{conjugatequ-1}, which will be slightly modified to
\[\db \mu + \p \bar \mu\]
as the real solution of the system \eqref{conjugatequ-2}, when
$\beta$ happens to equal to $\gamma$.}
\end{observation}

Recall a useful fact that $\db^*\G_{\db}y$ is the unique
solution, minimizing the $L^2$-norms of all the solutions, of the
equation
\[ \db x=y \]
on a compact complex manifold if the equation admits one, where
$x,y$ are complex differential forms of pure types and the operator
$\G_{\db}$ denotes the corresponding Green's operator of the
$\db$-Laplacian $\square$. In the K\"ahler case, we choose \[ \db
\mu = \p^* \G_{\p} (\db \zeta)=\p^* \G_{\p} \beta
\quad\text{and}\quad \p \nu = \db^* \G_{\db} (\p \xi) =  \db^*
\G_{\db} \overline{\gamma},\] where $\G_{\p}$ and $\G_{\db}$
coincide, with a uniform symbol $\G$ used afterwards. Then an
explicit solution of the system \eqref{conjugatequ-1} can be taken
as
\[ x = \p^* \G \db \zeta + \db^* \G \p \xi
=\p^* \G \beta + \db^* \G \overline{\gamma}. \] When $\beta$ happens
to equal to $\gamma$, one takes the real solution of the system
\eqref{conjugatequ-2} as
\beq\label{ees} x = \p^* \G
\db \zeta + \db^* \G \p \overline{\zeta} =\p^* \G \beta + \db^* \G
\overline{\beta} \eeq
and accordingly, notices that the operator $\G$
is real in this case.

By these, one is able to obtain the main result of this section:
\begin{theorem}\label{main}
  The system of equations
  \begin{equation}\label{w.22}
    \begin{cases}
      d(e^{\iota_{\varphi}|\iota_{\b{\varphi}}}(\omega(t)))=0,\\
      \omega(t)=\o{\omega(t)},\\
      \omega(0)=\omega_0,
    \end{cases}
  \end{equation}
admits a smooth solution $\omega(t)\in A^{1,1}(X_0)$, where
$\omega_0$ is a K\"{a}hler metric on the complex manifold $X_0$.
Therefore, we can construct a smooth K\"{a}hler metric
$e^{\iota_{\varphi}|\iota_{\b{\varphi}}}(\omega(t))$ on ${X}_t$.
\end{theorem}
\begin{proof}
We are going to present such an explicit expression for the solution
of the obstruction equation \eqref{reduced-equsys}, whose existence
is assured by Proposition \ref{2.9} and the remarks after it, with
the initial metric $\omega(0)=\omega_0$. The first-order system of equations
$$\label{a-reduced-equsys}
    \begin{cases}
      \b{\p}\omega_1=-\p(\varphi_1\l\omega_0),\\
      \p\omega_1=-\b{\p}(\b{\varphi}_1\l\omega_0),\\
    \end{cases}
$$
admits an explicit real solution, as given by \eqref{ees},
\[ \omega_1 = \p\db^*\G(\varphi_1\l\omega_0)+\db\p^*\G(\b{\varphi}_1\l\omega_0). \]
By induction, we may assume that (\ref{reduced-equsys}) has an
explicit real solution $\omega_k$ for $k\leq N-1$. Based the
construction \eqref{ees} above, one gets a real solution of
the $N$-th order equation $\eqref{reduced-equsys}_N$
$$\label{sol-reduced-equsys}
\omega_N=(\b{\varphi}\varphi\l\omega-\varphi\l\b{\varphi}\l\omega)_N+\left(\p\db^*\G(\varphi\l\omega)+\db\p^*\G(\b{\varphi}\l\omega)\right)_{N},
$$
using Proposition \ref{7for}.\eqref{7.1}.

Hence, we complete the induction and get a formal solution
$\omega(t)$ of (\ref{w.22}) with explicit expressions. By the  H\"older
$C^{k,\alpha}$-convergence and regularity argument in Subsections
\ref{convergence} and \ref{regularity}, the formal power series
$\omega(t)$ constructed above is smooth and solves the system
\eqref{w.22} of equations.
\end{proof}

\subsection{H\"older convergence}\label{convergence}
Consider an important power series in deformation theory of complex
structures
\begin{equation}\label{cru-ps}
A(t)=\frac{\beta}{16\gamma}\sum_{m=1}^\infty\frac{(\gamma
t)^m}{m^2}:=\sum_{m=1}^\infty A_m t^m,
\end{equation}
where $\beta, \gamma$ are positive constants to be determined. The
power series \eqref{cru-ps} converges for $|t|<\frac{1}{\gamma}$ and
has a nice property:
\begin{equation}\label{p-prin}
A^i(t)\ll {\left(\frac{\beta}{\gamma}\right)}^{i-1}A(t).
\end{equation}
 See \cite[Lemma 3.6 and its Corollary in Chapter 2]{MK} for these
basic facts. Here we use the following notation: For the series with
real positive coefficients
$$a(t)=\sum_{m=1}^\infty a_m t^m, \quad b(t)=\sum_{m=1}^\infty b_m t^m,$$
say that \emph{$a(t)$ dominates $b(t)$}, written as $b(t)\ll a(t)$,
if $ b_m\leq a_m$. But for a power series of (bundle-valued) complex differential
forms
$$\eta(t)=\sum_{i+j\geq 0}^\infty \eta_{i,j} t^i\bar{t}^j,$$ the
notation
$$\|\eta(t)\|_{k, \alpha}\ll A(t)$$ means
$$\sum_{i+j=m}\|\eta_{i,j}\|_{k, \alpha}
\leq A_m$$ with the $C^{k, \alpha}$-norm $\|\cdot\|_{k, \alpha}$ as
defined on \cite[Page 159]{MK}. In this manner, for a power
series of Beltrami differential $\psi(t)=\sum_{i+j=1}^\infty \psi_{i,j}
t^i\b t^j$, the notation
$$\|\psi(t)\|_{k, \alpha}\ll A(t)$$ indicates
$\sum_{i+j=m}\|\psi_{i,j}\|_{k, \alpha}
\leq A_m$, and similarly for $\bar{\psi}(t)$. This notation is also used to compare two or more power
series of (bundle-valued) complex differential forms degree by degree, such as $\|\psi(t)\|_{k, \alpha}\ll \|\eta(t)\|_{k, \alpha}\cdot\|\rho(t)\|_{k, \alpha}$ for three such power series.

For any complex differential form $\phi$, we have two a priori
elliptic estimates
\begin{equation}\label{ee1}
\|\overline{\partial}^*\phi\|_{k-1, \alpha}\leq C_1\|\phi\|_{k,
\alpha}
\end{equation}
and
\begin{equation}\label{db-ee2}
\|\G\phi\|_{k, \alpha}\leq C_{k,\alpha}\|\phi\|_{k-2, \alpha},
\end{equation}
where $\G$ is the associated Green's operator to the operator $\db$, $k>1$, $C_1$ and $C_{k,\alpha}$
depend on only on $k$ and $\alpha$, not on $\phi$. (See \cite[Proposition
$2.3$ in Chapter $4$]{MK}.)

According to the proof of Theorem \ref{main}, for $r\geq 1$, one
real solution of the obstruction equation \eqref{reduced-equsys} is
\begin{equation}\label{sol-reduced-equsys}
\omega_r=(\b{\varphi}\varphi\l\omega-\varphi\l\b{\varphi}\l\omega)_r+\left(\p\db^*\G(\varphi\l\omega)+\db\p^*\G(\b{\varphi}\l\omega)\right)_{r},
\end{equation}
and obviously
\begin{equation}
\label{sol-reduced-equsys-r}
\omega^{(r)}=(\b{\varphi}\varphi\l\omega-\varphi\l\b{\varphi}\l\omega)^{(r)}+\left(\p\db^*\G(\varphi\l\omega)+\db\p^*\G(\b{\varphi}\l\omega)\right)^{(r)}
\end{equation}
are real complex differential forms. Here for a power series of (bundle-valued)
complex differential forms
$$\eta(t)=\sum_{i+j= 0}^\infty \eta_{i,j} t^i\bar{t}^j,$$ one
denotes by $\eta^{(r)}$ the summation
$$\sum_{i+j=1}^{r}
\eta_{i,j} t^i\bar{t}^j:=\eta_1+\cdots+\eta_r.$$

Recall the holomorphic family
\begin{equation}\label{phi-ps-p}
\varphi(t) = \sum_{i=1}^{\infty}\varphi_{i}t^i:=\sum_{j=1}^{\infty}\varphi_j(t),
\end{equation}
{for $|t|< \rho$ a small positive constant,} of Beltrami differentials representing Kuranishi family develop as in \eqref{phi-ps-0},\eqref{phi-ps}:
$$
 \varphi_1(t)=t\eta,
$$
where $\eta$ is a base of for $\mathbb{H}^{0,1}(X_0,T^{1,0}_{X_0})$, and for $i\geq 2$,
$$
  \varphi_i=\frac{1}{2}\db^*\G\sum_{j+l=i}[\varphi_j,\varphi_{l}].
$$
Then it satisfies a nice convergence property:
$$\|\varphi(t)\|_{k, \alpha}\ll A(t)$$
as given in the proof of \cite[Proposition 2.4 in Chapter $4$]{MK}.
In this proof, $\beta$ and $\frac{\beta}{\gamma}$ should satisfy
that
\begin{equation}\label{initial-bg}
\beta\geq b,\ \frac{\beta}{\gamma}\leq b_k,
\end{equation}
where the constants $b,b_k>0$ and $b_k$ depends on $k$.
Here we follow the idea of proving the
convergence of $\varphi(t)$ there to obtain that
$$\|\omega^{(r)}\|_{k, \alpha}\ll A(t),\ \text{for any $r\geq 1$},$$
which implies the desired convergence $\|\omega^{(+\infty)}\|_{k,
\alpha}\ll A(t)$ immediately.
 Assume that they are chosen so that $$\|\omega^{(r-2)}\|_{k, \alpha},\|\omega^{(r-1)}\|_{k, \alpha}\ll A(t).$$
By the expression \eqref{sol-reduced-equsys-r} and the two a priori
elliptic estimates \eqref{ee1} \eqref{db-ee2}, one has
\begin{align*}
\|\omega^{(r)}\|_{k, \alpha}
 &=\|(\b{\varphi}\varphi\l\omega^{(r-2)}-\varphi\l\b{\varphi}\l\omega^{(r-2)})^{(r)}+\left(\p\db^*\G(\varphi\l\omega^{(r-1)})+\db\p^*\G(\b{\varphi}\l\omega^{(r-1)})\right)^{(r)}\\
 &\quad +(\b{\varphi}\varphi\l\omega_0-\varphi\l\b{\varphi}\l\omega_0)^{(r)}+\left(\p\db^*\G(\varphi\l\omega_0)+\db\p^*\G(\b{\varphi}\l\omega_0)\right)^{(r)}\|_{k, \alpha}\\
 &\ll 2\|\b{\varphi}^{(r-2)}\|_{k, \alpha}\|\varphi^{(r-2)}\|_{k, \alpha}\|\omega^{(r-2)}\|_{k, \alpha}
 +2 C_1C_{k,\alpha}\|\varphi^{(r-1)}\|_{k, \alpha}\|\omega^{(r-1)}\|_{k, \alpha}\\
 &\quad + 2\|\b{\varphi}^{(r-1)}\|_{k, \alpha}\|\varphi^{(r-1)}\|_{k, \alpha}\|\omega_0\|_{k, \alpha}
 +2 C_1C_{k,\alpha}\|\varphi^{(r)}\|_{k, \alpha}\|\omega_0\|_{k,
 \alpha}.
\end{align*}
Then we use induction and \eqref{p-prin} to get:
\begin{align*}
\|\omega^{(r)}\|_{k, \alpha}
 &\ll 2(A(t)+C_1C_{k,\alpha})A^2(t)+2\|\omega_0\|_{k,\alpha}(A(t)+C_1C_{k,\alpha})A(t)\\
 &\ll
 2(A(t)+C_1C_{k,\alpha})\left(\frac{\beta}{\gamma}+\|\omega_0\|_{k,\alpha}\right)A(t).
\end{align*}
Hence, we may choose $\beta,\gamma,\|\omega_0\|_{k, \alpha}$ so that the
following inequalities hold:
$$2(A(t)+C_1C_{k,\alpha})\left(\frac{\beta}{\gamma}+\|\omega_0\|_{k,\alpha}\right)<1,$$
\eqref{initial-bg} and $\omega^{(1)},\omega^{(2)}\ll A(t)$ according
to the above formulation, which are obviously possible as long as
$t$ is small and we notice that the H\"{o}lder norm depends on the
choice of the local coordinate charts when defined a differential
manifold as pointed out on \cite[Page $275$]{k}. Therefore, for small $t$, $\omega(t)$ converges in the $C^{k,
\alpha}$-norm and thus its positivity follows.

\subsection{Regularity argument}\label{regularity}
In this subsection we proceed to the regularity argument for the
power series constructed as above since there is possibly no uniform
lower bound for the convergence radius obtained in the last
subsection in the $C^{k,\alpha}$-norm as $k$ converges to $+\infty$.
We resort to the elliptic operator, the $\db$-Laplacian
\begin{equation}\label{dbar-Lap}
\square=\db^*\db+\db\db^*.
\end{equation}
Here the dual operators are defined with respect to the fixed
original K\"ahler metric $\omega_0$. By the classical K\"ahler
identity, Hodge decomposition and the induced commutativity of the
associated operators, one has
\begin{align*}
\square\omega
 &=\square(\b{\varphi}\varphi\l\omega-\varphi\l\b{\varphi}\l\omega)
    +\square\p\db^*\G(\varphi\l\omega)+\square\overline{\p\db^*\G(\varphi\l\omega)}\\
 &=\square(\b{\varphi}\varphi\l\omega-\varphi\l\b{\varphi}\l\omega)
    +\p\db^*(\varphi\l\omega)+\overline{\p\db^*(\varphi\l\omega)}
\end{align*}
according to the solution \eqref{sol-reduced-equsys}.

Consequently, $\omega$ is a solution of the two-order partial
differential equation
\begin{equation}\label{ellip-eq}
\square\omega-\square(\b{\varphi}\varphi\l\omega-\varphi\l\b{\varphi}\l\omega)
    -\p\db^*(\varphi\l\omega)-\db\p^*(\b\varphi\l\omega)=0.
\end{equation}

Similarly to the argument on \cite[Page 281]{k}, writing out the
last three terms in the left-hand side of \eqref{ellip-eq} locally,
we can find that the expressions of their principal parts (i.e., the
highest-order terms) contain factors from $\b{\varphi}\varphi$,
$\varphi$, or $\b{\varphi}$. Since $\varphi(t)\rightarrow 0$ as
$t\rightarrow 0$, taking sufficiently small $\epsilon$-disk
$\Delta_{\epsilon}\subseteq \mathbb{C}$, we can assume that the equation
\eqref{ellip-eq} is a linear elliptic partial differential equation
of $\om$ on $X_0$ when noticing the ellipticity of an operator only
concerns about its principal part. Thus, the interior estimates \cite{dn} for
elliptic systems of partial differential equations give
rise to the desired regularity of $\omega(t)$, i.e., $\omega(t)$ is
smooth on $X_0$ for each $t$ smaller than a uniform upper bound to
be obtained similarly to \cite[Appendix.\S 8]{k}. Then $\omega(t)$ can
be regarded as a real analytic family of $(1,1)$-forms in $t$ and it is smooth on $t$
by \cite[Proposition 2.2.3]{kp}.

\section{Stability of balanced structures with $(n-1,n)$-th mild
$\p\b{\p}$-lemma}\label{balanced}

In this section we will discuss the local stability problem of
balanced structures, satisfying various $\p\db$-lemmata. Recall that
a \emph{balanced metric} $\omega$ on an $n$-dimensional complex
manifold is a real positive $(1,1)$-form, satisfying that
$$d(\omega^{n-1})=0,$$
and a complex manifold is called \emph{balanced} if there
exists a balanced metric $\omega$ on it. Note that the existence of
a balanced metric $\omega$ is equivalent to that of a $d$-closed real
positive $(n-1,n-1)$-form $\Omega$ with the relation
$\Omega=\omega^{n-1}$ (see \cite[(4.8)]{Mich}).

\subsection{The $(n-1,n)$-th mild $\p\b{\p}$-lemma and
examples}\label{subs-mild} We are going to study a new kind of
\lq\lq$\p\db$-lemma", its relations with various analogous conditions in
the literature and its examples involved.

Let $X$ be a compact complex manifold of (complex) dimension $n$
with the following commutative diagram
$$\xymatrix{      & H^{n-1,n}_{\p}(X) \ar[dr]^{\iota^{n-1,n}_{\p,A}} &            \\
 H^{n-1,n}_{BC}(X) \ar[ur]^{\iota^{n-1,n}_{BC,\p}} \ar[dr]_{\iota^{n-1,n}_{BC,\db}} \ar[rr]^{\iota^{n-1,n}_{BC,A}} & &  H^{n-1,n}_{A}(X)  \\
                & H^{n-1,n}_{\db}(X) \ar[ur]_{\iota^{n-1,n}_{\db,A}}. &          }$$
Bott--Chern and Aeppli cohomology groups of $X$ are defined as
$$H^{\bullet,\bullet}_{\mathrm{BC}}(X)=:\frac{\ker \p\cap \ker\db}{\im \p\db}\quad
\text{and}\quad H^{\bullet,\bullet}_{\mathrm{A}}(X)=:\frac{\ker
\p\db}{\im \p+\im\db},$$ respectively, while
$H^{\bullet,\bullet}_{\p}(X)$ is defined similarly. Note that the
inequalities
\[\dim_{\mathbb{C}}H^{n-1,n}_{A}(X) \leq \dim_{\mathbb{C}}H^{n-1,n}_{\p}(X) \leq \dim_{\mathbb{C}}H^{n-1,n}_{BC}(X)\]
hold on any compact complex manifold \cite[Corollary 3.3]{au}.

\begin{definition}\label{mild}
The compact complex manifold $X$ satisfies \emph{the $(n-1,n)$-th
mild $\p\b{\p}$-lemma}, if the mapping
$\iota_{BC,\p}^{n-1,n}:H^{n-1,n}_{BC}(X) \rightarrow
H^{n-1,n}_{\p}(X)$, induced by the identity map, is injective.
Equivalently, for any $(n-2,n)$-complex differential form $\xi$,
there exists an $(n-2,n-1)$-form $\theta$ such that
\[ \p \db \theta = \p \xi. \]
Since $\iota_{BC,\p}^{n-1,n}:H^{n-1,n}_{BC}(X) \rightarrow
H^{n-1,n}_{\p}(X)$ is always surjective, the following conditions
are equivalent:
\[ \begin{aligned}
\text{the $(n-1,n)$-th mild $\p\b{\p}$-lemma}
&\Longleftrightarrow \iota_{BC,\p}^{n-1,n}:H^{n-1,n}_{BC}(X) \rightarrow H^{n-1,n}_{\p}(X)\ \text{an isomorphism} \\
&\Longleftrightarrow \dim_{\mathbb{C}} H^{n-1,n}_{BC}(X) =
\dim_{\mathbb{C}} H^{n-1,n}_{\p}(X).
\end{aligned}\]
\end{definition}
Let us give an example of a non-$\p \db$-manifold which satisfies the (1,2)-mild $\p\db$-lemma.
\begin{example}\emph{(\cite[Page 90]{N} and \cite[Example 1]{K})}\label{NdbMdb}
Let $X$ be the manifold in the case $(ii)$ of the completely-solvable
Nakamura manifold as given in \cite[Example 3.1]{AK}.
Then the manifold $X$ satisfies the $(1,2)$-mild $\p\db$-lemma, but not the $\p\db$-lemma.
\end{example}
\begin{proof}
It is shown in \cite[Table 5 in Appendix A]{AK} and \cite[the case B in Example 1]{K} that
the bases of $H^{2,1}_{BC}(X)$ and $H^{2,1}_{\db}(X)$ can be illustrated as follows:
\[ \begin{aligned}
H^{2,1}_{BC}(X) &= \langle [dz_{12\bar{3}}]_{BC}, [e^{-2z_1}dz_{12\bar{2}}]_{BC},
[e^{2z_1}dz_{13\bar{3}}]_{BC}, [dz_{12\bar{3}}]_{BC},[dz_{13\bar{2}}]_{BC} \rangle,\\
H^{2,1}_{\db}(X) &= \langle [dz_{12\bar{3}}]_{\db}, [e^{-2z_1}dz_{12\bar{2}}]_{\db},
[e^{2z_1}dz_{13\bar{3}}]_{\db}, [dz_{12\bar{3}}]_{\db},[dz_{13\bar{2}}]_{\db} \rangle,
\end{aligned}\]
which indicates that $\iota_{BC,\db}^{2,1}$ is actually an isomorphism.
It is obvious from \cite[Table 6 in Appendix A]{AK} that
$\dim_{\mathbb{C}}H^{1,1}_{BC}(X)=3$ and $\dim_{\mathbb{C}}H^{1,1}_{\db}(X)=5$,
which implies that the $\p\db$-lemma doesn't hold on $X$.
\end{proof}

There are three more similar conditions in relevance with the local
stability of balanced structures. \emph{The $(n-1,n)$-th weak
$\p\db$-lemma} on the compact complex manifold $X$, introduced by
Fu--Yau \cite{FY}, says that if for any real $(n-1,n-1)$-form $\psi$
such that $\db \psi$ is $\p$-exact, there exists an $(n-2,n-1)$-form
$\theta$, satisfying
\[ \p \db \theta = \db \psi. \]
And \emph{the $(n-1,n)$-th strong $\p\db$-lemma}, proposed by
Angella--Ugarte \cite{au}, states that the mapping
$\iota^{n-1,n}_{BC,A}: H^{n-1,n}_{BC}(X) \rightarrow
H^{n-1,n}_{A}(X)$, induced by the identity map, is injective, which
is equivalent to that for any $\p$-closed $(n-1,n)$-form $\Gamma$ of
the type $\Gamma=\p\xi+\db \psi$, there exists an $(n-2,n-1)$-form
$\theta$ such that
\[ \p \db \theta = \Gamma. \]
Angella--Ugarte \cite[Theorem 3.1]{au} show that the $(n-1,n)$-th
strong $\p\db$-lemma amounts to the sGG condition, carefully studied in \cite{PU},
and the vanishing of the first $\p\db$-degree $\Delta^1(X)$, introduced in \cite{at}, with the deformation
openness of the $(n-1,n)$-th strong $\p\db$-lemma proved in
\cite[Proposition 4.8]{au}. They also show in \cite[Corollary
3.3]{au} the equivalence:
\[ \begin{aligned}
\text{the $(n-1,n)$-th strong $\p\db$-lemma}
&\Longleftrightarrow \dim_{\mathbb{C}}H^{n-1,n}_{BC}(X) = \dim_{\mathbb{C}}H^{n-1,n}_{A}(X) \\
&\Longleftrightarrow \dim_{\mathbb{C}}H^{0,1}_{BC}(X) = \dim_{\mathbb{C}}H^{0,1}_{A}(X) \\
&\Longleftrightarrow b_1=2\dim_{\mathbb{C}}H^{0,1}_{A}(X).
\end{aligned}\]
Besides, the condition that the induced mapping
$\iota^{n-1,n}_{BC,\db}: H^{n-1,n}_{BC}(X) \rightarrow
H^{n-1,n}_{\db}(X)$ by the identity map is injective, is presented
by Angella--Ugarte \cite{AU} to study local conformal balanced
structures and global ones, which we may call \emph{the $(n-1,n)$-th
dual mild $\p\db$-lemma}.

After a simple check, we have the following observation:
\begin{observation}\label{s-eq}
\emph{The compact complex manifold $X$ satisfies the $(n-1,n)$-th
strong $\p\db$-lemma if and only if both of the mild one and the
dual mild one hold on $X$.}
\end{observation}
And the mild one and the dual mild one both imply the weak one. All
the four \lq\lq$\p\db$-lemmata" hold if the compact complex manifold $X$
satisfies the standard $\p\db$-lemma.

We refer the readers to \cite{R,UV,LUV,AK} as the background
materials on the theory of nilmanifolds and solvmanifolds to be
focused on in the rest of this subsection.

Recall that a \emph{nilmanifold $M$ with left-invariant complex
structure} is a compact quotient of a simply-connected nilpotent Lie
group $G$ of real even dimension by a lattice $\Gamma$ of maximal
rank, whose Lie algebra $\mathfrak{g}$ admits an integrable complex
structure $J$. It is clear that the \emph{invariant complex
structure} $J$ on $G$ descends to $M$ in a natural way and it is
given by an endomorphism $J:\mathfrak{g}\rightarrow \mathfrak{g}$ of
the Lie algebra $\mathfrak{g}$ such that $J^2=-\1$, satisfying the
\lq\lq Nijenhuis condition"
 $$[JX,JY]=J[JX,Y]+J[X,JY]+[X,Y],$$
for any $X,Y\in \mathfrak{g}$.

Let $\mathfrak{g}_\mathbb{C}$ be the complexification of
$\mathfrak{g}$ and $\mathfrak{g}_\mathbb{C}^*$ its dual.
We denote by $\mathfrak{g}^{1,0}$ and $\mathfrak{g}^{0,1}$ the
eigenspaces corresponding to the eigenvalues $\pm\sqrt{-1}$ of $J$
as an endomorphism of $\mathfrak{g}_\mathbb{C}^*$, respectively. The
decomposition
$$\mathfrak{g}_\mathbb{C}^*=\mathfrak{g}^{1,0}\bigoplus\mathfrak{g}^{0,1}$$
gives rise to a natural bigraduation on the complexified exterior
algebra
$$\bigwedge^*\mathfrak{g}_\mathbb{C}^*=\bigoplus_{p,q}\bigwedge^{p,q}
\mathfrak{g}^* =\bigoplus_{p,q}\bigwedge^p
\mathfrak{g}^{1,0}\bigoplus\bigwedge^q\mathfrak{g}^{0,1}.$$ We will
still use the Chevalley--Eilenberg differential $d$ of the Lie
algebra to denote its extension to the complexified exterior
algebra, i.e.,
$d:\bigwedge^*\mathfrak{g}_\mathbb{C}^*\rightarrow\bigwedge^{*+1}\mathfrak{g}_\mathbb{C}^*.$
It is well known that the endomorphism $J$ is a complex structure if
and only if \[d\mathfrak{g}^{1,0}\subset\bigwedge^{2,0}
\mathfrak{g}^*\bigoplus\bigwedge^{1,1} \mathfrak{g}^*.\] As for
nilpotent Lie algebras $\mathfrak{g}$, Salamon \cite{Sa} proves the
equivalence: $J$ is a \emph{complex structure} on $\mathfrak{g}$ if
and only if $\mathfrak{g}^{1,0}$ has a basis $\{\varpi^i\}_{i=1}^n$
such that $d\varpi^1=0$ and
$$d\varpi^i\in \mathcal{I}(\varpi^1,\cdots,\varpi^{i-1}),\ \text{for $i=2,\cdots,n$},$$
where $\mathcal{I}(\varpi^1,\cdots,\varpi^{i-1})$ is the ideal in
$\wedge^*\mathfrak{g}_\mathbb{C}^*$ generated by
$\{\varpi^1,\cdots,\varpi^{i-1}\}$. The work \cite{CFGU} shows that
a complex structure $J$ is\emph{ nilpotent} if and only if
$\mathfrak{g}^{1,0}$ admits a basis $\{\varpi^i\}_{i=1}^n$ with
$d\varpi^1=0$ and
$$d\varpi^i\in \wedge^2\langle\varpi^1,\cdots,\varpi^{i-1},\bar{\varpi}^1,\cdots,\bar{\varpi}^{i-1})\rangle,\ \text{for $i=2,\cdots,n$};$$
otherwise $J$ is \emph{non-nilpotent}. \emph{Abelian complex
structures} satisfy additionally that $d\mathfrak{g}^{1,0}\subset
\bigwedge^{1,1} \mathfrak{g}^*$. A nilpotent complex structure is
\emph{complex parallelizable} if and only if
$d\mathfrak{g}^{1,0}\subset \bigwedge^{2,0} \mathfrak{g}^*$.

Inspired by Ugarte--Villacampa \cite[Proposition 3.4 and Corollary
3.5]{UV}, one has:
\begin{proposition}\label{ms}
Let $M = \Gamma \backslash G$ be a $2n$-dimensional nilmanifold
endowed with an invariant complex structure $J$ (i.e., complex
$n$-dimensional $M$) and $\mathfrak{g}$ the Lie algebra of $G$. If
$(\mathfrak{g},J)$ satisfies the $(n-1,n)$-th mild $\p\db$-lemma and
$H^{n-2,n}_{\db}(M,J) \cong H^{n-2,n}_{\db}(\mathfrak{g},J)$, then
$(M,J)$ satisfies the $(n-1,n)$-th mild $\p\db$-lemma.
\end{proposition}

\begin{proof}
Let $\xi$ be an $(n-2,n)$-complex differential form on $M$. From the
isomorphism $H^{n-2,n}_{\db}(M,J) \cong
H^{n-2,n}_{\db}(\mathfrak{g},J)$ and \cite[Remark 3.3]{UV}, there
exists an $(n-2,n-1)$-form $\theta$ on $M$ such that
\[ \xi = \xi_{\nu} + \db \theta, \]
where $\xi_{\nu}$ denotes the image of $\xi$ under the
symmetrization process $A^{p,q}(M) \rightarrow
\bigwedge^{p,q}(\mathfrak{g^*})$. The $(n-1,n)$-th mild
$\p\db$-lemma on $(\mathfrak{g},J)$ implies that
\[ \p \db \tilde{\theta} = \p \xi_{\nu} \]
for some $\tilde{\theta} \in \bigwedge^{n-2,n-1}(\mathfrak{g^*})$.
It follows that
\[ \p \xi = \p \db (\tilde{\theta} + \theta). \]
\end{proof}

\begin{corollary}\label{abelian}
Let $M = \Gamma \backslash G$ be a $2n$-dimensional nilmanifold
endowed with an invariant abelian complex structure $J$. Then $M$
satisfies the $(n-1,n)$-th mild $\p\db$-lemma.
\end{corollary}

\begin{proof}
It is known that the isomorphism $H^{p,q}_{\db}(M,J) \cong
H^{p,q}_{\db}(\mathfrak{g},J)$ holds for any $(p,q)$ with the
abelian complex structure $J$ (cf. \cite[Remark 4]{CF},
\cite[Theorem 1.10]{R} and the discussion ahead of \cite[Remark
3.3]{UV}). And the abelian complex structure $J$ implies that $\p
\big( \bigwedge^{n-2,n}(\mathfrak{g}^*)\big)=0$.
\end{proof}

Similarly with \cite[Proposition 3.2]{UV}, we have
\begin{proposition}\label{notms}
Let $M = \Gamma \backslash G$ be a $2n$-dimensional nilmanifold
endowed with an invariant complex structure $J$ with $\mathfrak{g}$
the Lie algebra of $G$. If $(\mathfrak{g},J)$ does not satisfy the
$(n-1,n)$-th mild $\p\db$-lemma, then $(M,J)$ does not satisfy the
$(n-1,n)$-th mild $\p\db$-lemma either.
\end{proposition}

\begin{proof}
Suppose that the $(n-1,n)$-th mild $\p\db$-lemma holds on $(M,J)$,
that is, for any $(n-2,n)$-form $\xi$ on $M$, there exists an
$(n-2,n-1)$-form $\theta$ such that
\[ \p \db \theta = \p \xi. \]
Using the symmetrization process (cf. the proof of \cite[Proposition
3.2]{UV}) on the both sides, we have
\[ \p \db \theta_{\nu} = \p \xi_{\nu}, \]
which contradicts the assumption on the Lie algebra level
$(\mathfrak{g},J)$. Here the equalities
\[ (\p \alpha)_{\nu} = \p \alpha_{\nu}\quad \text{and} \quad (\db \alpha)_{\nu} = \db \alpha_{\nu}\]
for any $\alpha \in A^{p,q}(M)$ are used as in the proof of
\cite[Proposition 3.2]{UV}.
\end{proof}

The following result of the $(n-1,n)$-th dual mild $\p\db$-lemma is
almost the same as the mild one in Propositions \ref{ms} and
\ref{notms}, for which we omit the proofs.

\begin{proposition}\label{dms}
Let $M = \Gamma \backslash G$ be a $2n$-dimensional nilmanifold
endowed with an invariant complex structure $J$ and $\mathfrak{g}$
the Lie algebra of $G$. If $(\mathfrak{g},J)$ satisfies the
$(n-1,n)$-th dual mild $\p\db$-lemma and $H^{n-1,n}_{BC}(M,J) \cong
H^{n-1,n}_{BC}(\mathfrak{g},J)$, then $(M,J)$ satisfies the
$(n-1,n)$-th dual mild $\p\db$-lemma; Similarly if
$(\mathfrak{g},J)$ does not satisfy the $(n-1,n)$-th dual mild
$\p\db$-lemma, then $(M,J)$ does not satisfy the $(n-1,n)$-th dual
mild $\p\db$-lemma.
\end{proposition}

\begin{example}\label{ex-dms-notms}
\emph{The complex structure in the category $(i)$ of
\cite[Proposition 2.3]{UV}, i.e., the complex parallelizable case of
complex dimension $3$, satisfies the $(2,3)$-th weak $\p\db$-lemma
and the dual mild one, but does not satisfy the mild one. The
Iwasawa manifold belongs to the category $(i)$.}
\end{example}

\begin{proof}
Let $J$ be the complex structure in the category $(i)$, which
satisfies the $(2,3)$-th weak $\p\db$-lemma by the proof of
\cite[Proposition 3.6]{UV}. It is easy to check that, on the Lie
algebra level,
\[\p \big( \bigwedge^{1,3}(\mathfrak{g}^*)\big) = \langle \omega^{12\b{1}\b{2}\b{3}} \rangle,\quad
\db \big( \bigwedge^{1,2}(\mathfrak{g}^*) \big)=0\quad \text{and}
\quad \db \big( \bigwedge^{2,2}(\mathfrak{g}^*) \big)=0.\] However,
\[ \p \omega^{3\b{1}\b{2}\b{3}} = \omega^{12\b{1}\b{2}\b{3}} \notin
\p \db \big( \bigwedge^{1,2}(\mathfrak{g}^*) \big)=0,\] and thus
$(2,3)$-th mild $\p\db$-lemma does not hold on the nilmanifold, from
Proposition \ref{notms}. Here $\omega^{3\b{1}\b{2}\b{3}}$ denotes
$\omega^3 \wedge \overline{\omega^1} \wedge \overline{\omega^2}
\wedge \overline{\omega^3}$ as the notation used in \cite{UV}. Also,
it is clear that $(\mathfrak{g},J)$ satisfies the $(2,3)$-th dual
mild $\p\db$-lemma, with $H^{p,q}_{BC}(M,J) \cong
H^{p,q}_{BC}(\mathfrak{g},J)$ for any $p,q$ assured by \cite[Theorem
3.8]{A}. Therefore, the $(2,3)$-th dual mild $\p\db$-lemma holds on
the nilmanifold by Proposition \ref{dms}.
\end{proof}

An invariant balanced Hermitian structure $F$ on a nilmanifold $M$
is the one, coming from a balanced Hermitian structure on the Lie
algebra $\mathfrak{g}^*$.

\begin{proposition}\label{ex-3dim}
Let $M$ be a $6$-dimensional nilmanifold endowed with an invariant
balanced Hermitian structure $(J,F)$. Then $(M,J)$ satisfies the
$(2,3)$-th mild $\p\db$-lemma if and only if $J$ is abelian or
non-nilpotent, and $(M,J)$ satisfies the $(2,3)$-th dual mild
$\p\db$-lemma if and only if $J$ is complex parallelizable.
\end{proposition}

\begin{proof}
It is known, from \cite[Proposition 3.6]{UV}, $(M,J)$ satisfies the
$(2,3)$-th weak $\p\db$-lemma if and only if $J$ is abelian, complex
parallelizable or non-nilpotent. Example \ref{ex-dms-notms} shows
that the complex parallelizable case satisfies the $(2,3)$-th dual
mild $\p\db$-lemma but does not satisfy the mild one. Corollary
\ref{abelian} and \cite[Proposition 2.9]{AU} say that a
$2n$-dimensional nilmanifold endowed with an invariant abelian
complex structure satisfies the $(n-1,n)$-th mild $\p\db$-lemma but
never satisfies the $(n-1,n)$-th dual mild $\p\db$-lemma, especially
for the situation of the $6$-dimension (i.e., complex dimension $3$)
here. Hence, the only left to check is the non-nilpotent case.

As shown in the proof of \cite[Proposition 3.6]{UV}, the equality
holds
\[ \p \db \big( \bigwedge^{1,2}(\mathfrak{g}^{*}) \big) = \langle \omega^{13\b{1}\b{2}\b{3}}\rangle =
\p \big( \bigwedge^{1,3}(\mathfrak{g}^*) \big),\] according to the
category $(iii)$ of \cite[Proposition 2.3]{UV}, and the natural
inclusion
\[\big(\bigwedge^{*,*}(\mathfrak{g}^*),\db\big)\hookrightarrow \big(A^{*,*}(M),\db\big)\] induces an isomorphism
in the $\db$-cohomology in this case. Therefore, $(M,J)$ with the
non-nilpotent complex structure $J$ satisfies the $(2,3)$-th mild
$\p\db$-lemma by Proposition \ref{ms}. Meanwhile, it should be noted
that
\[ \db \omega^{13\b{2}\b{3}} = \pm \sqrt{-1} \omega^{12\b{1}\b{2}\b{3}}\quad \text{and} \quad
\p \omega^{12\b{1}\b{2}\b{3}} =0,\] according to the category
$(iii)$ of \cite[Proposition 2.3]{UV}. Hence, we have
\[ \text{the $\p$-closed $(2,3)$-form}\  \db \omega^{13\b{2}\b{3}} = \pm \sqrt{-1} \omega^{12\b{1}\b{2}\b{3}}
\notin \p \db \big( \bigwedge^{1,2}(\mathfrak{g}^{*}) \big) =
\langle \omega^{13\b{1}\b{2}\b{3}}\rangle.\] Then Proposition
\ref{dms} tells us that $(M,J)$ with $J$ non-nilpotent does not
satisfy the $(2,3)$-th dual mild $\p\db$-lemma.
\end{proof}

Explicit examples of different complex structures in Proposition
\ref{ex-3dim} have been provided in \cite{UV}.
Directly from Observation \ref{s-eq} and Proposition
\ref{ex-3dim}, one has:
\begin{corollary}
Let $M$ be a $6$-dimensional nilmanifold endowed with an invariant
balanced Hermitian structure $(J,F)$. Then the $(2,3)$-th strong
$\p\db$-lemma does not hold on $M$ except for a torus, i.e., the mild one and dual mild
one never hold simultaneously on $M$ except the torus case. Especially, the $(n-1,n)$-th
mild $\p\db$-lemma and the dual mild one are unrelated.
\end{corollary}

\begin{remark}
The $(2,3)$-th strong $\p\db$-lemma can hold on the
completely-solvable balanced Nakamura manifolds, with one concrete
example given in \cite[Example 4.10]{au}, which is of complex
dimension $3$ and satisfies $\dim_{\mathbb{C}} H^{0,1}_{BC}(X) =
\dim_{\mathbb{C}}H^{0,1}_{A}(X) =1$.
\end{remark}

Let $\pi:\mc{X}\to B$ be a differentiable family of compact
 $n$-dimensional complex manifolds  with the reference fiber
$\pi^{-1}(t_0)=X_0$ and the general fibers $X_t:=\pi^{-1}(t)$. Here
$B$ denotes a sufficiently small domain in $\mathbb{R}^k$. Fu--Yau \cite[Theorem 6]{FY} show that the balanced
structure is deformation open, assuming that the $(n-1,n)$-th weak
$\p\db$-lemma holds on the general fibers $X_t$ for $t \neq 0$.
Angella--Ugarte \cite[Theorem 4.9]{au} prove that if $X_0$ admits a
locally conformal balanced metric and satisfies the $(n-1,n)$-th
strong $\p\db$-lemma, then $X_t$ is balanced for $t$ small. Our main
result in this section, whose proof is postponed to the next
subsection, is
\begin{theorem}\label{blc-inv}
Let $X_0$ be a compact balanced manifold of complex dimension $n$,
satisfying the $(n-1,n)$-th mild $\p\db$-lemma. Then $X_t$ also
admits a balanced metric for $t$ small.
\end{theorem}

It is well known from \cite{ab} that small deformation of the
Iwasawa manifold, which satisfies the $(2,3)$-th weak $\p\db$-lemma
but does not satisfy the mild one from Example \ref{ex-dms-notms},
may not be balanced. Thus, the condition \lq\lq$(n-1,n)$-th mild
$\p\db$-lemma" in Theorem \ref{blc-inv} can't be replaced by the
weak one. It is an obvious generalization of Wu's result \cite[Theorem 5.13]{w}
that the balanced condition is preserved under small deformation if
the reference  fiber satisfies the $\p\db$-lemma. Based on Corollary
\ref{abelian} and Proposition \ref{ex-3dim}, one obtains:
\begin{corollary}\label{abelian-inv}
Let $M$ be a $2n$-dimensional nilmanifold endowed with an invariant
abelian balanced Hermitian structure. Then small deformation of $M$
is also balanced. Moreover, in the case of $6$-dimension,
this result still holds when $M$ is endowed with the non-nilpotent
balanced Hermitian structure.
\end{corollary}

Examples and results such as \cite[Proposition 4.4, Remark 4.6,
Remark 4.7 and Example 4.10]{au} and \cite[Corollary 8 and Corollary
9]{FY} become consequences of Theorem \ref{blc-inv} and Corollary \ref{abelian-inv}. And it is
interesting to note that $6$-dimensional nilmanifolds endowed with
an invariant abelian or non-nilpotent balanced Hermitian structure
provide solutions of the Strominger system with respect to the
Bismut connection or the Chern connection in the anomaly
cancellation condition, respectively. See \cite[Section
5]{UV} and \cite[Section 4]{UV2} for more details.

\begin{example}[{\cite[Example 3.7]{UV}}]\label{not-fy}
\emph{Ugarte--Villacampa constructed an explicit family of
nilmanifolds with invariant complex structures $I_{\lambda}$ for
$\lambda\in[0,1)$ (of complex dimension $3$), with the fixed
underlying manifold the Iwasawa manifold. The complex structure of
the reference  fiber $I_{0}$ is abelian and admits an invariant
balanced metric, satisfying the $(2,3)$-th mild $\p\db$-lemma by
Proposition \ref{ex-3dim}. The complex structures of $I_{\lambda}$
for $\lambda \neq 0$ are nilpotent from \cite[Cororllary 2]{CFP},
but neither complex-parallelizable nor abelian. And thus they do not
satisfy the $(2,3)$-th weak $\p\db$-lemma by \cite[Proposition
3.6]{UV}. However, the nilmanifolds $I_{\lambda}$ for $\lambda \neq
0$ admit balanced metrics.}
\end{example}

The example above proves that neither the $(n-1,n)$-th weak
$\p\db$-lemma nor the mild one is deformation open. And it shows
that the condition in \cite[Theorem 6]{FY} is not a necessary one
for the deformation openness of balanced structures as mentioned in
\cite[the discussion ahead of Example 3.7]{UV}. Fortunately, Corollary
\ref{abelian-inv} can be applied to this example. See also
\cite[Remark 4.7]{au}, where Corollary \ref{abelian-inv} can also be
applied.

Meanwhile, from Corollary \ref{abelian} and \cite[Proposition
2.9]{AU}, a $2n$-dimensional nilmanifold endowed with an invariant
abelian complex structure satisfies the $(n-1,n)$-th mild
$\p\db$-lemma but never satisfies the $(n-1,n)$-th dual mild
$\p\db$-lemma. It implies that the deformation openness of balanced
structures with the reference  fiber a $2n$-dimensional nilmanifold
endowed with an invariant abelian balanced Hermitian structure can
be obtained by Corollary \ref{abelian-inv}, but it can't be derived
from \cite[Theorem 4.9]{au}.

Besides, it is known that the deformation invariance of the
dimensions of the $(n-1,n-1)$-th Bott--Chern group
$H^{n-1,n-1}_{BC}(X_t)$ can assure the deformation openness of
balanced structures as shown in \cite[Proposition 4.1]{au}. See also
Proposition \ref{d-ext}, which is a kind of generalization of this
result. However, \cite[Example 4.10]{au} shows that small
deformation of a completely-solvable Nakamura threefold, which is
balanced and satisfies the $(2,3)$-th strong $\p\db$-lemma, is also
balanced. The $(2,2)$-th Bott--Chern number varies along this
deformation. Fortunately, Theorem \ref{blc-inv} is applicable to
this case and also possibly to some cases with deformation variance
of $(n-1,n-1)$-th Bott--Chern number.

Finally, from the perspective of Theorem \ref{blc-inv}, we may have
a clear picture of Angella--Ugarte's result \cite[Theorem 4.9]{au},
which states that if $X_0$ admits a locally conformal balanced
metric and satisfies the $(n-1,n)$-th strong $\p\db$-lemma, then
$X_t$ is balanced for $t$ small. Actually, the $(n-1,n)$-th strong
$\p\db$-lemma decomposes into the mild one and the dual mild one,
according to Observation \ref{s-eq}. A locally conformal balanced
metric can be transformed into a balanced one by the $(n-1,n)$-th
dual mild $\p\db$-lemma, from \cite[Theorem 2.5]{AU}. Then the
$(n-1,n)$-th mild $\p\db$-lemma assures that the deformation
openness of balanced structures starts from the transformed balanced
metric on the reference fiber, thanks to Theorem \ref{blc-inv}.

\subsection{Proof for stability of balanced
structures with mild $\p\db$-lemma}\label{proof-bal} In this subsection, we will prove the
local stability Theorem \ref{blc-inv} of balanced structures with
the balanced reference fiber $X_0$ with the $(n-1,n)$-th mild
$\p\b{\p}$-lemma. Similarly to the K\"ahler case, we will reduce the proof to Kuranishi family as described in the beginning of Section \ref{kahler-section}.

Our goal is to construct a power series $\Omega(t)\in
A^{n-1,n-1}(X_0)$ such that
\begin{equation}\label{wan.2}
    \begin{cases}
      d(e^{\iota_{\varphi}|\iota_{\b{\varphi}}}(\Omega(t)))=0,\\
      \Omega(t)=\o{\Omega(t)},\\
      \Omega(0)=\omega^{n-1},
    \end{cases}
\end{equation}
and show the H\"older $C^{k,\alpha}$-convergence and regularity of
the power series $\Omega(t)$. Then it is clear that
$e^{\iota_{\varphi}|\iota_{\b{\varphi}}}(\Omega(t)))$ will be a
positive $(n-1,n-1)$-form on $X_t$ for $t$ small, due to the
positivity of its initial complex differential form $\omega^{n-1}$
and the convergence argument. By the real property of
$e^{\iota_{\varphi}|\iota_{\b{\varphi}}}$ as in \cite[Lemma
$2.8$]{RZ15}, it suffices to solve the following system of equations
\begin{equation}\label{wan.2.1}
    \begin{cases}
      d(e^{\iota_{\varphi}|\iota_{\b{\varphi}}}(\Omega(t)))=0,\\
      \Omega(0)=\omega^{n-1},\\
    \end{cases}
\end{equation}
since $\frac{\Omega(t)+\o{\Omega(t)}}{2}$ from one solution
$\Omega(t)$ of \eqref{wan.2.1} becomes one of the system
\eqref{wan.2}. The resolution of the system \eqref{wan.2.1} below is
a bit different from the one for the K\"ahler case, which relies
more on the form type $(n-1,n-1)$.

As both $e^{\iota_{(\1-\b{\varphi}\varphi)^{-1}\b{\varphi}}}$ and
$e^{\iota_{\varphi}}$ are invertible operators when $t$ is
sufficiently small, it follows that for any $\Omega \in
A^{n-1,n-1}(X_0)$,
\begin{align}\label{2.1.1}
  e^{\iota_{\varphi}|\iota_{\b{\varphi}}}(\Omega)=e^{\iota_{\varphi}}\circ e^{\iota_{(\1-\b{\varphi}\varphi)^{-1}\b{\varphi}}}\circ e^{-\iota_{(\1-\b{\varphi}\varphi)^{-1}\b{\varphi}}}\circ e^{-\iota_{\varphi}}\circ e^{\iota_{\varphi}|\iota_{\b{\varphi}}}(\Omega).
\end{align}
Set
\begin{align}\label{2.1.2}
  \tilde{\Omega}=e^{-\iota_{(\1-\b{\varphi}\varphi)^{-1}\b{\varphi}}}\circ e^{-\iota_{\varphi}}\circ e^{\iota_{\varphi}|\iota_{\b{\varphi}}}(\Omega),
\end{align}
where $\Omega$ and $\tilde{\Omega}$ are apparently one-to-one
correspondence. And it is easy to check that the operator
$$e^{-\iota_{(\1-\b{\varphi}\varphi)^{-1}\b{\varphi}}}\circ
e^{-\iota_{\varphi}}\circ e^{\iota_{\varphi}|\iota_{\b{\varphi}}}$$
preserves the form types and thus $\tilde{\Omega}$ is still an
$(n-1,n-1)$-form. In fact, for any $(p,q)$-form $\alpha$ on $X_0$,
we will find
\begin{align*}
\begin{split}
& e^{-\iota_{(\1-\b{\varphi}\varphi)^{-1}\b{\varphi}}}\circ
e^{-\iota_{\varphi}}\circ e^{\iota_{\varphi}|\iota_{\b{\varphi}}}(\alpha) \\
=&\alpha_{i_1\cdots i_p\o{j_1}\cdots\o{j_q}}
dz^{i_1}\wedge\cdots\wedge dz^{i_p} \wedge(\1-\b{\varphi}\varphi)\l
d\b{z}^{j_1}\wedge \cdots\wedge (\1-\b{\varphi}\varphi)\l
d\b{z}^{j_q}\in A^{p,q}(X_0),
\end{split}
\end{align*}
where $\alpha=\alpha_{i_1\cdots
i_p\o{j_1}\cdots\o{j_q}}dz^{i_1}\wedge\cdots\wedge dz^{i_p}\wedge
d\b{z}^{j_1}\wedge \cdots\wedge d\b{z}^{j_q}$. Together with
\eqref{2.1.1} and \eqref{2.1.2}, we obtain that
\begin{align}\label{2.2.1}
\begin{split}
  d(e^{\iota_{\varphi}|\iota_{\b{\varphi}}}(\Omega))&=d\circ e^{\iota_{\varphi}}\circ e^{\iota_{(\1-\b{\varphi}\varphi)^{-1}\b{\varphi}}}(\tilde{\Omega}) \\
                                                    &=e^{\iota_{\varphi}}\circ(\b{\p}+[\p, \iota_{\varphi}]+\p)\circ e^{\iota_{(\1-\b{\varphi}\varphi)^{-1}\b{\varphi}}}(\tilde{\Omega}) \\
                                                    &=e^{\iota_{\varphi}}\circ(\b{\p}+[\p,\iota_{\varphi}]+\p)
                                                    \big(\tilde{\Omega}+\iota_{(\1-\b{\varphi}\varphi)^{-1}\b{\varphi}}(\tilde{\Omega})\big), \\
\end{split}
\end{align}
where Proposition \ref{main1} is used in the second equality of \eqref{2.2.1}
and the third equality results from the form type of
$\tilde{\Omega}$.

By the invertibility of the operator $e^{\iota_{\varphi}}$ and the
form-type comparison, the equation
$d(e^{\iota_{\varphi}|\iota_{\b{\varphi}}}(\Omega))=0$ amounts to
 \begin{equation}\label{wan.10}
    \begin{cases}
      (\b{\p}+[\p, \iota_{\varphi}])\tilde{\Omega}=0,\\
      \p\tilde{\Omega}+(\b{\p}+\p\circ \iota_{\varphi})\circ \iota_{(\1-\b{\varphi}\varphi)^{-1}\b{\varphi}}(\tilde{\Omega})=0.
    \end{cases}
 \end{equation}
Then the second equation in (\ref{wan.10}) and Lemma \ref{aaaa}
imply
 \begin{align}\label{wan.11}
 \begin{split}
   \iota_{\varphi}\circ \p\tilde{\Omega}&=-\iota_{\varphi}\circ(\b{\p}+\p\circ \iota_{\varphi})\circ\iota_{(\1-\b{\varphi}\varphi)^{-1}\b{\varphi}}(\tilde{\Omega})\\
   &=-(\b{\p}\circ \iota_{\varphi}-\iota_{\frac{1}{2}[\varphi,\varphi]}+\iota_{\varphi}\circ\p\circ \iota_{\varphi})\circ \iota_{(\1-\b{\varphi}\varphi)^{-1}\b{\varphi}}(\tilde{\Omega})\\
   &=-(\b{\p}\circ \iota_{\varphi}+\frac{1}{2}\p\circ\ \iota_{\varphi}\circ \iota_{\varphi}+\frac{1}{2}\iota_{\varphi}\circ \iota_{\varphi}\circ\p)\circ \iota_{(\1-\b{\varphi}\varphi)^{-1}\b{\varphi}}(\tilde{\Omega})\\
   &=-(\b{\p}\circ \iota_{\varphi}+\frac{1}{2}\p\circ\ \iota_{\varphi}\circ \iota_{\varphi})\circ
   \iota_{(\1-\b{\varphi}\varphi)^{-1}\b{\varphi}}(\tilde{\Omega}),
   \end{split}
 \end{align}
where the form type of $\tilde{\Omega}$ is also used in the fourth
equality of \eqref{wan.11}. Substituting (\ref{wan.11}) into
(\ref{wan.10}), one obtains that (\ref{wan.10}) is equivalent to
\begin{equation}\label{wan.12}
    \begin{cases}
      \big(\b{\p}+\p\circ \iota_{\varphi}+\b{\p}\circ \iota_{\varphi}\circ \iota_{(\1-\b{\varphi}\varphi)^{-1}\b{\varphi}}+\frac{1}{2}\p\circ\ \iota_{\varphi}\circ \iota_{\varphi}\circ \iota_{(\1-\b{\varphi}\varphi)^{-1}\b{\varphi}}\big)\tilde{\Omega}=0,\\
      \big(\p+\b{\p}\circ \iota_{(\1-\b{\varphi}\varphi)^{-1}\b{\varphi}}+\p\circ \iota_{\varphi}\circ \iota_{(\1-\b{\varphi}\varphi)^{-1}\b{\varphi}}\big)\tilde{\Omega}=0.
    \end{cases}
\end{equation}

For the resolution of $\p\db$-equations, we need a lemma due to \cite[Theorem $4.1$]{P1}:
\begin{lemma}[]\label{ddbar-eq}
Let $(X,\omega)$ be a compact Hermitian complex manifold with the
pure-type complex differential forms $x$ and $y$. Assume that the
$\p \db$-equation
\begin{equation}\label{ddb-eq}
\p \db x =y
\end{equation}
admits a solution. Then an explicit solution of the $\p
\db$-equation \eqref{ddb-eq} can be chosen as
$$(\p\db)^*\G_{BC}y,$$
which uniquely minimizes the $L^2$-norms of all the solutions with
respect to $\omega$. Besides, the equalities hold
\[\G_{BC}(\p\db) = (\p\db) \G_{A}\quad \text{and} \quad (\p\db)^*\G_{BC} = \G_{A}(\p\db)^*,\]
where $\G_{BC}$ and $\G_{A}$ are the associated Green's operators of
$\square_{BC}$ and $\square_{A}$, respectively. Here $\square_{BC}$
is defined in \eqref{bc-Lap} and $\square_{{A}}$ is the \emph{second
Kodaira--Spencer operator} (often also called \emph{Aeppli
Laplacian})
$$\square_{{A}}=\p^*\db^*\db\p+\db\p\p^*\db^*+\db\p^*\p\db^*+\p\db^*\db\p^*+\db\db^*+\p\p^*.$$
\end{lemma}

\begin{proof}
We shall use the Hodge decomposition of $\square_{BC}$ on $X$:
\begin{equation}\label{bc-hod}
A^{p,q}(X)=\ker\square_{BC}\oplus \textrm{Im}~(\p\db)
\oplus(\textrm{Im}~\p^*+\textrm{Im}~\db^*),
\end{equation}
whose three parts are orthogonal to each other with respect to the
$L^2$-scalar product defined by $\omega$, combined with the equality
$$\1=\mathbb{H}_{BC}+\square_{BC}\G_{BC},$$
where $\mathbb{H}_{BC}$ is the harmonic projection operator. And it
should be noted that
\begin{equation}\label{bc-ker}
\ker\square_{BC}=\ker\p\cap\ker\db\cap\ker(\p\db)^*.
\end{equation}
Then two observations follow:
\begin{enumerate}[(1)]
    \item \label{case-1}
$\square_{BC}\p\db(\p\db)^*=\p\db(\p\db)^*\square_{BC};$
    \item \label{case-2}
$\G_{BC}\p\db(\p\db)^*=\p\db(\p\db)^*\G_{BC}.$
 \end{enumerate}
It is clear that \eqref{case-1} implies \eqref{case-2}. Actually,
\eqref{case-1} yields
\[ \G_{BC}\square_{BC}\p\db(\p\db)^*\G_{BC}=\G_{BC}\p\db(\p\db)^*\square_{BC}\G_{BC}. \]
A routine check shows that
$$\G_{BC}\square_{BC}\p\db(\p\db)^*\G_{BC}=(\1-\mathbb{H}_{BC})\p\db(\p\db)^*\G_{BC}=\p\db(\p\db)^*\G_{BC}$$
and
$$\G_{BC}\p\db(\p\db)^*\square_{BC}\G_{BC}=\G_{BC}\p\db(\p\db)^*(\1-\mathbb{H}_{BC})=\G_{BC}\p\db(\p\db)^*;$$
while, the statement \eqref{case-1} is proved by a direct
calculation:
$$\square_{BC}\p\db(\p\db)^*=(\p\db)(\p\db)^*(\p\db)(\p\db)^*=\p\db(\p\db)^*\square_{BC}.$$
Hence, one has
$$(\p\db)(\p\db)^*\G_{BC}y=\G_{BC}(\p\db)(\p\db)^*y=\G_{BC}\square_{BC}y=(\1-\mathbb{H}_{BC})y=y,$$
where $y \in \textrm{Im}~\p \db$ due to the solution-existence of
the $\p\db$-equation.

To see that the solution $(\p\db)^*\G_{BC}y$ is the unique
$L^2$-norm minimum, we resort to the Hodge decomposition of the
operator $\square_{A}$:
\begin{equation}\label{A-H}
A^{p,q}(X)=\ker\square_{A} \oplus (Im~\p +Im~\db ) \oplus
Im~(\p\db)^*,
\end{equation}
where $\ker \square_{A} = \ker (\p \db) \cap \ker \p^* \cap \ker
\db^*$. Let $z$ be an arbitrary solution of the $\p\db$-equation
\eqref{ddb-eq}, which decomposes into three components $z_1 + z_2
+z_3$ with respect to the Hodge decomposition \eqref{A-H} of
$\square_{A}$. By the Hodge theory of $\square_{A}$, the equality
holds
\[ \ker (\p\db) = \ker \square_{A} \oplus ( \textrm{Im}\p + \textrm{Im}\db ) ,\]
which implies that $\p \db (z_1 + z_2)=0$. Hence, it follows that
\[ \p \db z = \p \db z_3 =y. \]
After noticing that $\db^* z_3 = \p^* z_3 =0$, we get
\[ (\p\db)^* y = (\p\db)^* \p \db z_3 = \square_{A} z_3, \]
which implies that
\[ \mathbb{G}_{A} (\p\db)^* y = \mathbb{G}_{A} \square_{A} z_3= (\1-\mathbb{H}_{A})z_3 = z_3. \]
Then it is obvious that
\begin{equation}\label{bc-d-db}
\square_{BC} (\p\db)=\p \db (\p\db)^* \p \db=(\p\db)\square_{A},
\end{equation}
with $\G_{BC}(\p\db)=(\p\db)\G_{A}$ established as well. Taking
adjoint operators of both sides in \eqref{bc-d-db}, we will find
that
\[  (\p\db)^* \square_{BC} = \square_{A} (\p\db)^*,\]
implying the equality
\[(\p\db)^*\mathbb{G}_{BC}=\mathbb{G}_{A}(\p\db)^*.\] And thus, we
obtain that
\[ z_3 =  \mathbb{G}_{A} (\p\db)^* y = (\p\db)^*\mathbb{G}_{BC} y. \]
Therefore, \[ \|z\|^2 = \|z_1\|^2 +\|z_2\|^2+\|z_3\|^2\geq \|z_3\|^2
= \|(\p\db)^*\mathbb{G}_{BC}y \|^2, \] and the equality holds if and
only if $z_1 =z_2 =0$, i.e., $z=z_3=(\p\db)^*\mathbb{G}_{BC} y$.
\end{proof}

Now we arrive at:
\begin{proposition}\label{fml-blc}
Let $\omega$ be a balanced metric on $X_0$, which satisfies the
$(n-1,n)$-th mild $\p\b{\p}$-lemma. Then the system \eqref{wan.2} of
equations is formally solved.
\end{proposition}

\begin{proof}
It suffices to resolve \eqref{wan.12} with  initial value
$\tilde{\Omega}(0)=\omega^{n-1}$. Actually, the solution
$\tilde{\Omega}(t)$ of the system \eqref{wan.12}, satisfying that
$\tilde{\Omega}(0)=\omega^{n-1}$, corresponds to the one of
\eqref{wan.2.1}, by the relation below as in \eqref{2.1.2}
\[
\tilde{\Omega}(t)=e^{-\iota_{(\1-\b{\varphi}\varphi)^{-1}\b{\varphi}}}\circ
e^{-\iota_{\varphi}}\circ
e^{\iota_{\varphi}|\iota_{\b{\varphi}}}(\Omega(t)).
\]
Therefore, we focus on
$$
    \begin{cases}
      \big(\b{\p}+\p\circ \iota_{\varphi}+\b{\p}\circ \iota_{\varphi}\circ \iota_{(\1-\b{\varphi}\varphi)^{-1}\b{\varphi}}+\frac{1}{2}\p\circ\ \iota_{\varphi}\circ \iota_{\varphi}\circ \iota_{(\1-\b{\varphi}\varphi)^{-1}\b{\varphi}}\big)\tilde{\Omega}(t)=0,\\
      \big(\p+\b{\p}\circ \iota_{(\1-\b{\varphi}\varphi)^{-1}\b{\varphi}}+\p\circ \iota_{\varphi}\circ \iota_{(\1-\b{\varphi}\varphi)^{-1}\b{\varphi}}\big)\tilde{\Omega}(t)=0,\\
      \tilde{\Omega}(0)=\omega^{n-1},
    \end{cases}
$$
and solve this system of equations also by the iteration method used
in the K\"ahler case.

Assume that $\tilde{\Omega}(t)$ can develop into a power series as
follows:
\begin{equation*}
    \begin{cases}
      \tilde{\Omega}(t)=\sum_{k=0}^{\infty}\tilde{\Omega}_k,\\
      \tilde{\Omega}_k=\sum_{i+j=k}\tilde{\Omega}_{i,j}t^i\b{t}^j,
    \end{cases}
\end{equation*}
where $\tilde{\Omega}_k$ is the $k$-degree homogeneous part in the
expansion of $\tilde{\Omega}(t)$  and $\tilde{\Omega}_{i,j}$ are all
smooth $(n-1,n-1)$-forms on $X_0$. Here we follow the notation
described before Proposition \ref{2.9}. And in \eqref{phi-ps-p} the
holomorphic family of integrable Beltrami differentials on $X_0$ develops a power series of Beltrami
differentials in $t$ as
$$
\varphi(t) = \sum_{i=1}^{\infty}\varphi_{i}t^i.
$$
Hence, we need to
solve
\begin{equation}\label{wan.16}
    \begin{cases}
      \left(\big(\b{\p}+\p\circ \iota_{\varphi}+\b{\p}\circ \iota_{\varphi}\circ \iota_{(\1-\b{\varphi}\varphi)^{-1}\b{\varphi}}+\frac{1}{2}\p\circ\ \iota_{\varphi}\circ \iota_{\varphi}\circ \iota_{(\1-\b{\varphi}\varphi)^{-1}\b{\varphi}}\big)\tilde{\Omega}(t)\right)_k=0,\\
      \left(\big(\p+\b{\p}\circ \iota_{(\1-\b{\varphi}\varphi)^{-1}\b{\varphi}}+\p\circ \iota_{\varphi}\circ \iota_{(\1-\b{\varphi}\varphi)^{-1}\b{\varphi}}\big)\tilde{\Omega}(t)\right)_k=0,\\
      \tilde{\Omega}(0)=\omega^{n-1},
    \end{cases}
\end{equation}
for any $k\geq 0$.

The case $k=0$ of the equations \eqref{wan.16} holds since $\omega$
is a balanced metric. By induction, we assume that the system
\eqref{wan.16} of equations is solved for each $k\leq l$ and the
solutions are denoted by $\{\tilde{\Omega}_k\}_{k\leq l}\in
A^{n-1,n-1}(X_0)$. By the form-type consideration, one observes
that the $(n-1,n)$-th mild $\p\db$-lemma produces $\mu$ and $\nu$,
such that the following equalities
\[ \left\{ \begin{aligned}
\p \db \mu &= -\db  \Big( \iota_{(\1-\b{\varphi}\varphi)^{-1}\b{\varphi}} \big( \tilde{\Omega}(t) \big) \Big)_{l+1},\\
\db \p \nu &= -\p \Big( \big( \iota_{\varphi}+ \frac{1}{2}
\iota_{\varphi}\circ \iota_{\varphi}\circ
\iota_{(\1-\b{\varphi}\varphi)^{-1}\b{\varphi}}\big)\tilde{\Omega}(t) \Big)_{l+1} \\
\end{aligned} \right. \]
hold as the ones \eqref{mslmm} in Observation \ref{expslt}. Then
Lemma \ref{ddbar-eq} enables us to determine the two explicit
solutions \beq\label{ddbarslt} \left\{ \begin{aligned}
\mu &= - (\p\db)^* \G_{BC} \db  \Big( \iota_{(\1-\b{\varphi}\varphi)^{-1}\b{\varphi}} \big( \tilde{\Omega}(t) \big) \Big)_{l+1}, \\
\nu &= (\p\db)^* \G_{BC} \p \Big( \big( \iota_{\varphi}+ \frac{1}{2}
\iota_{\varphi}\circ \iota_{\varphi}\circ
\iota_{(\1-\b{\varphi}\varphi)^{-1}\b{\varphi}}\big)\tilde{\Omega}(t)
\Big)_{l+1}.
\end{aligned} \right. \eeq
Hence, the $(l+1)$-th order of the system \eqref{wan.16} is solved,
yielding that
\[\tilde{\Omega}_{l+1} + \Big( \iota_{\varphi}\circ \iota_{(\1-\b{\varphi}\varphi)^{-1}\b{\varphi}} \big( \tilde{\Omega}(t)\big) \Big)_{l+1} =
\db \mu + \p \nu,\] in the same manner as \eqref{cbslt}, where $\mu$
and $\nu$ are given by \eqref{ddbarslt}. Therefore, we complete the
proof.
\end{proof}

Then we come to the  H\"older  convergence argument for $\tilde{\Omega}(t)$.
For $l=1,2,\cdots$, one $l$-degree canonical formal solution of
\eqref{wan.16} is given by induction:
\begin{equation}\label{can-sol-Om}
\begin{aligned}
\tilde{\Omega}_l
 =&-\Big( \iota_{\varphi}\circ
\iota_{(\1-\b{\varphi}\varphi)^{-1}\b{\varphi}} \big(
\tilde{\Omega}(t)\big) \Big)_l -\db (\p\db)^* \G_{BC} \db  \Big(
\iota_{(\1-\b{\varphi}\varphi)^{-1}\b{\varphi}} \big(
\tilde{\Omega}(t) \big) \Big)_l\\
& + \p (\p\db)^* \G_{BC} \p \Big( \big( \iota_{\varphi}+ \frac{1}{2}
\iota_{\varphi}\circ \iota_{\varphi}\circ
\iota_{(\1-\b{\varphi}\varphi)^{-1}\b{\varphi}}\big)\tilde{\Omega}(t)
\Big)_l
\end{aligned}
\end{equation}
and obviously
\begin{equation}\label{can-sol-Om-up}
\begin{aligned}
\tilde{\Omega}^{(l)}
 =&-\Big( \iota_{\varphi}\circ
\iota_{(\1-\b{\varphi}\varphi)^{-1}\b{\varphi}} \big(
\tilde{\Omega}(t)\big) \Big)^{(l)} -\db (\p\db)^* \G_{BC} \db  \Big(
\iota_{(\1-\b{\varphi}\varphi)^{-1}\b{\varphi}} \big(
\tilde{\Omega}(t) \big) \Big)^{(l)}\\
& + \p (\p\db)^* \G_{BC} \p \Big( \big( \iota_{\varphi}+ \frac{1}{2}
\iota_{\varphi}\circ \iota_{\varphi}\circ
\iota_{(\1-\b{\varphi}\varphi)^{-1}\b{\varphi}}\big)\tilde{\Omega}(t)
\Big)^{(l)}.
\end{aligned}
\end{equation}
In this proof we follow the notations in Subsection \ref{convergence}.

We use an important a priori estimate for the three terms in the
left-hand side of the above equality: for any complex differential
form $\phi$,
\begin{equation}\label{ee2}
\|\G_{BC}\phi\|_{k, \alpha}\leq C_{k,\alpha}\|\phi\|_{k-4, \alpha},
\end{equation}
where $k>3$ and $C_{k,\alpha}$ depends on only on $k$ and $\alpha$,
not on $\phi$. See \cite[Appendix.Theorem $7.4$]{k} for example. Assume
that
$$\|\tilde{\Omega}^{(r-1)}\|_{k, \alpha},\|\tilde{\Omega}^{(r-2)}\|_{k, \alpha},\|\tilde{\Omega}^{(r-3)}\|_{k, \alpha}\ll
A(t).$$ So, by the expression \eqref{can-sol-Om-up} and the a priori
estimate \eqref{ee2},
\begin{align*}
   &\|-\Big( \iota_{\varphi}\circ
\iota_{(\1-\b{\varphi}\varphi)^{-1}\b{\varphi}} \big(
\tilde{\Omega}(t)\big) \Big)^{(r)}\|_{k, \alpha}\\
=&\|-(\varphi\lc(\1-\b{\varphi}\varphi)^{-1}\b{\varphi}\lc\tilde{\Omega}^{(r-2)})^{(r)}
-(\varphi\lc(\1-\b{\varphi}\varphi)^{-1}\b{\varphi}\lc\tilde{\Omega}_0)^{(r)}\|_{k, \alpha}\\
 \ll & 2\|\varphi^{(r-2)}\|_{k, \alpha}^2\cdot
\|\tilde{\Omega}^{(r-2)}\|_{k, \alpha}+2\|\varphi^{(r-1)}\|_{k,
\alpha}^2\cdot
\|\tilde{\Omega}_0\|_{k, \alpha}\\
\ll &
2\left(\frac{\beta}{\gamma}\right)\left(\frac{\beta}{\gamma}+\|\tilde{\Omega}_0\|_{k,
\alpha}\right)A(t),
\end{align*}

\begin{align*}
   &\|-\db (\p\db)^* \G_{BC} \db  \Big(
\iota_{(\1-\b{\varphi}\varphi)^{-1}\b{\varphi}} \big(
\tilde{\Omega}(t) \big) \Big)^{(r)}\|_{k, \alpha}\\
=&\Big\|\db(\p\db)^*\G_{BC}\db\Big((\1-\b{\varphi}\varphi)^{-1}\b{\varphi}\lc\tilde{\Omega}^{(r-1)}
 +(\1-\b{\varphi}\varphi)^{-1}\b{\varphi}\lc\tilde{\Omega}_0\Big)^{(r)}\Big\|_{k, \alpha}\\
 \ll&
 C_1^2C_{k,\alpha}\Big(2\|\varphi^{(r-1)}\|_{k, \alpha}\|\tilde{\Omega}^{(r-1)}\|_{k,
\alpha}+2\|\varphi^{(r)}\|_{k, \alpha}\|\tilde{\Omega}_0\|_{k,
\alpha}\Big)\\
 \ll & 2C_1^2C_{k,\alpha}\left(\frac{\beta}{\gamma}+\|\tilde{\Omega}_0\|_{k,
\alpha}\right)A(t)
\end{align*}
and similarly,
\begin{align*}
   &\|\p (\p\db)^* \G_{BC} \p \Big( \big( \iota_{\varphi}+ \frac{1}{2}
\iota_{\varphi}\circ \iota_{\varphi}\circ
\iota_{(\1-\b{\varphi}\varphi)^{-1}\b{\varphi}}\big)\tilde{\Omega}(t)
\Big)^{(r)}\|_{k, \alpha}\\
 \ll & C_1^2C_{k,\alpha}\left(\frac{\beta}{\gamma}+\|\tilde{\Omega}_0\|_{k,
\alpha}\right)\left(1+\left(\frac{\beta}{\gamma}\right)^2\right)A(t).
\end{align*}
Hence, we choose $\beta,\gamma,\|\tilde{\Omega}_0\|_{k, \alpha}$ so
that the following inequalities hold:
$$2\left(\frac{\beta}{\gamma}\right)\left(\frac{\beta}{\gamma}+\|\tilde{\Omega}_0\|_{k,
\alpha}\right),
2C_1^2C_{k,\alpha}\left(\frac{\beta}{\gamma}+\|\tilde{\Omega}_0\|_{k,
\alpha}\right),
C_1^2C_{k,\alpha}\left(\frac{\beta}{\gamma}+\|\tilde{\Omega}_0\|_{k,
\alpha}\right)\left(1+\left(\frac{\beta}{\gamma}\right)^2\right)
<\frac{1}{3},$$ \eqref{initial-bg} for the integrable Beltrami
differential $\varphi(t)$ and
$$\|\tilde{\Omega}^{(1)}\|_{k, \alpha},\|\tilde{\Omega}^{(2)}\|_{k,
\alpha},\|\tilde{\Omega}^{(3)}\|_{k, \alpha}\ll A(t)$$ according the
above formulation, which are obviously possible as $t$ is small. Thus we obtain that
$$\|\tilde{\Omega}^{(r)}\|_{k, \alpha}\ll A(t),\ \text{for any $r\geq 1$},$$
which implies the desired convergence $\|\tilde{\Omega}^{(+\infty)}\|_{k,
\alpha}\ll A(t)$ immediately.

Finally, inspired by the elliptic argument \cite[Appendix.\S 8]{k},
we will apply the interior estimate to obtain the regularity of
$\tilde{\Omega}(t)$, which is a local problem. By the canonical
formal solution expression \eqref{can-sol-Om} of \eqref{wan.16}, one just needs to
use the following strongly elliptic second-order
pseudo-differential equation
\begin{equation}
\label{reg-eqn}
\begin{aligned}
\square\tilde{\Omega}(t)
 =&-\square\Big( \iota_{\varphi}\circ
\iota_{(\1-\b{\varphi}\varphi)^{-1}\b{\varphi}} \big(
\tilde{\Omega}(t)\big) \Big) -\db \db^*\db(\p\db)^* \G_{BC} \db
\Big( \iota_{(\1-\b{\varphi}\varphi)^{-1}\b{\varphi}} \big(
\tilde{\Omega}(t) \big) \Big)\\
& + \square\p (\p\db)^* \G_{BC} \p \Big( \big( \iota_{\varphi}+
\frac{1}{2} \iota_{\varphi}\circ \iota_{\varphi}\circ
\iota_{(\1-\b{\varphi}\varphi)^{-1}\b{\varphi}}\big)\tilde{\Omega}(t)
\Big),
\end{aligned}
\end{equation}
where $\square$ is the $\db$-Laplacian defined by \eqref{dbar-Lap}.
Recall that an elliptic partial differential operator of order $2m$ is
pseudo-differential and so is its inverse, whose order becomes $-2m$
as a pseudo-differential operator.

We cover $X:=X_0$ by coordinates neighborhoods $X_j$, $j=1,\cdots,
J$. Let $z_j=(z_j^1,\cdots,z_j^n)$ be the local holomorphic
coordinates on $X_j$ with $z_j^k=x_j^k+\sqrt{-1}x_j^{k+n}$ and put
$t=x^{2n+1}+\sqrt{-1}x^{2n+2}\in \Delta_{\epsilon}$ a $\epsilon$-disk in
$\mathbb{C}$. By these $2n+2$
real coordinates, $X_j\times \Delta_{\epsilon}$ is identified with
an open set $U_j$ of a $(2n+2)$-dimensional torus
$\mathds{T}^{2n+2}$.

Choose a partition of unity subordinate to $X_j$, that is, a set
$\{\rho_j\}$ of $C^\infty$ functions $\rho_j(x)$ on $X$ so that
$\sup\rho_j\subset X_j$ and for any $x\in X$,
$\sum_{j=1}^{J}\rho_j(x)\equiv 1$. For each $l=1,2,\cdots$, choose a
smooth function $\eta^l(t)$ with values in $[0,1]$:
\begin{equation}\label{eta-l}
\eta^l(t)\equiv
    \begin{cases}
      1,\ \text{for $|t|\leq (\frac{1}{2}+\frac{1}{2^{l+1}})r$};\\
      0,\ \text{for $|t|\geq (\frac{1}{2}+\frac{1}{2^{l}})r$},
    \end{cases}
 \end{equation}
where $r$ is a positive constant to be determined. Notice that $r$
is crucially used to give the uniform bound for the convergence
radius of $\tilde{\Omega}(t)$.

Set
\begin{equation}\label{ro-eta}
\rho_j^l(x,t)=\rho_j(x)\eta^l(t).
\end{equation}
Recall in the proof for $\varphi(t)$ in \cite[Appendix.\S 8]{k}, one
should also set
$$\chi_j^l(x,t)=\chi_j(x)\eta^l(t),$$
where the smooth function $\chi_j(x)$ with $\sup\chi_j\subset X_j$
is identically equal to $1$ on some neighborhood of the support of
$\rho_j$. But here we will replace its role directly by $\eta^l(t)$
to avoid the trouble caused by the presence of Green's operator
$\G_{BC}$.

First we will prove that $\eta^3\tilde{\Omega}$ is $C^{k+1,\alpha}$.
Consider the equation:
\begin{equation}\label{F21}
\square(\triangle^h_i\rho_j^3\tilde{\Omega})
 =F_1^1+F_2^1+F_3^1,
\end{equation}
where $F_1^1,F_2^1,F_3^1$ denote the three terms with respect to the
ones in the right-hand side of \eqref{reg-eqn} after the
corresponding operations, respectively, and $i=1,\cdots, 2n$. Here
$\triangle^h_i$ is the difference quotient as
\cite[Appendix.(8.14)]{k}. In particular,
$$F_2^1=-\triangle^h_i\left(\rho_j^3\db \db^*\db(\p\db)^* \G_{BC} \db
\Big( \iota_{(\1-\b{\varphi}\varphi)^{-1}\b{\varphi}} \big(
\tilde{\Omega}(t) \big) \Big)\right)+\text{lower-order terms of
$\tilde{\Omega}(t)$}.$$ In this proof the \lq\lq order" refers to the one of a
pseudo-differential operator. Since $\square$ is an elliptic linear
differential operator whose principal part is of diagonal type, by
\cite[Appendix.Theorem 2.3]{k} and \eqref{F21}, one obtains the a
priori estimate
\begin{equation}\label{4-ck}
\begin{aligned}
\|\triangle^h_i\rho_j^3\tilde{\Omega}\|_{k, \alpha}
 &\leq C_k(\|
F_1^1+F_2^1+F_3^1 \|_{k-2,
\alpha}+\|\triangle^h_i\rho_j^3\tilde{\Omega}\|_{0})\\
 &\leq C_k(\|
F_1^1\|_{k-2, \alpha}+\| F_2^1\|_{k-2, \alpha}+\| F_3^1 \|_{k-2,
\alpha}+\|\triangle^h_i\rho_j^3\tilde{\Omega}\|_{0}),
\end{aligned}
\end{equation}
where $C_k$ is a positive constant, possibly depending on $k$. Now
let us estimate the first three terms in the right-hand side of the
above inequality. Here we just estimate the second one, the most
troublesome one, since the other two terms are quite analogous.

We need an equality for the Green's operator: Let $E:=E(x, D)$ be an
elliptic linear partial differential operator on a smooth manifold
 and $G_E$, $\mathbb{H}_E$ its associated Green's operator and orthogonal projection to $\ker E$ defined such as in
\cite[Appendix.Definition 7.2]{k}. Then for any smooth function $f$
and differential form $\varpi$ on this manifold,
\begin{equation}\label{leibniz}
fG_E\varpi=G_E(f\varpi)-G_E(f\mathbb{H}_E\varpi)+\mathbb{H}_E(fG_E\varpi)+\text{lower-order
terms of $\varpi$}.
\end{equation}
Applying this equality to $\rho_j^3\db \db^*\db(\p\db)^* \G_{BC} \db
\Big( \iota_{(\1-\b{\varphi}\varphi)^{-1}\b{\varphi}} \big(
\tilde{\Omega}(t) \big) \Big)$ with $E=\square_{BC}$, $f=\rho_j^3$
and $\varpi=\db \Big(
\iota_{(\1-\b{\varphi}\varphi)^{-1}\b{\varphi}} \big(
\tilde{\Omega}(t) \big) \Big)$, one obtains
\begin{equation}\label{Lb-app}
\begin{aligned}
   &\rho_j^3\db \db^*\db(\p\db)^* \G_{BC} \db \Big(
\iota_{(\1-\b{\varphi}\varphi)^{-1}\b{\varphi}} \big(
\tilde{\Omega}(t) \big) \Big)\\
 =&\db \db^*\db(\p\db)^* \G_{BC} \db \Big(
\iota_{(\1-\b{\varphi}\varphi)^{-1}\b{\varphi}} \big(\rho_j^3
\tilde{\Omega}(t) \big) \Big)+\text{lower-order terms of
$\tilde{\Omega}(t)$},
\end{aligned}
\end{equation}
where we use the fact that $\db \Big(
\iota_{(\1-\b{\varphi}\varphi)^{-1}\b{\varphi}} \big(
\tilde{\Omega}(t) \big) \Big)$ is $\p\db$-exact and the equalities
\eqref{bc-hod}, \eqref{bc-ker}. Thus, by \eqref{Lb-app} and the
useful formula
\begin{equation}\label{eta-2l-1-2l+1}
\eta^{2l-1}\cdot\eta^{2l+1}=\eta^{2l+1},\ l=1,2,\cdots,
\end{equation} one gets the estimate on the first term of
$F_2^1$: $$\label{}
\begin{aligned}
   &\| -\triangle^h_i\left(\rho_j^3\db \db^*\db(\p\db)^* \G_{BC} \db \Big(
\iota_{(\1-\b{\varphi}\varphi)^{-1}\b{\varphi}} \big(
\tilde{\Omega}(t) \big) \Big)\right)\|_{k-2, \alpha}\\
 \leq& L\|\eta^1\varphi \|_{0}\|\triangle^h_i(\rho_j^3\tilde{\Omega}(t))\|_{k, \alpha}
 +L'\|\eta^3\tilde{\Omega}(t)\|_{k, \alpha}\|\eta^1\varphi\|_{k,
 \alpha},
\end{aligned}
$$
where $L,L'$ are positive numbers. Hence, one obtains:
\begin{equation}\label{M_2^1}
\| F_2^1\|_{k-2, \alpha}\leq
M_2^1A(r)\|\triangle^h_i(\rho_j^3\tilde{\Omega}(t))\|_{k, \alpha}
+M_k\|\eta^3\tilde{\Omega}(t)\|_{k, \alpha}\|\eta^1\varphi\|_{k,
 \alpha},
\end{equation}
where $M_2^1,M_k$ are positive numbers. Similarly, we are able to
get the analogous estimates for $F_1^1$ and $F_3^1$ with the positive
constants $M_1^1$ and $M_3^1$. Thus,
\begin{equation}\label{dq-3}
\begin{aligned}
   &\| \triangle^h_i(\rho_j^3\tilde{\Omega}(t))\|_{k, \alpha}\\
\leq&C_k(M_1^1+M_2^1+M_3^1)A(r)\|\triangle^h_i(\rho_j^3\tilde{\Omega}(t))\|_{k,
\alpha}+C_k\|\rho_j^3\tilde{\Omega}(t)\|_{1}+C_kM_k\|\eta^3\tilde{\Omega}(t)\|_{k,
\alpha}\|\eta^1\varphi\|_{k,
 \alpha}.
\end{aligned}
\end{equation}
Choose a sufficiently small $r$ such that
\begin{equation}\label{r-1/2}
C_k(M_1^1+M_2^1+M_3^1)A(r)\leq \frac{1}{2}
\end{equation}
and if $|t|<r$
\begin{equation}\label{r-om-phi}
\tilde{\Omega}(t), \varphi(t) \in C^{k, \alpha}.
\end{equation}
Then $$\label{} \eta^3\tilde{\Omega}(t), \eta^1\varphi(t) \in C^{k,
\alpha}.
$$
Therefore, by \eqref{dq-3}, one knows that
$$
   \| \triangle^h_i(\rho_j^3\tilde{\Omega}(t))\|_{k, \alpha}
\leq
2C_k\|\rho_j^3\tilde{\Omega}(t)\|_{1}+2C_kM_k\|\rho_j^3\tilde{\Omega}(t)\|_{k,
\alpha}\|\eta^1\varphi\|_{k,
 \alpha},
$$
where the right-hand side is bounded and independent of $h$. Hence,
we have proved $$\rho_j^3\tilde{\Omega}(t)\in C^{k+1, \alpha}$$ by
\cite[Appendix Lemma 8.2.(iii)]{k}. Summing with respect $j$, one
knows that $\eta^3\tilde{\Omega}(t)\in C^{k+1, \alpha}$ by
\eqref{ro-eta}.

Next we shall prove that $\eta^5\tilde{\Omega}$ is $C^{k+2,\alpha}$.
Consider the equation:
\begin{align*}
\square(\triangle^h_i D_\beta(\rho_j^5\tilde{\Omega}))
 =F_1^2+F_2^2+F_3^2,
\end{align*}
where $D_\beta=\frac{\p\ }{\p x^\beta}$ and $F_1^2,F_2^2,F_3^2$
denote the three terms with respect to the ones in the right-hand
side of \eqref{reg-eqn} after the corresponding operations,
respectively.  In particular,
$$F_2^2=D_\beta F_2^1+\text{lower-order terms of
$\tilde{\Omega}(t)$}.$$  By \cite[Appendix.Theorem 2.3]{k}, one
obtains the a priori estimate
\begin{equation}\label{dq-3f}
\begin{aligned}
\|\triangle^h_iD_\beta(\rho_j^5\tilde{\Omega})\|_{k, \alpha}
 &\leq C_k(\|
F_1^2+F_2^2+F_3^2 \|_{k-2,
\alpha}+\|\triangle^h_iD_\beta(\rho_j^5\tilde{\Omega})\|_{0})\\
 &\leq C_k(\|
F_1^2\|_{k-2, \alpha}+\| F_2^2\|_{k-2, \alpha}+\| F_3^2 \|_{k-2,
\alpha}+\|\triangle^h_iD_\beta(\rho_j^5\tilde{\Omega})\|_{0}),
\end{aligned}
\end{equation}
where $C_k$ is the same as in \eqref{4-ck}. Now let us estimate the
first three terms in the right-hand side of the above inequality.
Here we just estimate the second one since the other two terms are
quite analogous. We use the equality \eqref{leibniz} for the Green's
operator again. Then one gets the estimate on the first term of
$F_2^2$: $$\label{}
\begin{aligned}
   &\| -\triangle^h_iD_\beta\left(\rho_j^5\db \db^*\db(\p\db)^* \G_{BC} \db \Big(
\iota_{(\1-\b{\varphi}\varphi)^{-1}\b{\varphi}} \big(
\tilde{\Omega}(t) \big) \Big)\right)\|_{k-2, \alpha}\\
 \leq& L\|\eta^3\varphi \|_{0}\|\triangle^h_iD_\beta(\rho_j^5\tilde{\Omega}(t))\|_{k, \alpha}
 +L''\|\eta^5\tilde{\Omega}(t)\|_{k+1, \alpha}\|\eta^3\varphi\|_{k+1,
 \alpha},
\end{aligned}
$$
where $L,L''$ are positive numbers. Hence, one obtains:
\begin{equation}\label{F22}
\| F_2^2\|_{k-2, \alpha}\leq
M_2^1A(r)\|\triangle^h_iD_\beta(\rho_j^5\tilde{\Omega}(t))\|_{k,
\alpha} +M_{k+1}\|\eta^5\tilde{\Omega}(t)\|_{k+1,
\alpha}\|\eta^3\varphi\|_{k+1,
 \alpha},
\end{equation}
where $M_2^1,M_{k+1}$ are positive numbers. Crucially, $M_2^1$ is
exactly the same as in \eqref{M_2^1}, since the range of the
function $\eta^l(t)$ is $[0,1]$ by \eqref{eta-l}, although
$M_{k+1}$ is possibly different from $M_k$ in \eqref{M_2^1}.
 Similarly,
we are able to get the analogous estimates for $F_1^2$ and $F_3^2$
with the positive constants $M_1^1$ and $M_3^1$, which are the same
as in the above argument for $\eta^3\tilde{\Omega}(t)\in C^{k+1,
\alpha}$. Thus, by \eqref{dq-3f} and \eqref{F22},
\begin{equation}\label{dq-5}
\begin{aligned}
   &\| \triangle^h_iD_\beta(\rho_j^5\tilde{\Omega}(t))\|_{k, \alpha}\\
\leq&C_k(M_1^1+M_2^1+M_3^1)A(r)\|\triangle^h_iD_\beta(\rho_j^5\tilde{\Omega}(t))\|_{k,
\alpha}+C_k\|\rho_j^5\tilde{\Omega}(t)\|_{2}+C_kM_{k+1}\|\eta^5\tilde{\Omega}(t)\|_{k+1,
\alpha}\|\eta^3\varphi\|_{k+1,
 \alpha}.
\end{aligned}
\end{equation}
 Choose the same sufficiently small $r$ as in the above argument for $\eta^3\tilde{\Omega}(t)\in C^{k+1,
\alpha}$, given by \eqref{r-1/2} and \eqref{r-om-phi}. Therefore, by
\eqref{dq-5},
$$
   \| \triangle^h_iD_\beta(\rho_j^5\tilde{\Omega}(t))\|_{k, \alpha}
\leq
2C_k\|\rho_j^5\tilde{\Omega}(t)\|_{2}+2C_kM_{k+1}\|\eta^5\tilde{\Omega}(t)\|_{k+1,
\alpha}\|\eta^3\varphi\|_{k+1,
 \alpha},
$$
where the right-hand side is bounded when we use the formula
\eqref{eta-2l-1-2l+1}, $\eta^3\tilde{\Omega}(t)\in C^{k+1, \alpha}$
that we just proved in the above argument and also the fact that
 $\eta^3\varphi(t)\in C^{k+1, \alpha}$ proved in
\cite[Appendix.\S 8]{k}. Hence, we have proved
$$D_\beta(\rho_j^5\tilde{\Omega}(t))\in C^{k+1, \alpha}$$ by
\cite[Appendix Lemma 8.2.(iii)]{k} and thus
$$\rho_j^5\tilde{\Omega}(t)\in C^{k+2, \alpha}.$$ Summing with
respect $j$, one obtains that $\eta^5\tilde{\Omega}(t)\in C^{k+2,
\alpha}$ by \eqref{ro-eta}. Notice that in this procedure $r$ has
not been replaced.

We can also prove that, for any $l=1,2,\cdots$,
$\eta^{2l+1}\tilde{\Omega}$ is $C^{k+l, \alpha}$, where
 $r$ can be chosen independent of $l$. Since $\eta^{2l+1}(t)$ is
identically equal to $1$ on $|t|<\frac{r}{2}$ which is independent
of $l$, $\tilde{\Omega}(t)$ is $C^{\infty}$ on $X_0$ with
$|t|<\frac{r}{2}$.  Then $\tilde{\Omega}(t)$ can be considered as a real analytic family of $(n-1,n-1)$-forms in $t$
 and it is smooth on $t$ by \cite[Proposition 2.2.3]{kp} again.

\section{Stability of $p$-K\"ahler
structures}\label{other-str}
This section is to prove a local stability theorem of $p$-K\"ahler
structures with deformation invariance of Bott--Chern numbers. We will first study
obstruction of extension for $d$-closed forms and then the un-obstruction of real extension for transverse form
via its two equivalent definitions.

Consider the differentiable family $\pi: \mathcal{X} \rightarrow
B$ of compact complex $n$-dimensional manifolds
over a sufficiently small domain in $\mathbb{R}^k$ with the
reference fiber $X_0:= \pi^{-1}(0)$ and the general fibers $X_t:=
\pi^{-1}(t).$ Here we fix a family of
hermitian metrics on $X_t$.
\subsection{Obstruction of extension for $d$-closed and $\p\db$-closed forms}
\label{d-closed}
Inspired by Wu's result \cite[Theorem 5.13]{w}, one has:
\begin{proposition}\label{d-ext}
Let $r$ and $s$ be non-negative integers. Assume that the reference
fiber $X_0$ satisfies the $\p\db$-lemma. Then any $d$-closed
$(r,s)$-form $\Omega_0$ and $\p_0\db_0$-closed $(r,s)$-form $\Psi_0$
on $X_0$ can be extended unobstructed to a $d$-closed $(r,s)$-form
$\Omega_t$ and a $\p_t\db_t$-closed $(r,s)$-form $\Psi_t$ on its
small differentiable deformation $X_t$, respectively.
\end{proposition}
\begin{proof}
By use of the extension map
$e^{\iota_{\varphi}|\iota_{\o{\varphi}}}$, we can construct two
$(r,s)$-forms $e^{\iota_{\varphi}|\iota_{\o{\varphi}}}(\Omega_0)$
and $e^{\iota_{\varphi}|\iota_{\o{\varphi}}}(\Psi_0)$ on $X_t$,
starting with $\Omega_0$ and $\Psi_0$, respectively.

Let $F_t$ be the orthogonal projection to $\mathds{F}_t$, the kernel
of
$$\square_{{BC},t}=\p_t\db_t\db_t^*\p_t^*+\db_t^*\p_t^*\p_t\db_t+\db_t^*\p_t\p_t^*\db_t+\p_t^*\db_t\db_t^*\p_t+\db_t^*\db_t+\p_t^*\p_t$$
and $\G_t$ denote the associated Green's operator with respect to a smooth family of Hermitian metrics on $X_t$. Then $F_t$ and
$\G_t$ are $C^{\infty}$ differentiable in $t$ since the $\dim
\mathds{F}_t$ is deformation invariant, thanks to \cite[Theorem
5.12]{w}. Therefore, the desired $d$-closed $(r,s)$-form is
\begin{equation}\label{4-omt}
\Omega_t =\big(\p_t\db_t\db_t^*\p_t^*\G_t+F_t\big)\big(e^{\iota_{\varphi}|\iota_{\o{\varphi}}}(\Omega_0)\big)
\end{equation}
as long as one notices that
$$\left.\Omega_t\right|_{t=0} =\big(\square_{{BC},0}\G_0+F_0\big)\Omega_0=\Omega_0$$
by the Hodge decomposition of the operator $\square_{{BC},0}$ and
the $d$-closedness of $\Omega_0$.

The construction of $\Psi_t$ is quite similar to the one of
$\Omega_t$ in \eqref{4-omt}. Denote by $\widetilde{F}_t$ the orthogonal
projection to $\widetilde{\mathds{F}}_t$, the kernel of
\[ \square_{A,t}= \p_t\db_t\db_t^*\p_t^*+\db_t^*\p_t^*\p_t\db_t
+\p_t\db_t^*\db_t\p_t^*+\db_t\p_t^*\p_t\db_t^*+\db_t\db_t^*+\p_t\p_t^*
\] and by $\widetilde{\mathbb{G}}_t$ the associated Green's
operator with respect to a smooth family of Hermitian metrics on $X_t$. By the same token, $\widetilde{F}_t$ and
$\widetilde{\mathbb{G}}_t$ are also $C^{\infty}$ differentiable in
$t$ since the $\dim \widetilde{\mathds{F}}_t$ is deformation
invariant. Then the construction of $\Psi_t$ goes as follows:
\[ \Psi_t =\Big(\big(\p_t\db_t\db_t^*\p_t^*+\p_t\db_t^*\db_t\p_t^*+
\db_t\p_t^*\p_t\db_t^*+\db_t\db_t^*+\p_t\p_t^*\big)\widetilde{\G}_t+\widetilde{F}_t\Big)
\Big(e^{\iota_{\varphi}|\iota_{\o{\varphi}}}(\Psi_0)\Big), \] where
it is easy to see that
\[ \left.\Psi_t\right|_{t=0}= \big( \square_{A,0}\widetilde{\mathbb{G}}_0 + \widetilde{F}_0 \big) \Psi_0 = \Psi_0. \]
\end{proof}

\begin{remark}\label{d-ext-rem}
It follows easily from the proposition above that any small deformation
of a pluriclosed manifold, satisfying the $\p\db$-lemma, is still
pluriclosed. Recall that a compact complex manifold is called
\emph{pluriclosed} if it admits a $\p\db$-closed positive $(1,1)$-form.
Moreover, it follows from the proof that the theorem still holds
when the $\p\db$-lemma assumption is replaced by the
infinitesimal  deformation invariance of $(r,s)$-Bott--Chern and
Aeppli numbers, respectively. These results are possibly known to experts.
\end{remark}

We can also prove this proposition by another way inspired by the
results of \cite{au,FY,w} in the $(n-1,n-1)$-case, i.e., any
$d$-closed $(n-1,n-1)$-form $\Omega$ on a complex manifold $X$
satisfying the $\p\db$-lemma can be extended unobstructed to a
$d$-closed $(n-1,n-1)$-form on the
small differentiable deformation $X_t$
of $X$.

Let $f_t: X_t\rightarrow X_0$ be a diffeomorphism depending on $t$
with $f_0= \emph{identity}$. And then one obtains a $d$-closed
$(2n-2)$-form $\Om_t=f_t^*\Om$ on $X_t$, which is decomposed as
$$\Om_t=\Om_t^{n-2,n}+\Om_t^{n-1,n-1}+\Om_t^{n,n-2}$$
with respect to the complex structure on $X_t$. It is easy to check
the following properties:

\begin{enumerate}[$(1)$]
    \item \label{appom}
$\Om_t^{n-1,n-1}$ approaches to $\Omega$ as $t\rightarrow 0$;

    \item \label{appo}
$\Om_t^{n-2,n}$ and $\Om_t^{n,n-2}$ approach to 0 as $t\rightarrow
0$.
 \end{enumerate}

Recall that \cite[Theomem 5.12]{w} or \cite[Corollary 3.7]{at} says
that if $X_0$ satisfies the $\p\db$-lemma, so does the general fiber
$X_t$. So we can choose an $(n-2,n-1)$-form $\Psi_1$ and an
$(n-1,n-2)$-form $\Psi_2$ on $X_t$ such that
$$\p_t\db_t \Psi_1=\db_t \Om_t^{n-1,n-1}=-\p_t \Om_t^{n-2,n},$$
$$-\p_t\db_t \Psi_2=\p_t \Om_t^{n-1,n-1}=-\db_t \Om_t^{n,n-2},$$
where $\Psi_1$ and $\Psi_2$ can be set as
$$\Psi_1,\Psi_2\bot_{\om_t} \ker(\p_t\db_t).$$
 Put
$$\widetilde{\Om}_t=\Om_t^{n-1,n-1}+ \p_t\Psi_1+\db_t
\Psi_2.$$ Obviously, $\widetilde{\Om}_t$ is a $d$-closed
$(n-1,n-1)$-form on $X_t$. Then Fu--Li--Yau proved the
following highly nontrivial estimates in \cite[Sections $4$ and
$5$]{FLY}: for some $0<\alpha<1$
$$\|\p_t\Psi_1\|_{C^0(\om_t)}\leq C\|\p_t \Om_t^{n-2,n}\|_{C^{0,\alpha}(\om_t)}$$
and similarly
$$\|\db_t
\Psi_2\|_{C^0(\om_t)}\leq C\|\db_t
\Om_t^{n,n-2}\|_{C^{0,\alpha}(\om_t)},$$ where $C$ is a uniform
constant. By use of these two estimates, one knows that
$\widetilde{\Om}_t$ is indeed $d$-closed extension of $\Om$ since
$\p_t \Om_t^{n-2,n}$ and $\db_t\Om_t^{n,n-2}$ approaches to zero
uniformly as $t\rightarrow 0$.

It is easy to see that Theorem \ref{0blc-inv} is impossible to
obtain by Fu--Yau's result since the proof would rely on the
deformation openness of $(n-1,n)$-th mild $\p\db$-lemma, which
contradicts with Ugarte--Villacampa's Example \ref{not-fy}.

\subsection{Un-obstruction of extension for transverse forms}\label{p-kahler}
In this subsection we study some basic properties and local stabilities of $p$-K\"ahler
structures, which seem more pertinent to the
nature of the stability problem of complex structures.

Let $V$ be a complex vector space of complex dimension $n$ with its
dual space $V^{*}$, namely the space of complex linear functionals
over $V$. Denote the complexified space of the exterior $m$-vectors
of $V^{*}$ by $\bigwedge^{m}_{\mathbb{C}} V^{*}$, which admits a
natural direct sum decomposition
\[ \bigwedge^{m}_{\mathbb{C}} V^{*} = \sum_{r+s=m} \bigwedge^{r,s} V^*, \]
where $\bigwedge^{r,s} V^*$ is the complex vector space
of $(r,s)$-forms on $V^*$.
The case $m=1$ exactly reads \[ \bigwedge^{1}_{\mathbb{C}} V^{*} =
V^* \bigoplus \overline{V^{*}},\] where the natural isomorphism $V^*
\cong \bigwedge^{1,0}V^*$ is used. Let $q\in \{1, \cdots, n\}$ and $p=n-q$. Obviously, the complex dimension
$N$ of $\bigwedge^{q,0}V^*$ equals to the combination number $C^{q}_n$. After a basis $\{
\beta_i \}_{i=1}^N$ of the complex vector space $\bigwedge^{q,0}V^*$ is fixed,
the canonical Pl\"ucker embedding as in \cite[Page 209]{GH} is given
by
$$\begin{array}{cccc}
\rho: & G(q,n) & \hookrightarrow & \mathbb{P}(\bigwedge^{q,0}V^*) \\
      & \Lambda & \mapsto & [\cdots,\Lambda_{i},\cdots]. \\
\end{array} $$
Here $G(q,n)$ denotes the Grassmannian of $q$-planes in the vector
space $V^*$ and $\mathbb{P}(\bigwedge^{q,0}V^*)$ is the
projectivization of $\bigwedge^{q,0}V^*$. A $q$-plane in $V^*$ can
be represented by a decomposable $(q,0)$-form $\Lambda \in
\bigwedge^{q,0}V^*$ up to a nonzero complex number, and
$\{\Lambda_i\}_{i=1}^N$ are exactly the coordinates of $\Lambda$
under the fixed basis $\{ \beta_i \}_{i=1}^N$. \emph{Decomposable
$(q,0)$-forms} are those forms in $ \bigwedge^{q,0}V^*$ that can be
expressed as $\gamma_1 \bigwedge \cdots \bigwedge \gamma_q$ with
$\gamma_i \in V^* \cong \bigwedge^{1,0}V^*$ for $1 \leq i \leq q$.
Set
\begin{equation}\label{knpq}
k=(N-1)-pq
\end{equation} to be the codimension of $\rho(G(q,n))$ in
$\mathbb{P}(\bigwedge^{q,0}V^*)$, whose locus characterizes the
decomposable $(q,0)$-forms in $\mathbb{P}(\bigwedge^{q,0}V^*)$.

Now we list several positivity notations and refer the readers to \cite{HK,H,Demailly} for more details. A $(q,q)$-form $\Theta$ in
$\bigwedge^{q,q}V^*$ is defined to be \emph{strictly positive (resp. positive)} if
\[ \Theta =\sigma_{q}\sum_{i,j=1}^N \Theta_{i\b j} \beta_i \wedge \b\beta_j,\]
where $\Theta_{ij}$ is a positive (resp. semi-positive) hermitian matrix of the size $N \times N$ with $N=C_{n}^q$ under the
basis $\{\beta_i \}_{i=1}^N$ of the complex vector space $\bigwedge^{q,0}V^*$ and $\sigma_{q}$ is defined to be the constant
$2^{-q}(\sqrt{-1})^{q^2}$.
According to this definition, the fundamental form of a hermitian metric on a complex manifold is actually a strictly positive $(1,1)$-form everywhere.
A $(p,p)$-form $\Gamma\in \bigwedge^{p,p}V^*$ is called
\emph{weakly positive}
if
the volume form $$\Gamma\wedge\sigma_{q}\tau\wedge\bar{\tau}$$ is
positive for every nonzero decomposable $(q,0)$-form $\tau$
of $V^*$;
a $(q,q)$-form $\Upsilon\in \bigwedge^{q,q}V^*$ is said to be \emph{strongly positive} if
$\Upsilon$ is a convex combination
$$\Upsilon=\sum\gamma_s \sqrt{-1}\alpha_{s,1}\wedge\bar\alpha_{s,1}\wedge\cdots\wedge\sqrt{-1}\alpha_{s,q}\wedge\bar\alpha_{s,q},$$
where $\alpha_{s,i}\in V^*$ and $\gamma_s\geq 0$.
As shown in \cite[Chapter III.\S\ 1.A]{Demailly},  the sets of weakly positive and strongly positive forms are closed convex cones,
 and by definition, the weakly positive cone is dual to the strongly positive cone via the pairing
$$\bigwedge^{p,p}V^*\times \bigwedge^{q,q}V^*\longrightarrow \mathbb{C};$$ all weakly positive forms
are real.
An element $\Xi$ in
$\bigwedge^{p,p}V^*$ is called \emph{transverse}, if
the volume form $$\Xi\wedge\sigma_{q}\tau\wedge\bar{\tau}$$ is
strictly positive for every nonzero decomposable $(q,0)$-form $\tau$
of $V^*$. There exist many various names for this
terminology and we refer to \cite[Appendix]{abb} for a list.

These positivity notations on complex vector spaces can be extended pointwise to
complex differential forms on a complex manifold.
Let $M$ be a complex manifold of dimension $n$. Then:

\begin{definition}[{\cite[Definition $1.11$]{aa}}, for example]
Let $p$ be an integer,
$1\leq p\leq n$. Then $M$ is called a \emph{$p$-K\"ahler manifold}
if there exists a $d$-closed transverse $(p,p)$-form on $M$.
\end{definition}

The duality between the weakly positive and strongly positive cones of forms is used to define
corresponding positivities for currents.
\begin{definition}[]\label{po-cur}
A current $T$ of bidegree $(q,q)$ on $M$ is strongly positive (resp. positive)
if the pairing $\langle T, u\rangle \geq 0$ for all weakly positive (resp. strongly positive) test forms
$u\in A^{p,p}(M)$ at each point. Clearly, each positive current is real.
\end{definition}

We are going to discuss several basics of transverse forms, such as the equivalent characterizations of
$1$-K\"ahlerness and $(n-1)$-K\"ahlerness by a slightly different approach from Alessandrini-Andreatta \cite[Proposition 1.15]{aa}.
Let $V$ be furnished with a Hermitian inner product and $V^*$ with
the dual inner product, which will extend to
$\bigwedge^{m}_{\mathbb{C}}V^*$.
Denote by $\bigwedge^{p,p}_{\mathbb{R}}V$ the (real) vector space of
real $(p,p)$-forms of $V^*$, which consists of invariant complex
$(p,p)$-forms of $V^*$ under conjugation. Then it is well known that, for
every $\Omega\in \bigwedge^{p,p}_{\mathbb{R}}V^*$, there exist real
numbers $\{\lambda_1,\cdots,\lambda_N\}$ and an orthogonal basis
$\{\eta_1,\cdots,\eta_N\}$ of $\bigwedge^{p,0}V^*$, satisfying
$|\eta_j|^2=2^p$ for $1 \leq j \leq N$, such that
$$\Omega=\sigma_p\sum_{j=1}^N\lambda_j\eta_j\wedge\bar{\eta}_j,$$
which is called the \emph{canonical form} for real $(p,p)$-forms (See
\cite{HK} for more details). The \emph{positive (resp. negative) index} of
$\Omega$ is the number of positive (resp. negative) ones in
$\{\lambda_1,\cdots,\lambda_N\}$.

\begin{proposition}
Let $\Omega$ be a transverse $(p,p)$-form of $V^*$. Then the positive index
of $\Omega$ is no less than $N-k$, where $k$ is given by
\eqref{knpq}.
\end{proposition}
\begin{proof}
It is clear that
$$\Omega=\sigma_p(\lambda_1\eta_1\wedge\bar\eta_1+\cdots+\lambda_N\eta_N\wedge\bar\eta_N)$$
for some real numbers $\{\lambda_1,\cdots,\lambda_N\}$ and some
orthogonal basis $\{\eta_1,\cdots,\eta_N\}$ of $\bigwedge^{p,0}V^*$,
satisfying that $|\eta_j|^2=2^p$ for $1 \leq j \leq N$.

Suppose that the positive index $\Omega\leq N-k-1$. Without loss of
generality, we may assume that $\{\lambda_{N-k},\cdots,\lambda_N\}$
are all non-positive real numbers.

Apparently, $\bigwedge^{p,0}V^*$ and $\bigwedge^{q,0}V^*$ are dual
vector spaces. And thus, the zero loci of $\eta_1,\cdots,\eta_{N-(k+1)}$
define $N-(k+1)$ hyperplanes in $\mathbb{P}(\wedge^{q,0}V^*)$, denoted by
$\eta_1^\perp,\cdots,\eta_{N-(k+1)}^\perp$, respectively.
Since the codimension of $\rho(G(q,n))$ is $k$ and the complex plane
$\eta_1^\perp\cap\cdots\cap\eta_{N-(k+1)}^\perp$ is of dimension
$k$, there exists some nonzero decomposable
$(q,0)$-form $\tau$, which lies in
$\eta_1^\perp\cap\cdots\cap\eta_{N-(k+1)}^\perp$. Then it follows
that
$$\Omega\wedge\sigma_{q}\tau\wedge\bar\tau\leq 0,$$
which contradicts with the definition \lq\lq transverse".
\end{proof}

\begin{corollary}[{\cite[Proposition 1.15]{aa}}]\label{}
A complex manifold $M$ is $1$-K\"ahler if and only if $M$ is
K\"ahler; $M$ is $(n-1)$-K\"ahler if and only if $M$ is balanced.
\end{corollary}
\begin{proof}
Both cases follow from the expression formula \eqref{knpq},
$$N-k=N-(N-1)+(n-1)\cdot 1=n. $$
\end{proof}

\begin{example}
There exists a transverse $(p,p)$-form $\Omega$ of $V^*$, whose positive
index is exactly $N-k$.
\end{example}
Actually, let $\{\eta_1,\cdots,\eta_N\}$ still denote the orthogonal
basis of $\bigwedge^{p,0}V^*$ as above. We shall construct an
element $\Omega$ in $\bigwedge^{p,p}_{\mathbb{R}}V^*$ below
\[ \Omega= \sigma_p \sum_{i=1}^{N-k} \lambda_i\eta_i\wedge\bar\eta_i, \]
with $\lambda_1,\cdots,\lambda_{N-k}>0$. It is clear that the
dimension of the complex plane
$\eta_1^\perp\cap\cdots\cap\eta_{N-k}^\perp$ is $k-1$. Hence, we can
slightly change the orthogonal basis $\{\eta_1,\cdots,\eta_N\}$ to
another one with the same feature, such that
$\eta_1^\perp\cap\cdots\cap\eta_{N-k}^\perp$ has no intersection
with $\rho(G(q,n))$. Let us check that $\Omega$ is transverse. It is
easy to see that
$$\Omega\wedge\sigma_{q}\tau\wedge\bar\tau\geq 0$$
for each nonzero decomposable $(q,0)$-form $\tau$. Now suppose that
$\Omega\wedge\sigma_{q}\tau\wedge\bar\tau=0$ for some decomposable
$(q,0)$-form $\tau\neq0$. Then it follows that
$$
    \begin{cases}
     \eta_1\wedge\tau=0,\\
     \qquad\vdots\\
     \eta_{N-k}\wedge\tau=0,\\
    \end{cases}
$$
which implies that
$\tau\in\eta_1^\perp\cap\cdots\cap\eta_{N-k}^\perp$. This
contradicts with the choice of the orthogonal basis
$\{\eta_1,\cdots,\eta_N\}$. Therefore, $\Omega$ is a transverse
$(p,p)$-form.

\begin{example}[{See also the example on \cite[Page 50]{HK}}]
There exists a transverse $(p,p)$-form $\Omega$ of $V^*$, whose negative
index is larger than $0$.
\end{example}
In fact, let us slightly modify the example above. Consider the
element in $\bigwedge^{p,p}_{\mathbb{R}}V^*$,
\begin{equation}\label{om-pp}
\Omega=\sigma_p(\lambda_1\eta_1\wedge\bar\eta_1+\cdots+\lambda_{N-k}\eta_{N-k}\wedge\bar\eta_{N-k}+\lambda_{N-k+1}\eta_{N-k+1}\wedge\bar\eta_{N-k+1}),
\end{equation}
where $\lambda_1,\cdots,\lambda_{N-k}>0$, $\lambda_{N-k+1}$ is some
negative number to be fixed later, and the complex plane
$\eta_1^\perp\cap\cdots\cap\eta_{N-k}^\perp$ has no intersection
with $\rho(G(q,n))$ in $\mathbb{P}(\wedge^{q,0}V^*)$. Construct a
function over $\rho(G(q,n))$, defined by
$$f([\tau])=\frac{\sigma_p(\lambda_1\eta_1\wedge\bar\eta_1+\cdots+\lambda_{N-k}\eta_{N-k}\wedge\bar\eta_{N-k})\wedge\sigma_{q}\tau\wedge\bar\tau}
{\sigma_p\eta_{N-k+1}\wedge\bar\eta_{N-k+1}\wedge\sigma_{q}\tau\wedge\bar\tau}.$$
It is easy to check that $f$ is well-defined on $\rho(G(q,n))$. The
function $f$ has positive values when $[\tau]\in
\rho(G(q,n))\setminus \eta^\perp_{N-k+1}$ and attains to $+\infty$
when $[\tau]\in \rho(G(q,n))\cap\eta^\perp_{N-k+1}$. Then $f$ can
obtain its minimum value over $\rho(G(q,n))$ by an elementary
analysis, which is denoted by $a>0$. Let $-a<\lambda_{N-k+1}<0$.
Then $\Omega$ constructed in \eqref{om-pp} is transverse.

In \cite{ab}, Alessandrini--Bassanelli proved that
$(n-1)$-K\"ahlerian property is not preserved under the small
deformations for balanced manifolds nor, more generally, for
$p$-K\"ahler manifolds  $(p > 1)$, while, based on the above argument
on the $p$-K\"ahler structures and Proposition \ref{d-ext},
we have the following local stability theorem of $p$-K\"ahlerian structures:
\begin{theorem}\label{c-p-kahler}
For any positive integer $p\leq n-1$, any
small differentiable deformation $X_t$ of a
$p$-K\"ahler manifold $X_0$ satisfying the $\p\db$-lemma is still
$p$-K\"ahlerian.
\end{theorem}

Alessandrini--Bassanelli \cite[Section 4]{aabb} constructed a
smooth proper modification $\tilde{X}$ of $\mathbb{CP}^5$, which
will be $p$-K\"ahler for $2 \leq p \leq 5$, but non-K\"ahler. It is
clear that the non-K\"ahler Moishezon $5$-fold $\tilde{X}$ is
$p$-K\"ahler for $p=2,3$ (the most interesting parts in this
theorem), satisfying the $\p \db$-lemma due to \cite{DGMS}, which
indicates that Theorem \ref{c-p-kahler} does not just concern
K\"ahler or balanced $\p \db$-structures. It is worth noticing  that
Alssandrini--Bassanelli conjectured on \cite[Page $299$]{aabb}
that a $p$-K\"ahlerian complex manifold is also $q$-K\"ahler for
$p\leq q\leq n$. Besides, Theorem \ref{c-p-kahler} might help to
produce examples of $p$-K\"ahler $\p \db$-manifolds, which are not
in the Fujiki class, since being in the Fujiki class is not an open
property under deformations thanks to \cite{Cam} and \cite{LP}.

We will present two proofs for this theorem and use the following trivial lemma of Calculus in both proofs.
\begin{lemma}\label{interval}
Let $f(z,t)$ be a real continuous function on $K \times \Delta_{\epsilon}$, where $K$ is a compact set and
$\Delta_{\epsilon}=\{t \in \mathbb{R}^k\ \big |\ |t|<\epsilon \}$. Assume that
\[ f(z,0)>0, \quad \text{for}\ z \in K.\]
Then there exists some positive number $\delta>0$, such that
\[ f(z,t)>0,\]
for $z \in K$ and $t\in \Delta_\delta$.
\end{lemma}

\begin{proof}
It is clear that for each $z \in K$, there exists an open neighborhood $U_z$ in $K$ and some $\delta_{z}>0$, such that
\[ f(z,t)>0, \]
for $z \in U_z$ and $t \in \Delta_{\delta_z}$. Compactness of $K$ enables us to find
a finite open covering of $K$, say $U_{z_1}, \cdots, U_{z_m}$. Then we may set $\delta$ to be
\[ \min\{ \delta_{z_1},\cdots,\delta_{z_m} \}. \]
Therefore, it follows that
\[ f(z,t) >0 ,\]
for $z \in K$ and $t\in \Delta_\delta$.
\end{proof}

Next, we proceed to the first proof of Theorem \ref{c-p-kahler}, which is based on an equivalent definition of transverse $(p,p)$-forms via strongly positive currents and
their extension property.

Let $\pi: \mathcal{X} \rightarrow B$ be a  differentiable family of compact complex manifolds
with the reference fiber $X_0:= \pi^{-1}(0)$ and the general fibers $X_t:=\pi^{-1}(t)$.
It is known from \cite[Lemma $2.8$]{RZ15} that the extension map in \eqref{map}
\[ e^{\iota_{\varphi(t)}|\iota_{\overline{\varphi(t)}}}:
 A^{n-p,n-q}(X_0)\> A^{n-p,n-q}(X_t) \]
is a linear isomorphism, depending smoothly on $t$, and its inverse map $e^{-\iota_{\varphi(t)}|-\iota_{\overline{\varphi(t)}}}$ is defined by \eqref{ephi-inv}.
Since the dual spaces of $A^{n-p,n-q}(X_0)$ and $A^{n-p,n-q}(X_t)$ are exactly the spaces of $(p,q)$-currents on $X_0$ and $X_t$,
set as $\mathfrak{D}'^{p,q}(X_0)$ and $\mathfrak{D}'^{p,q}(X_t)$, with the weak topologies respectively,
the adjoint map $\big( e^{-\iota_{\varphi(t)}|-\iota_{\overline{\varphi(t)}}} \big)^*$, given by
\begin{equation}\label{adj-map}
 \big( e^{-\iota_{\varphi(t)}|-\iota_{\overline{\varphi(t)}}} \big)^*:
\mathfrak{D}'^{p,q}(X_0) \> \mathfrak{D}'^{p,q}(X_t),
\end{equation}
is defined by the following formula:
\beq\label{ext-current-1}
\left\langle  \big( e^{-\iota_{\varphi(t)}|-\iota_{\overline{\varphi(t)}}} \big)^*T, \Omega_t  \right\rangle
=\left\langle  T, e^{-\iota_{\varphi(t)}|-\iota_{\overline{\varphi(t)}}} \big(\Omega_t\big)  \right\rangle,
\eeq
where $T$ is a $(p,q)$-current on $X_0$, $\Omega_t$ is an $(n-p,n-q)$-form on $X_t$ and the pairing $\langle \bullet,\bullet\rangle$
is the natural pairing between currents and smooth complex differential forms of pure type on $X_0$ or $X_t$.
It is clear that every $(n-p,n-q)$-form $\Omega_t$ on $X_t$ can be expressed as $e^{\iota_{\varphi(t)}|\iota_{\overline{\varphi(t)}}}\big(\Omega\big)$,
for some $(n-p,n-q)$-form $\Omega$ on $X_0$, due to the linear isomorphism $e^{\iota_{\varphi(t)}|\iota_{\overline{\varphi(t)}}}$.
Then the formula \eqref{ext-current-1} now reads:
\beq\label{ext-current-2}
\left\langle \big(e^{-\iota_{\varphi(t)}|-\iota_{\overline{\varphi(t)}}} \big)^*T, e^{\iota_{\varphi(t)}|\iota_{\overline{\varphi(t)}}} \big( \Omega \big) \right\rangle
=\left\langle  T, \Omega  \right\rangle,
\eeq
which can be regarded as the defining formula of the adjoint map \eqref{adj-map}.
It is easy to see that this adjoint map is a linear homeomorphism, depending smoothly on $t$. Using this map,
one obtains the following extension proposition, to be also of independent interest.
\begin{proposition}\label{p-eq}
The natural map $e^{\iota_{\varphi(t)}|\iota_{\overline{\varphi(t)}}}$ sends each kind of positive
$(p,p)$-forms (defined in this subsection) on $X_0$ bijectively onto the corresponding one on $X_t$;
the map $\big( e^{-\iota_{\varphi(t)}|-\iota_{\overline{\varphi(t)}}} \big)^*$ sends strongly positive and positive
$(p,p)$-currents homeomorphically onto the corresponding ones on $X_t$, respectively.
\end{proposition}

\begin{proof}
What we need to show is that the map $e^{\iota_{\varphi(t)}|\iota_{\overline{\varphi(t)}}}$ actually sends smooth forms or currents, satisfying
some positive condition, to the ones with the same positivity, since it can be similarly proved that
the inverse map of $e^{\iota_{\varphi(t)}|\iota_{\overline{\varphi(t)}}}$ shares the same property.

It is clear that the map $e^{\iota_{\varphi(t)}|\iota_{\overline{\varphi(t)}}}$ sends strongly positive $(p,p)$-forms $X_0$ to the strongly positive ones
on $X_t$, by the very definition of $e^{\iota_{\varphi(t)}|\iota_{\overline{\varphi(t)}}}$ and strong positivity. As to the positive case,
we may choose a positive $(p,p)$-form $\Omega$ on $X_0$, which can be locally written as
\[ \Omega = \sigma_p \sum_{|I|=|J|=p} \Omega_{I\b{J}} dz^{I} \wedge d\b{z}^{J},\]
where $\Omega_{I\b{J}}$ is a semi-positive hermitian matrix of the size $N \times N$ with $N=C_{n}^p$ everywhere and varies with respect to the local
coordinates $\{z^i\}_{i=1}^n$. The local expression of $e^{\iota_{\varphi(t)}|\iota_{\overline{\varphi(t)}}}\big(\Omega\big)$ amounts to
\[ \sigma_p \sum_{|I|=|J|=p} \Omega_{I\b{J}} e^{\iota_{\varphi(t)}}\big(dz^{I}\big)
\wedge \overline{e^{\iota_{\varphi(t)}} \big(dz^{J}\big)}, \]
where $$\left\{ e^{\iota_{\varphi(t)}}\big(dz^{I}\big)
\wedge \overline{e^{\iota_{\varphi(t)}} \big(dz^{J}\big)} \right\}_{|I|=|J|=p}$$ is a local basis
of $(p,p)$-forms on $X_t$ as shown in \cite[Lemma 2.4]{RZ15} for example. And thus, $e^{\iota_{\varphi(t)}|\iota_{\overline{\varphi(t)}}}\big(\Omega\big)$ is clearly a positive $(p,p)$-form on $X_t$.
To see the case of weak positivity, let $\Omega$ be a weakly positive $(p,p)$-form on $X_0$. We need to show that
$e^{\iota_{\varphi(t)}|\iota_{\overline{\varphi(t)}}} \big( \Omega \big)$ is weakly positive on $X_t$. Fix a point $w$ on $X_t$
and a strongly positive $(n-p,n-p)$-form \[\eta_t \in \bigwedge^{n-p,n-p}T^{*(1,0)}_{w, X_t}.\]
Let $\mathcal{X} \stackrel{(\rho,\pi)}{\simeq} X_0 \times B$ be the diffeomorphism for the  differentiable family $\pi:\mathcal{X}\> B$, which
induces the integrable Beltrami differential form $\varphi(t)$ (cf.\cite{C}). Then $\eta_t$ can be expressed as $e^{\iota_{\varphi(t)}|\iota_{\overline{\varphi(t)}}}\big(\eta\big)$
for some
\[ \eta \in \bigwedge^{n-p,n-p} T^{*(1,0)}_{\rho(w), X_0}. \]
Besides, the equality holds
\[ \left.e^{\iota_{\varphi(t)}|\iota_{\overline{\varphi(t)}}} \big( \Omega \big)\right|_{w} \wedge \eta_t
= \left.e^{\iota_{\varphi(t)}|\iota_{\overline{\varphi(t)}}} \big( \Omega \big)\right|_{w} \wedge e^{\iota_{\varphi(t)}|\iota_{\overline{\varphi(t)}}} \big( \eta \big)
= e^{\iota_{\varphi(t)}|\iota_{\overline{\varphi(t)}}} \big( \left.\Omega\right|_{\rho(w)} \wedge \eta \big).\]
From the very definition of weakly positive $(p,p)$-form $\Omega$ on $X_0$, it follows that \[\left.\Omega\right|_{\rho(w)} \wedge \eta\]
is a positive $(n,n)$-form at $\rho(w)$ on $X_0$ and thus
\[\left.e^{\iota_{\varphi(t)}|\iota_{\overline{\varphi(t)}}} \big( \Omega \big)\right|_{w} \wedge \eta_t\]
is also a positive $(n,n)$-form at $w$ on $X_t$, by the definition of $e^{\iota_{\varphi(t)}|\iota_{\overline{\varphi(t)}}}$. Therefore,
the $(p,p)$-form $e^{\iota_{\varphi(t)}|\iota_{\overline{\varphi(t)}}} \big( \Omega \big)$ is weakly positive on $X_t$, since $w$ and $\eta_t$
can be arbitrarily chosen.

The statements on the currents can be proved directly from the definition \ref{po-cur} for the positivities of currents, the formula \eqref{ext-current-2} and the results
of smooth forms shown above.
\end{proof}

To study local stabilities of transverse $(p,p)$-forms, we need an equivalent characterization of them, as in \cite[Claim on Page $5$]{a} or implicit in
\cite[The proofs of Lemma 1.22 and Theorem 1.17]{aa}: a $(p,p)$-form
$\Omega$ is transverse on an $n$-dimensional  compact complex manifold $M$
if and only if \[\int_M \Omega \wedge T>0,\]
for every nonzero strongly positive $(n-p,n-p)$-current $T$ on it.

\begin{proposition}\label{ext-trans}
Let $\pi: \mathcal{X} \> B$ be a  differentiable family of compact complex $n$-dimensional manifolds and
$\Omega_t$  a family of real $(p,p)$-forms on $X_t$, depending smoothly on $t$. Assume that
$\Omega_0$ is a transverse $(p,p)$-form on $X_0$. Then $\Omega_t$ is also transverse on $X_t$ for $t$ small.
\end{proposition}

\begin{proof}
Fix a hermitian metric $\omega_0$ on $X_0$ and define a real function $f(T,t)$ as follows:
\[ f(T,t) = \int_{X_t} \Omega_t \wedge \big(e^{-\iota_{\varphi(t)}|-\iota_{\overline{\varphi(t)}}} \big)^*T, \]
where $T$ varies in the space of strongly positive $(n-p,n-p)$-currents on $X_0$, satisfying
\[ \int_{X_0} \omega_0^p \wedge T =1, \]
and $t \in \Delta_{\epsilon}$. As pointed out in \cite[Propositon $1.7$]{aa} or \cite[Propositon I.5]{S}, the space of strongly positive $(n-p,n-p)$-currents $T$ on $X_0$, satisfying
$\int_{X_0} \omega_0^p \wedge T =1$, is a compact set, denoted by $K$ here. Then it is obvious that
the function $f(T,t)$ is continuous on $K \times \Delta_{\epsilon}$. The assumption that the real $(p,p)$-form $\Omega_0$ is transverse on $X_0$
implies that
\[ f(T,0)>0, \]
for each $T \in K$. Hence, Lemma \ref{interval} provides a positive number $\delta>0$, such that
\[ f(T,t)>0, \]
for $T \in K$ and $t\in \Delta_\delta$. This indeed shows that $\Omega_t$
is transverse for $t\in \Delta_\delta$, by Proposition \ref{p-eq} and the following application of
the formula \eqref{ext-current-2}:
\[ \int_{X_t} \left(e^{\iota_{\varphi(t)}|\iota_{\overline{\varphi(t)}}}\big(\omega_0\big)\right)^{p} \wedge \big(e^{-\iota_{\varphi(t)}|-\iota_{\overline{\varphi(t)}}} \big)^*T
= \int_{X_t} e^{\iota_{\varphi(t)}|\iota_{\overline{\varphi(t)}}}\big(\omega^p_0\big) \wedge \big(e^{-\iota_{\varphi(t)}|-\iota_{\overline{\varphi(t)}}} \big)^*T
= \int_{X_0} \omega_0^p \wedge T =1,\]
where $e^{\iota_{\varphi(t)}|\iota_{\overline{\varphi(t)}}}\big(\omega_0\big)$ can be regarded as a fixed hermitian metric on $X_t$.
\end{proof}

\begin{proof}[The first proof of Theorem \ref{c-p-kahler}]
Let $\Omega_0$ be a $d$-closed transverse $(p,p)$-form on the $p$-K\"ahler manifold $X_0$. Proposition \ref{d-ext} assures that
there exists a $d$-closed real extension $\Omega_t$ of $\Omega_0$ on $X_t$, depending smoothly on $t$.
Therefore, by Proposition \ref{ext-trans}, $\Omega_t$ is actually a $d$-closed transverse $(p,p)$-form
on $X_t$ for small $t$, which implies that $X_t$ is a $p$-K\"ahler manifold.
\end{proof}

Then one also has another proof:
\begin{proof}[{The second proof of Theorem \ref{c-p-kahler}}]
Let $\Omega_0$ be a $d$-closed transverse $(p,p)$-form on $X_0$ and $\Omega_t$ its $d$-closed real extension on $X_t$ constructed in \eqref{4-omt}
by Proposition \ref{d-ext}. To prove the theorem,  we just need: there exists a uniform small constant $\varepsilon>0$,
such that for any $t\in \Delta_\varepsilon$ and any nonzero
decomposable $(q,0)$-form $\tau$ at any given point $x\in X_0$  with $p+q=n$,
\begin{align}\label{11}
  \Omega_t(x)\wedge \sigma_q e^{\iota_{\varphi}}(\tau)\wedge e^{\iota_{\b{\varphi}}}(\b{\tau})>0,
\end{align}
where $\varphi:=\varphi(x,t)$ is induced by the
small differentiable deformation. Since
$e^{\iota_{\varphi}}$ induces an isomorphism between decomposable $(q,0)$-forms of $X_0$
and those of $X_t$, $\Omega_t$ will be the desired $d$-closed transverse $(p,p)$-form on $X_t$.

In fact, let $\omega$ be a Hermitian metric on $X_0$.
For any fixed point $x\in X_0$, we define a continuous function $f_{x}(t,[\tau])$
on $\Delta_\epsilon\times Y_{x}$ by
\begin{equation}\label{fxttau}
f_{x}(t,[\tau]):=\frac{ \Omega_t(x)\wedge \sigma_q e^{\iota_{\varphi}}(\tau)\wedge e^{\iota_{\b{\varphi}}}(\b{\tau})}{|\tau|^2_{\omega(x)}\cdot \omega(x)^n},
 \end{equation}
where $Y_x=\rho(G(q,n))|_{x}\subset \mb{P}(\bigwedge^{q,0}T^*_{X_0}|_x)$ is compact. Notice that
$\Omega_t\wedge \sigma_q e^{\iota_{\varphi}}(\tau)\wedge e^{\iota_{\b{\varphi}}}(\b{\tau})$ can be considered as an $(n,n)$-form on $X_0$.
Then by the transversality of $\Omega_0$,
$$f_x(0,[\tau])=\frac{ \Omega_0(x)\wedge \sigma_q \tau\wedge \b{\tau}}{|\tau|^2_{\omega(x)}\cdot \omega(x)^n}>0.$$
Thus, by the continuity of $f_{x}(t,[\tau])$ on $t$ and $[\tau]$,  Lemma \ref{interval} gives rise to a constant $\epsilon_x>0$ depending only on $x$, such that
$$\label{12}
  f(x,\overline{\Delta}_{\epsilon_x/2},Y_x):=f_{x}(\overline{\Delta}_{\epsilon_x/2},Y_x)>0.
$$
Let $\{U_j\}_{j=1}^J$ be trivializing covering of $\bigwedge^{q,0}T^*_{X_0}$, and choose any $x\in U_j$ for some $j$. Then one can identify $Y_x$ and $Y_y$ for any point $y\subset U_j$, and $f_y$ is defined on $Y_x$.
So by Lemma \ref{interval}, there exists an open neighbourhood $V_x\subset U_j$ of $x$ such that
$$
  f(V_x, \overline{\Delta}_{\epsilon_x/2},Y_x)>0
$$
since $f$ is continuous on $\overline{\Delta}_{\epsilon_x/2}\times Y_x$.

Noticing that $X_0$ is compact, one obtains a finite open covering $V_{x_i}$, $i=1,\cdots, m$, $X_0=\cup_{i=1}^m V_{x_i}$. Set
$$\varepsilon:=\min_{1\leq i\leq m}\epsilon_{x_i}/2>0.$$
Then
$$
  f(x,\overline{\Delta}_{\varepsilon},Y_x)=f_x(\overline{\Delta}_{\varepsilon},Y_x)>0
$$
for any $x\in X_0$.

Therefore,  by the definition \eqref{fxttau},
for any $|t|\leq \varepsilon$,
$$
  \Omega_t(x)\wedge \sigma_q e^{\iota_{\varphi(x)}}(\tau)\wedge e^{\iota_{\b{\varphi}(x)}}(\b{\tau})=f_{x}(t,[\tau])|\tau|^2_{\omega(x)}\cdot \omega(x)^n>0
$$
for any nonzero decomposable $(q,0)$-form $\tau$ at $x$. This is exactly the desired inequality \eqref{11}.
\end{proof}

\begin{remark} \label{p-k-rem}
Proposition \ref{ext-trans} and the second proof of Theorem \ref{c-p-kahler} actually show that any smooth real extension of a transverse $(p,p)$-form is still transverse, which also plays an important role in \cite{RWZ19}.
So the obstruction to extend a $d$-closed transverse $(p,p)$-form on a compact complex manifold
lies in the $d$-closedness. Hence, by Remark \ref{d-ext-rem}, the condition of the $\p\db$-lemma in Theorem \ref{c-p-kahler} can be replaced by the deformation invariance
of $(p,p)$-Bott--Chern numbers. Moreover, this condition may be weakened as some kind of the $\p\db$-lemma if the power series method works, just similarly to
what we have done in the balanced case in Section \ref{balanced}.
\end{remark}

Finally, we have to mention an interesting conjecture proposed by
Demailly--P$\breve{\textrm{a}}$un for the \lq global' stability of K\"ahler structures,
i.e., the K\"ahler property should be open for the countable Zariski
topology on the base.
\begin{conjecture}[{\cite[Cojecture 5.1]{DP}}] Let $X\longrightarrow S$ be a
deformation of compact complex manifolds over an irreducible base
$S$. Assume that one of the fibres $X_{t_0}$ is K\"ahler. Then there
exists a countable union $S'\subsetneq S$ of analytic subsets in the
base such that $X_t$ is K\"ahler for $t\in S\setminus S'$. Moreover,
$S$ can be chosen so that the K\"ahler cone is invariant over
$S\setminus S'$, under parallel transport by the Gauss-Manin
connection.
\end{conjecture}

\appendix
\section{Proof of Proposition \ref{7for}}\label{7proof}
One chooses a holomorphic coordinate chart $(z^1,\cdots,z^n)$ on the
complex manifold $X$ throughout this proof.

\begin{enumerate}
\item \quad Set $\phi = \phi^s_{\bar{t}} \dz^t \frac{\p}{\p z^s}$ and
$\alpha=\frac{1}{p!q!} \alpha_{I \bar{J}} dz^{I} \w \dz^{J}$
locally. One calculates:

 $\begin{aligned} & \phi \lc
\overline{\phi} \lc \alpha -
(\phi \lc \overline{\phi}) \lc \alpha \\
=& \phi \lc \lk \sum_{1 \leq u \leq q} \frac{1}{p!q!}
\overline{\phi}^{\bar{j_u}}_t \alpha_{I\bar{J}} dz^{I} \w \cdots \w
\big(dz^t\big)_{j_u} \w \cdots \rk
- \lk \phi^t_{\bar{l}} \overline{\phi}^{\bar{s}}_t \dz^l \frac{\p}{\pz^s} \rk \lc \alpha \\
=& \phi \lc \lk \sum_{1 \leq u \leq q} \frac{1}{p!q!}
\overline{\phi}^{\bar{j_u}}_t \alpha_{I\bar{J}} dz^{I} \w \cdots
\w \big(dz^t\big)_{j_u} \w \cdots \rk \\
 & - \sum_{1 \leq u \leq q} \frac{1}{p!q!} \phi^t_{\bar{l}}
\overline{\phi}^{\bar{j_u}}_t \alpha_{I\bar{J}} dz^{I} \w \cdots
\w \big(\dz^l\big)_{j_u} \w \cdots  \\
=& \sum_{\begin{subarray}{c} 1 \leq v \leq p \\ 1 \leq u \leq
q\end{subarray}} \frac{1}{p!q!} \phi^{i_v}_{\bar{l}}
\overline{\phi}^{\bar{j_u}}_t \alpha_{I\bar{J}} \cdots \w
\big(\dz^l\big)_{i_{v}} \w \cdots dz^{i_p} \w \dz^{j_1} \cdots \w
\big(dz^t\big)_{j_u} \w \cdots,
\end{aligned}$

while

 $\begin{aligned} & \overline{\phi} \lc \phi \lc
\alpha - (\overline{\phi}
\lc \phi) \lc \alpha \\
=& \overline{\phi} \lc \lk \sum_{1 \leq v \leq p} \frac{1}{p!q!}
\phi^{i_v}_{\bar{l}} \alpha_{I\bar{J}} \cdots \w
\big(\dz^l\big)_{i_v} \w \cdots \w \dz^{J} \rk
- \lk \overline{\phi}^{\bar{l}}_{t} \phi^{s}_{\bar{l}} dz^t \frac{\p}{\p z^s} \rk \lc \alpha \\
=& \overline{\phi} \lc \lk \sum_{1 \leq v \leq p} \frac{1}{p!q!}
\phi^{i_v}_{\bar{l}} \alpha_{I\bar{J}} \cdots
\big(\dz^l\big)_{i_v} \w \cdots \w \dz^{J} \rk \\
 & - \sum_{1 \leq v \leq p} \frac{1}{p!q!}
\overline{\phi}^{\bar{l}}_{t} \phi^{i_v}_{\bar{l}}
\alpha_{I\bar{J}} \cdots \w \big(dz^t\big)_{i_v} \w \cdots  \dz^{J} \\
=& \sum_{\begin{subarray}{c} 1 \leq v \leq p \\ 1 \leq u \leq
q\end{subarray}} \frac{1}{p!q!} \phi^{i_v}_{\bar{l}}
\overline{\phi}^{\bar{j_u}}_t \alpha_{I\bar{J}} \cdots \w
\big(\dz^l\big)_{i_v} \w \cdots dz^{i_p} \w \dz^{j_1} \cdots \w
\big(dz^t\big)_{j_u} \w \cdots.
\end{aligned}$

\item \quad This follows directly from the generalized commutator formula \eqref{f1} and $$\p(\phi
\lc \phi \lc \overline{\phi}) = 0.$$

\item \quad This is proved in \cite[Appendix B]{C}.

\item \quad Set $\phi = \phi^s_{\bar{t}} \dz^t \frac{\p}{\p z^s}$ and
$\alpha=\frac{1}{p!q!} \alpha_{I \bar{J}} dz^{I} \w \dz^{J}$
locally.  Thus,

$\begin{aligned} & \overline{\phi} \lc \overline{\phi} \lc \phi \lc
\alpha -
\phi \lc \overline{\phi} \lc \overline{\phi} \lc \alpha \\
=& \overline{\phi} \lc \overline{\phi} \lc \lk \sum_{1 \leq v \leq
p} \frac{1}{p!q!}\phi^{i_v}_{\bar{t}} \alpha_{I\bar{J}}
\cdots \w \big(\dz^t\big)_{i_v} \w \cdots \w \dz^{J} \rk\\
&- 2\phi \lc \lk \sum_{1 \leq u < l \leq q} \frac{1}{p!q!}
\overline{\phi}^{\bar{j_u}}_i \overline{\phi}^{\bar{j_l}}_j
\alpha_{I\bar{J}} dz^{I} \w \cdots \w \big(dz^i\big)_{j_u} \w
\cdots \w \big(dz^j\big)_{j_l} \w \cdots \rk \\
=& 2 \sum_{\begin{subarray}{c} 1 \leq v \leq p \\ 1 \leq u \leq q
\end{subarray}} \frac{1}{p!q!} \overline{\phi}^{\bar{j_u}}_i
\overline{\phi}^{\bar{t}}_{j} \phi^{i_v}_{\bar{t}} \alpha_{I\bar{J}}
\cdots \w \big(dz^j\big)_{i_v} \w \cdots
dz^{i_p} \w \dz^{j_1} \cdots \w \big(dz^{i}\big)_{j_u} \w \cdots \\
& +2 \sum_{\begin{subarray}{c} 1 \leq v \leq p\\ 1 \leq u < l \leq q
\end{subarray}} \frac{1}{p!q!} \overline{\phi}^{\bar{j_u}}_i
\overline{\phi}^{\bar{j_l}}_j \phi^{i_v}_{\bar{t}} \alpha_{I\bar{J}}
\cdots \w \big(\dz^t\big)_{i_v} \w \cdots dz^{i_p} \w \dz^{j_1}
\cdots \w \big(dz^i\big)_{j_u} \w \cdots
\w \big(dz^j\big)_{j_i} \w \cdots \\
\end{aligned}$

$\begin{aligned}
& -2 \sum_{\begin{subarray}{c} 1 \leq v \leq p\\ 1 \leq u < l \leq q
\end{subarray}} \frac{1}{p!q!} \overline{\phi}^{\bar{j_u}}_i
\overline{\phi}^{\bar{j_l}}_j \phi^{i_v}_{\bar{t}} \alpha_{I\bar{J}}
\cdots \w \big(\dz^t\big)_{i_v} \w \cdots dz^{i_p} \w \dz^{j_1}
\cdots \w \big(dz^i\big)_{j_u} \w \cdots
\w \big(dz^j\big)_{j_l} \w \cdots \\
& -2 \left( \sum_{1 \leq u < l \leq q}
\frac{1}{p!q!} \phi^{i}_{\bar{t}} \overline{\phi}^{\bar{j_u}}_i
\overline{\phi}^{\bar{j_l}}_j \alpha_{I\bar{J}} dz^{I} \w \cdots
\w \big(\dz^t\big)_{j_u} \w \cdots \w \big(dz^j\big)_{j_l} \w \cdots \right.\\
& \quad \quad \quad + \left. \sum_{1 \leq u < l \leq q}
\frac{1}{p!q!} \overline{\phi}^{\bar{j_u}}_i \phi^{j}_{\bar{t}}
\overline{\phi}^{\bar{j_l}}_j \alpha_{I\bar{J}} dz^{I} \w \cdots \w
\big( dz^i \big)_{j_u} \w
\cdots \w \big( \dz^t \big)_{j_l} \w \cdots \right) \\
=& 2 ( \overline{\phi} \lc \phi \overline{\phi} \lc \alpha
- \overline{\phi} \phi \lc \overline{\phi} \lc \alpha ). \\
\end{aligned}$
\end{enumerate}

\section{Proof of Observation \ref{closed}} \label{app-b}
We will omit the sub-index in many places without
danger of confusion. By use of \eqref{7.5} and \eqref{7.3} in
Proposition \ref{7for}, the integrability condition \eqref{int} and
the commutator formula \eqref{f1} repeatedly, one has
\begin{align*}
&\quad \db(\varphi\lc\omega)\\ &= \frac{1}{2} [\varphi,\varphi] \lc
\omega + \varphi \lc \db \omega\\
&= - \frac{1}{2} \varphi \lc \varphi \lc \p \omega + \varphi \lc \p
( \varphi \lc \omega ) + \varphi \lc \big( \db(\bar\varphi \varphi
\lc \omega - \varphi \lc \bar \varphi \lc
\omega) - \p (\varphi \lc \omega) \big) \\
&= - \frac{1}{2} \varphi \lc \varphi \lc \big( \p (\varphi \bar
\varphi \lc \omega - \bar \varphi \lc \varphi \lc \omega ) -\db
(\bar \varphi \lc \omega)\big) + \varphi \lc \db (\bar\varphi
\varphi \lc \omega - \varphi \lc \bar \varphi \lc \omega ) \\
&= - \frac{1}{2} \varphi \lc \varphi \lc \big( \p (\varphi \bar
\varphi \lc \omega - \bar \varphi \lc \varphi \lc \omega ) -
\db (\bar \varphi \lc \omega) \big) \\
& \quad + \db (\varphi \lc \bar\varphi \varphi \lc \omega) -
\frac{1}{2} [\varphi,\varphi] \lc \bar \varphi \varphi \lc \omega -
\varphi \lc \frac{1}{2} [\varphi,\varphi] \lc \bar \varphi \lc
\omega - \varphi \lc \varphi \lc \db (\bar \varphi \lc \omega ).
\end{align*}
Then \eqref{7.1} and \eqref{7.7} in Proposition \ref{7for} yield
that
\begin{align*}
&\quad \db(\varphi\lc\omega)\\
&= - \frac{1}{2} \varphi \lc \varphi\lc \big( \p (\varphi \lc \bar
\varphi \lc \omega) + \db ( \bar \varphi \lc \omega ) \big) +
\frac{1}{2} \db (\varphi \lc \varphi \lc \bar \varphi \lc \omega) +
\frac{1}{2} \varphi \lc \varphi \lc \varphi \lc \p (\bar \varphi \lc
\omega ) \\
&= - \frac{1}{2} \varphi \lc \varphi\lc \big( \p (\varphi \lc \bar
\varphi \lc \omega) + \db ( \bar \varphi \lc \omega ) \big) +
\frac{1}{2} \varphi \lc [\varphi, \varphi] \lc \bar \varphi \lc
\omega + \frac{1}{2} \varphi \lc \varphi \lc \db (\bar \varphi \lc
\omega) + \frac{1}{2}\varphi \lc \varphi \lc \varphi \lc \p (\bar \varphi \lc \omega) \\
&= -
\frac{1}{2}\big(\varphi\lc\p(\varphi\lc\varphi\lc\bar\varphi\lc\omega)
-\varphi\lc\varphi\lc\p(\varphi\lc\bar\varphi\lc\omega)\big).
\end{align*}
Note that the equality $\iota_{\varphi} \circ
\iota_{[\varphi,\varphi]} = \iota_{[\varphi,\varphi]} \circ
\iota_{\varphi}$ holds (cf. \cite[Page 361]{C}). Then the last
equality but one above is also equal to \begin{align*} & -
\frac{1}{2} \varphi \lc \varphi \lc\big( \p (\varphi \lc \bar
\varphi \lc \omega) + \db ( \bar \varphi \lc \omega ) \big) +
\frac{1}{2}[\varphi, \varphi] \lc \varphi \lc \bar \varphi \lc
\omega + \frac{1}{2} \varphi \lc \varphi \lc \db (\bar \varphi \lc
\omega)
+ \frac{1}{2} \varphi \lc \varphi \lc \varphi \lc \p (\bar \varphi \lc \omega)\\
=& \frac{1}{2} \Big(
\varphi\lc\varphi\lc\varphi\lc\p(\bar\varphi\lc\omega)
+2\big(\varphi\lc\p(\varphi\lc\varphi\lc\bar\varphi\lc\omega)
-\varphi\lc\varphi\lc\p(\varphi\lc\bar\varphi\lc\omega)\big)\Big).
\end{align*}
Hence, we get the equality
\begin{align*}
&- \frac{1}{2}
\big(\varphi\lc\p(\varphi\lc\varphi\lc\bar\varphi\lc\omega)
-\varphi\lc\varphi\lc\p(\varphi\lc\bar\varphi\lc\omega) \big)\\
=&\frac{1}{2} \Big(
\varphi\lc\varphi\lc\varphi\lc\p(\bar\varphi\lc\omega)
+2\big(\varphi\lc\p(\varphi\lc\varphi\lc\bar\varphi\lc\omega)-\varphi\lc\varphi\lc\p(\varphi\lc\bar\varphi\lc\omega)\big)
\Big),
\end{align*}
which implies that \begin{align*} \db (\varphi \lc \omega) &=
-\frac{1}{2}
\big(\varphi\lc\p(\varphi\lc\varphi\lc\bar\varphi\lc\omega)
- \varphi\lc\varphi\lc\p(\varphi\lc\bar\varphi\lc\omega)\big) \\
& = \frac{1}{6}  \varphi\lc\varphi\lc\varphi\lc\p
(\bar\varphi\lc\omega).
\end{align*}

Therefore, $\db(\varphi\lc \omega)_{N} =0$ reduces to the equality
$$\label{closed-check} \db (\varphi \lc \omega)_k =0, \quad
\textrm{for}\ k=1,2,3,$$ which are easily to be checked by the above
formulation. This concludes the proof of Observation \ref{closed}.

\end{document}